\documentclass[12pt]{amsart}
\title{A GR\"{O}BNER BASIS FOR SCHUBERT PATCH IDEALS}
\date{\today}
\author{Emmanuel Neye}
\address{Department of Mathematics, University of Saskatchewan, Saskatoon, CANADA}
\email{\href{mailto: eon836@mail.usask.ca}{eon836@mail.usask.ca}}

\usepackage{pgf,tikz,pgfplots}
\pgfplotsset{compat=1.15}
\usepackage{mathrsfs}
\usepackage{amssymb}
\usetikzlibrary{arrows,patterns}
\usepackage{fancyvrb}
\usepackage{url}

\numberwithin{equation}{section}

\usepackage{makecell}

\usepackage{enumerate}

\makeatletter
\newcommand*\bigcdot{\mathpalette\bigcdot@{.5}}
\newcommand*\bigcdot@[2]{\mathbin{\vcenter{\hbox{\scalebox{#2}{$\m@th#1\bullet$}}}}}
\makeatother

\usepackage[top=1in, left=1in, right=1in, bottom=1in]{geometry}
\usepackage{amsfonts, amssymb, latexsym, graphicx, hyperref, amsmath,enumitem,amsthm}
\usepackage{color}
\usepackage[vcentermath]{youngtab}
\usepackage{ytableau}
\usepackage{float}
\usepackage{subfig}
\ytableausetup{centertableaux}
\usepackage{algorithm,algorithmic}
\usetikzlibrary{calc, shapes, backgrounds,arrows,positioning,decorations.pathmorphing}

\newtheorem*{mthm*}{Main Theorem}
\newtheorem*{thm*}{Theorem}
\newtheorem{thm}{Theorem}[section]
\newtheorem{lem}[thm]{Lemma}
\newtheorem*{conjecture*}{Conjecture}

\newtheorem{cor}[thm]{Corollary}
\newtheorem{prop}[thm]{Proposition}

\theoremstyle{definition}
\newtheorem{defn}[thm]{Definition}
\newtheorem{ex}[thm]{Example}

\newtheorem{defnnota}[thm]{Definition/Notation}
\newtheorem{remark}[thm]{Remark}

\newcommand{\rank}{\mathrm{rank}\,}

\newcommand{\cellsize}{19}
\newlength{\cellsz} 
\newsavebox{\cell}
\sbox{\cell}{\begin{picture}(\cellsize,\cellsize)
	\put(0,0){\line(1,0){\cellsize}}
	\put(0,0){\line(0,1){\cellsize}}
	\put(\cellsize,0){\line(0,1){\cellsize}}
	\put(0,\cellsize){\line(1,0){\cellsize}}
	\end{picture}}
\newcommand\cellify[1]{\def\thearg{#1}\def\nothing{}%
	\ifx\thearg\nothing
	\vrule width0pt height\cellsz depth0pt\else
	\hbox to 0pt{\usebox{\cell} \hss}\fi%
	\vbox to \cellsz{
		\vss
		\hbox to \cellsz{\hss$#1$\hss}
		\vss}}
\newcommand\tableau[1]{\vtop{\let\\\cr
		\baselineskip -16000pt \lineskiplimit 16000pt \lineskip 0pt
		\ialign{&\cellify{##}\cr#1\crcr}}}

\begin{document}
	
	\maketitle

	\begin{abstract}
		Schubert patch ideals are a class of generalized determinantal ideals.
		They are prime defining ideals of open patches of Schubert varieties in the type $A$ flag variety.
		In this paper, we adapt the linkage-theoretic approach of E. Gorla, J. Migliore, and U. Nagel to prove a conjecture of A. Yong, namely, that the essential minors of every Schubert patch ideal form a Gr\"{o}bner basis.
		Using the same approach, we recover the result of A. Woo and A. Yong that the essential minors of a Kazhdan-Lusztig ideal form a Gr\"{o}bner basis.
		With respect to the standard grading of assigning degree 1 to each variable, we also show that homogeneous Schubert patch ideals and homogeneous Kazhdan-Lusztig ideals (and hence, Schubert determinantal ideals) are glicci.
	\end{abstract}

	\setlength{\parindent}{15pt}
	\section{Introduction}\label{sec:intro}
	Let $G = GL_n(\mathbb{K})$ denote the general linear group of invertible $n \times n$ matrices with entries in a field $\mathbb{K}$.
	Let $B_{-}$ (resp. $B_{+}$) be the subgroup of lower (resp. upper) triangular matrices in $G$.
	We work with the {complete flag variety} $B_{-}\backslash G$.
	Given a variety $X \subseteq B_{-}\backslash G$, set $\mathcal{M}_{v,X} := X \cap \big(\Omega_{v_0}^\circ v_0v\big)$, where $\Omega_v^\circ:= B_{-} \backslash B_{-}vB_{-}$ is the opposite Schubert cell associated to $v$, and $v_0 \in S_n$ is the long word permutation $v_0(i) = n-i+1$.
	A. Knutson \cite{knutson2008schubert} called the variety $\mathcal{M}_{v,X}$ an $\boldsymbol{X}$ \textbf{patch}  and E. Insko and A. Yong \cite{insko2012patch} called the defining ideal $Q_{v,X}$ of this $X$ patch a \textbf{patch ideal}. 
	
	In this paper, we are specifically interested in when the varieties $X \subseteq B_{-}\backslash G$ are the Schubert varieties $X_w$, $w \in S_n$ (see \cite[Example 2.4]{insko2012patch}). These varieties $X_{w}$, $w \in S_n$, are the Zariski closures of the Schubert cells $X_w^\circ:= B_{-} \backslash B_{-}wB_{+}$.
	In this case, we call the $\mathcal{M}_{v,X} = \mathcal{M}_{v,w}$ \textbf{Schubert patch varieties}, and their defining ideals are called \textbf{Schubert patch ideals}.

	Kazhdan-Lusztig ideals, introduced in \cite{woo2008governing}, are related to Schubert patch ideals; they are prime defining ideals of the varieties $\mathcal{N}_{v,w}:= X_w \cap \Omega_v^{\circ}$, $v,w \in S_n$. 	
	To be explicit, given an arbitrary Schubert patch ideal, on setting some specific variables of this ideal to zero, we obtain a Kazhdan-Lusztig ideal.
	
	Below is a conjecture due to A. Yong which was publicized through talks and conversations.
	
	\begin{conjecture*}[A. \hspace{-.20em}Yong]
	The\hspace{-.05em} es\hspace{-.05em}sential\hspace{-.05em} mino\hspace{-.05em}rs\hspace{-.05em} of\hspace{-.05em} S\hspace{-.05em}c\hspace{-.05em}hubert\hspace{-.05em} patch\hspace{-.05em} ideals\hspace{-.05em} form\hspace{-.05em} a\hspace{-.05em} G\hspace{-.05em}r\"{o}bner\hspace{-.05em} basis.
	\end{conjecture*}
	
	The main result of this paper is:  
	\begin{mthm*}\label{prb:main}
		A. Yong's conjecture is true.
	\end{mthm*}

	This result is stated precisely and proved as Theorem \ref{lem:gb-patch-multigrading2}. This is done via the linkage-theoretic approach of Gorla, Migliore, and Nagel \cite{gorla2013grobner}, though in a multigraded setting.
		This approach involves a key lemma that gives a sufficient condition for a set of polynomials to form a Gr\"{o}bner basis for the ideal it generates.
		This key lemma appears as Lemma \ref{suffCondGbasis2} in this paper. 
	The standard graded version of this lemma is due to	
	Gorla, Migliore and Nagel \cite[Lemma 1.12]{gorla2013grobner}, where it was used to show that the generators of some families of generalized determinantal ideals form Gr\"{o}bner bases.
	It was also used in \cite[Lemma 3.1]{fieldsteel2019gr} to show that the natural generators of the double determinantal ideals form Gr\"{o}bner bases.
	As shown in \cite[Lemma 5.2]{woo2008governing}, Kazhdan-Lusztig ideals are homogeneous with respect to a certain positive multigrading.
	By the same argument, Schubert patch ideals are also homogeneous with respect to the same positive multigrading.
	
	The proof of the Main Theorem above can also be used to give an alternative proof to the main result in \cite[Theorem 2.1]{woo2012grobner}, namely, that the essential minors generating a Kazhdan-Lusztig ideal form a Gr\"{o}bner basis.
	Since Schubert determinantal ideals are special cases of Kazhdan-Lusztig ideals, we therefore get an alternative proof to the well-known fact that the essential minors of every Schubert determinantal ideal form a Gr\"{o}bner basis (\cite[Theorem B]{knutson2005grobner}).
	Schubert determinantal ideals are defining ideals for matrix Schubert varieties -- these varieties are sets of  matrices that satisfy certain conditions on ranks of their submatrices.	

	The question ``Is every arithmetically Cohen-Macaulay subvariety of projective space in the \textbf{G}orenstein \textbf{l}iaison \textbf{c}lass of a \textbf{c}omplete \textbf{i}ntersection (shortened glicci)?'' is one of the open problems in liaison theory.
	Many families of ideals generated by minors of generic matrices have been shown to be glicci.
	Examples include the standard determinantal ideals \cite[Chapter 3]{kleppe2001gorenstein}, symmetric mixed ladder determinantal ideals \cite{gorla2010symmetric} and mixed ladder determinantal ideals \cite{gorla2007mixed}.
	With a little addition to our proof of the Main Theorem above, we also show that:
	\begin{thm*}
		Schubert patch ideals that are homogeneous with respect to the standard grading are glicci.
	\end{thm*}
	This result is stated precisely and proved as Theorem \ref{thm:vex-glicci}.
	We can infer from the proof of Theorem \ref{thm:vex-glicci} that Kazhdan-Lusztig ideals that are homogeneous under the standard grading are also glicci.
	In particular, Schubert determinantal ideals are glicci.
	Though it is already known that Schubert determinantal ideals are glicci \cite{klein2020geometric}, the existing proof relies on combinatorial results of Knutson and Miller \cite{knutson2005grobner}.

Not all Schubert patch ideals and Kazhdan-Lusztig ideals are homogeneous with respect to the standard grading in their respective underlying ring.
Characterizing the pairs $(v,w) \in S_n \times S_n$ for which the Kazhdan-Lusztig ideal $I_{v,w}$ is homogeneous with respect to the standard grading is an open problem in \cite[Problem 5.5]{woo2008governing}.
We obtain some pattern avoidance results on when Kazhdan-Lusztig ideals and Schubert patch ideals are homogeneous (see Propositions \ref{conj:homo-kaz} and \ref{conj:homo-kaz2}).

\subsection*{Structure of the paper} In Section \ref{sec:backgd}, we give background on Schubert determinantal ideals, Kazhdan-Lusztig ideals and Schubert patch ideals.
In Section \ref{sec:multigraded-set}, we discuss some basics of Gorenstein liaison.
We then present a lemma that is key to our work (see Lemma \ref{suffCondGbasis2}).
This lemma is presented in a multigraded setting and it gives a sufficient condition for a set of polynomials to form a Gr\"{o}bner basis for the ideal it generates.
The Main Theorem is proven in Section \ref{sec:grobner-s-d-i-k-l-i} (see Theorem \ref{lem:gb-patch-multigrading2}).
In Section \ref{sec:gbiliaison-s-d-i}, under the standard grading, we show that the homogeneous Schubert patch ideals are glicci.
Finally, in Section \ref{sec:hom-KL}, we discuss the problem of when Schubert patch ideals and Kazhdan-Lusztig ideals are homogeneous.

	\vspace{0.25cm}
	\noindent \textbf{Acknowledgments.} 
	Thanks to my supervisor, Jenna Rajchgot, for all her advice and
	guidance while working on this project.
	The original version of this paper was on recovering the known Gr\"{o}bner basis result for Kazhdan-Lusztig ideals \cite{woo2012grobner}.
	Thanks to Alexander Yong for telling us about his conjecture and for suggesting to try adapt our existing proofs on Kazhdan-Lusztig ideals to this setting.

	\section{Combinatorial and Schubert Variety Background}\label{sec:backgd}
	In this section, we recall some important definitions and results that we will need in this paper.
	Let $\mathbb{K}$ be an algebraically closed field.

	\subsection{Permutations and Partial Permutations}
	The primary references for this subsection are \cite{knutson2005grobner} and \cite{miller2004combinatorial}.
	Let $M_{mn}$ denote the space of $m \times n$ matrices over  $\mathbb{K}$, $\mathbf{X}$ be a matrix of variables $\mathbf{x} = \{ x_{\alpha\beta}\}_{1 \leq \alpha \leq m,\,1\leq  \beta \leq n}$
	and $\mathbb{K}[\mathbf{x}] = \mathbb{K}[x_{\alpha\beta}\,|\,1 \leq \alpha \leq m,\,1\leq  \beta \leq n]$ denote the coordinate ring of $M_{mn}.$
	A \textbf{partial permutation matrix} is a  matrix having all entries equal to 0 except for at most one entry equal to 1 in each row and column. 
	The \textbf{Rothe diagram} $D(w)$ of a partial permutation matrix $w \in M_{mn}$ consists of all locations (called ``boxes") in the $m \times n$ grid neither due south nor due east of a nonzero entry in $w$. 
	Alternatively, the set $D(w)$ can be described as follows. Place a dot $\Large \bigcdot$ in position $(w(j),j)$ for $1 \leq j \leq n$. For each dot, draw the line that extends to the right and below that dot. The boxes that are not in touch with any line are the boxes of $D(w)$ (See Example \ref{ex:diag-comp}).
	The \textbf{length} of $w$ is the cardinality $\ell(w)$ of its diagram $D(w)$.
	The \textbf{essential set} $\mathcal{E}ss(w)$ consists of the boxes $(p,q)$ in $D(w)$ such that neither $(p,q+1)$ nor $(p+1,q)$ lies in $D(w)$.
	
	Permutation matrices in this paper are obtained by permuting rows of the identity matrix.
	In other words, if $S_n$ is the symmetric group on $\{1,2,\ldots,n\}$ and $w = (w_{ij})$ is an $n \times n$ permutation matrix for a permutation $w \in S_n$, then for each $j$, $w_{ij} = 1$ if $i = w(j)$ and 0 otherwise. 
	For instance, the permutation $w = 3421 \in S_4$ has the permutation matrix
	\vspace{-0.15cm}
	\[\small w = \begin{pmatrix}
		0 & 0 & 0 & 1\\
		0 & 0 & 1 & 0\\
		1 & 0 & 0 & 0\\
		0 & 1 & 0 & 0
	\end{pmatrix}. \vspace{-0.15cm}\]
	The following are some standard terminologies about permutations.
	Let $v$ be a permutation in $S_n$. An integer $i$, $1\leq i < n$, is an \textbf{ascent} of $v$ if $v(i) < v(i+1)$ and a \textbf{descent} of $v$ if $v(i) > v(i+1)$.
	The \textbf{adjacent transposition} $s_i$, $1 \leq i < n$, is the 2-cycle $(i,i+1)$ which swaps $i$ and $i+1$. Each element $w \in S_n$ can be written as product of simple (adjacent) transpositions $s_{i_1}s_{i_2}\cdots s_{i_k}$.
	Among all these expressions $s_{i_1}s_{i_2}\cdots s_{i_k}$ for $w$, an expression for which $k$ is minimal is called a \textbf{reduced word} for $w$, and such $k$ is the length of $w$.
	A \textbf{subword} of a word $s_{i_1}s_{i_2}\cdots s_{i_k}$ is a word of the form $s_{j_1}s_{j_2}\cdots s_{j_k'}$, where $\{j_1,j_2,\ldots,j_{k'}\}$ is an ordered subset of $\{i_1,i_2,\ldots,i_k\}$. 
	The (strong) \textbf{Bruhat order} is defined by $u \leq w$ if every (or some) reduced word for $w$ has a subword that is a reduced word for $u$.
	
	Below is a straightforward fact which is closely related to \cite[Lemma 6.5]{woo2012grobner}.  
	
	\begin{lem}\label{lem6.5}
		Let $w \in S_n$ and $b$ be a descent of $w$.
		Then the placement of boxes of $D(w)$ and $D(ws_b)$ agree in all columns except columns $b$ and $b + 1$. Moreover, to obtain $D(ws_b)$ from $D(w)$, move all the boxes of $D(w)$ in column $b$ strictly below row $w(b + 1)$ one box to the right, and delete the box (that must appear) in position $(w(b + 1), b)$ of $D(w)$. The positions of the boxes of $D(w)$ in column $b$ strictly above row $w(b + 1)$ remain unchanged in $D(ws_b)$. 
	\end{lem}

\begin{ex}\label{ex:diag-comp}
	Let $w = 2476315$ with a descent at $b = 4$.
	To illustrate Lemma \ref{lem6.5}, we have:
	\[	D(w) =  \vcenter{\hbox{
			\begin{tikzpicture}[scale=.4]
				\draw (0,0) rectangle (7,7);
				
				\draw (0,6) --(0,7) -- (5,7) -- (5,6) -- (0,6);
				\draw (1,6) -- (1,7); \draw (2,6) -- (2,7); \draw (3,6) -- (3,7); \draw (4,6) -- (4,7);
				
				\draw (1,4) --(1,5) -- (4,5) -- (4,4) -- (1,4);
				\draw (2,4) -- (2,5); \draw (3,4) -- (3,5);
				
				\draw (2,1) -- (2,3) -- (4,3) -- (4,2) -- (3,2) -- (3,1) -- (2,1);
				\draw (2,2) -- (3,2) -- (3,3);

				\filldraw (0.5,5.5) circle (.5ex); \draw[line width = .2ex] (0.5,0) -- (0.5,5.5) -- (7,5.5);
				\filldraw (1.5,3.5) circle (.5ex); \draw[line width = .2ex] (1.5,0) --(1.5,3.5) --(7,3.5);
				\filldraw (2.5,0.5) circle (.5ex); \draw[line width = .2ex] (2.5,0) --(2.5,0.5) --(7,0.5);
				\filldraw (3.5,1.5) circle (.5ex); \draw[line width = .2ex] (3.5,0) --(3.5,1.5) --(7,1.5);
				\filldraw (4.5,4.5) circle (.5ex); \draw[line width = .2ex] (4.5,0) --(4.5,4.5) --(7,4.5);
				\filldraw (5.5,6.5) circle (.5ex); \draw[line width = .2ex] (5.5,0) --(5.5,6.5) --(7,6.5);
				\filldraw (6.5,2.5) circle (.5ex); \draw[line width = .2ex] (6.5,0) --(6.5,2.5) --(7,2.5);
	\end{tikzpicture}}} 
	\quad
	\text{and}
	\quad 
	D(ws_4) =  \vcenter{\hbox{
			\begin{tikzpicture}[scale=.4]
				\draw (0,0) rectangle (7,7);
				
				\draw (0,6) --(0,7) -- (5,7) -- (5,6) -- (0,6);
				\draw (1,6) -- (1,7); \draw (2,6) -- (2,7); \draw (3,6) -- (3,7); \draw (4,6) -- (4,7);
				
				\draw (1,4) --(1,5) -- (3,5) -- (3,4) -- (1,4);
				\draw (2,4) -- (2,5);
				
				\draw (2,1) -- (2,3) -- (3,3) -- (3,1) -- (2,1);
				\draw (2,2) -- (3,2);

				\draw (4,2) rectangle (5,3);

				
				\filldraw (0.5,5.5) circle (.5ex); \draw[line width = .2ex] (0.5,0) -- (0.5,5.5) -- (7,5.5);
				\filldraw (1.5,3.5) circle (.5ex); \draw[line width = .2ex] (1.5,0) --(1.5,3.5) --(7,3.5);
				\filldraw (2.5,0.5) circle (.5ex); \draw[line width = .2ex] (2.5,0) --(2.5,0.5) --(7,0.5);
				\filldraw (3.5,4.5) circle (.5ex); \draw[line width = .2ex] (3.5,0) --(3.5,4.5) --(7,4.5);
				\filldraw (4.5,1.5) circle (.5ex); \draw[line width = .2ex] (4.5,0) --(4.5,1.5) --(7,1.5);
				\filldraw (5.5,6.5) circle (.5ex); \draw[line width = .2ex] (5.5,0) --(5.5,6.5) --(7,6.5);
				\filldraw (6.5,2.5) circle (.5ex); \draw[line width = .2ex] (6.5,0) --(6.5,2.5) --(7,2.5);
	\end{tikzpicture}}} 
	\vspace{-0.5cm}.\]
	\qed
\end{ex}
	
	\subsection{Schubert Determinantal Ideals}
	The primary reference for this subsection is \cite{miller2004combinatorial}.
	Schubert determinantal ideals are defined from partial permutation matrices.
	Given a partial permutation $w \in M_{mn}$,
	the \textbf{matrix Schubert variety} $\overline{X}_w \subseteq M_{mn}$ is the subvariety
	\[\overline{X}_w := \{\mathbf{X} \in M_{mn}\,\,|\,\, \rank(\mathbf{X}_{p \times q}) \leq  \rank(w_{p \times q}) \,\, \text{for all $p$ and $q$}\},\]
	where $\mathbf{X}_{p \times q}$ (resp. $w_{p \times q}$) denotes the upper left $p \times q$  submatrix of $\mathbf{X}$ (resp. $w$). 
	The associated \textbf{Schubert determinantal ideal} $I_w \subseteq \mathbb{K}[\mathbf{x}]$ is the ideal 
	\[I_w := \langle \text{minors of size $1 + r_{pq}(w)$ in $\mathbf{X}_{p \times q}$} \,|\,(p,q) \in \{1,\ldots,m\}\times\{1,\ldots,n\} \rangle,\]
	where $r_{pq}(w)$, for each $p,q$, is the rank of the submatrix $w_{p \times q}$ of $w$.
	Fulton \cite{fulton1992flags} showed that the Schubert determinantal ideal $I_w$ is the prime defining ideal of $\overline{X}_w$, and that the \emph{essential minors} of size $1 + r_{pq}(w)$ in $\mathbf{X}_{p \times q}$, for all $(p,q) \in \mathcal{E}ss(w)$, generate $I_w$.
	\begin{ex}
	Consider the partial permutation matrix
	\(\small w =\begin{pmatrix} 
		1 & 0 & 0  \\ 
		0  & 0 & 0  \\ 
		0& 0 & 1   \\ 
	\end{pmatrix} \in M_{33}.\)
	The matrix Schubert variety $\overline{X}_w$ is the set of $3 \times 3$ matrices whose minors of size 2 in the upper left $2 \times 3$ and $3 \times 2$  submatrices vanish.
	In other words,
	\[\overline{X}_w =  \left\{\mathbf{X} \in M_{33}\,\,\big|\,\,
	\begin{vmatrix}
		x_{11} & x_{12}\\
		x_{21} & x_{22}
	\end{vmatrix} = \begin{vmatrix}
		x_{11} & x_{13}\\
		x_{21} & x_{23}
	\end{vmatrix} = \begin{vmatrix}
		x_{12} & x_{13}\\
		x_{22} & x_{23}
	\end{vmatrix} = \begin{vmatrix}
		x_{11} & x_{12}\\
		x_{31} & x_{32}
	\end{vmatrix} = \begin{vmatrix}
		x_{21} & x_{22}\\
		x_{31} & x_{32}
	\end{vmatrix} = 0
	\right\}\]
	and 
	\[I_w = \left\langle \begin{vmatrix} x_{11} & x_{12}\\
		x_{21} & x_{22}
	\end{vmatrix},\, \begin{vmatrix}
		x_{11} & x_{13}\\
		x_{21} & x_{23}
	\end{vmatrix},\, \begin{vmatrix}
		x_{12} & x_{13}\\
		x_{22} & x_{23}
	\end{vmatrix},\, \begin{vmatrix}
		x_{11} & x_{12}\\
		x_{31} & x_{32}
	\end{vmatrix},\, \begin{vmatrix}
		x_{21} & x_{22}\\
		x_{31} & x_{32}
	\end{vmatrix} \right\rangle.\]
	From the diagram \(D(w) = \vcenter{\hbox{
			\begin{tikzpicture}[scale=.4]
				\draw (0,0) rectangle (3,3);
				
				\draw (1,0) rectangle (2,1);
				\draw (1,1) rectangle (2,2);
				\draw (2,1) rectangle (3,2);

				\filldraw (0.5,2.5) circle (.5ex);
				\draw[line width = .2ex] (0.5,0) -- (0.5,2.5) -- (3,2.5);
				\filldraw (2.5,0.5) circle (.5ex);
				\draw[line width = .2ex] (2.5,0) -- (2.5,0.5) -- (3,0.5);
	\end{tikzpicture}}},\,\)
	we have that $\ell(w) = 3$ and $\mathcal{E}ss(w) = \{(2,3), (3,2)\}$.
		\hfill \qedsymbol
	\end{ex}
	
	Given $v, w \in S_n$, $v \leq w$ in (strong) Bruhat order if and only if $\overline{X}_w \subseteq \overline{X}_v$.
	One can check that $v \leq w$ in Bruhat order if and only if $w$ lies in $\overline{X}_v$ (see for example \cite[Lemma 15.19]{miller2004combinatorial}). In other words, $v \leq w$ if and only if 
	\(\rank(v_{p \times q}) \geq \rank(w_{p \times q})\)
	for all $p,q$.

	\subsection{Kazhdan-Lusztig Ideals}\label{subsec:kaz-lus}
	Woo and Yong introduced Kazhdan-Lusztig ideals in \cite{woo2008governing} and studied their Gr\"{o}bner geometry in \cite{woo2012grobner}. We recall some of the basics here, noting that our conventions differ from theirs.
	Fix an arbitrary permutation $v \in S_n$ and define a specialized generic matrix $\mathbf{X}^{(v)}$ of size $n \times n$ as follows: 
	for all $j$, set
	\(\mathbf{X}^{(v)}_{v(j),j} = 1,\)
	and, for all $j$, set
	\(\mathbf{X}^{(v)}_{v(j),b} = 0\) for $b > j$, and $\mathbf{X}^{(v)}_{a,j} = 0$ for $a > v(j)$.
	For all other coordinates $(i,j)$, set
	\(\mathbf{X}^{(v)}_{i,j} = x_{ij}.\)
	Let $\mathbf{x}^{(v)} \subseteq \mathbf{x}$ denote the variables appearing in $\mathbf{X}^{(v)}$.
	
	\begin{defnnota}\label{defnnota:xlast}
		Fix the lexicographic term order on $\mathbb{K}[\mathbf{x}^{(v)}]$ which is induced from the following lexicographic order on $\mathbb{K}[\mathbf{x}]$: \(x_{1n} \succ x_{2n}\succ \cdots \succ x_{nn} \succ x_{1,n-1}\succ x_{2,n-1} \succ \cdots \succ x_{n,n-1}\succ\cdots\succ x_{11} \succ x_{21} \succ \cdots \succ x_{n1}.\) 
		We denote by $x_{\text{last}}$ the variable in $\mathbf{X}^{(v)}$ that is maximal with respect to the term order $\succ$. This variable is the rightmost, then uppermost variable appearing in $\mathbf{X}^{(v)}$.
	\end{defnnota}
	
	\begin{ex}\label{ex:coordinateK}
		If $n = 5$ and $v = 34512 \in S_n$, then
	\[\small \mathbf{X} = 
	\begin{pmatrix}
		x_{11} & x_{12} & x_{13} & x_{14} & x_{15}\\
		x_{21} & x_{22} & x_{23} & x_{24} & x_{25}\\
		x_{31} & x_{32} & x_{33} & x_{34} & x_{35}\\
		x_{41} & x_{42} & x_{43} & x_{44} & x_{45}\\
		x_{51} & x_{52} & x_{53} & x_{54} & x_{55}		
	\end{pmatrix}
	\quad \text{and} \quad
	\mathbf{{X}}^{(v)} = 
	\begin{pmatrix}
		{ x_{11}} & { x_{12}} & { x_{13}} & 1 & {0}\\
		{ x_{21}} & {x_{22}} & {x_{23}} & 0 & 1\\
		1 & {0} & {0} & 0 & 0\\
		0 & 1 & {0} & 0 & 0\\
		0 & 0 & 1 & 0 & 0
	\end{pmatrix}
	.\]
	Here, $\mathbf{x}^{(v)} = \{x_{11}, x_{12}, x_{13}, x_{21}, x_{22}, x_{23}\} \subseteq \mathbf{x}$ and $x_{\text{last}} = x_{13}$.	
	\qed	
	\end{ex}
	
	\begin{defn}
	Let $v,w$ be arbitrary permutations in $S_n$. 
	Let $Y^{(v)}$ be the subset of all matrices $M_{nn}$ of the same shape as $\mathbf{X}^{(v)}$ but with entries in $\mathbb{K}$ instead of variables.
	Then, 
	\[\mathcal{N}_{v,w} := \{Z \in Y^{(v)}\,\,|\,\, \text{rank}(Z_{p \times q}) \leq  \text{rank}(w_{p \times q}), \,\, \text{$1 \leq p,q \leq n$}\}\]
	is called a \textbf{Kazhdan-Lusztig variety}.
\end{defn}

The Kazhdan-Lusztig variety $\mathcal{N}_{v,w}$ is isomorphic to the intersection of the Schubert variety $X_w=\overline{B_{-} \backslash B_{-}wB_{+}}$ with the opposite Schubert cell $\Omega_v^\circ = B_{-} \backslash B_{-}vB_{-}$.
	
	\begin{defn}
		Let $v,w$ be arbitrary permutations in $S_n$. The \textbf{Kazhdan-Lusztig ideal} $I_{v,w} \subseteq \mathbb{K}[\mathbf{x}^{(v)}]$ is the defining ideal for $\mathcal{N}_{v,w}$.
		Precisely, 
		\[{I}_{v,w} = \langle \text{minors of size $1+ \rank(w_{p \times q})$ in $\mathbf{X}^{(v)}_{p \times q}$,\,\, $1 \leq p,q \leq n$} \rangle.\]
	\end{defn}
	
	As stated in \cite{woo2012grobner}, one consequence of \cite[Lemma 3.10]{fulton1992flags} is that for any $w \in S_n$, the minors of size $1 + \rank(w_{p \times q})$ in $\mathbf{X}^{(v)}_{p \times q}$, for all $(p,q) \in \mathcal{E}ss(w)$, are sufficient to generate $I_{v,w}$.
	
	\begin{ex}\label{ex:2431K}
		Let $v = 34512$ (as in Example \ref{ex:coordinateK}) and $w = 24513$ be permutations in $S_5$.
		The Rothe diagram $D(w)$ of $w$ is given below:
		\[D(w) = \vcenter
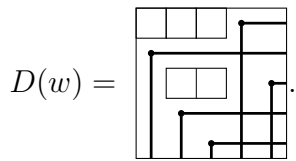
{\hbox{
				\begin{tikzpicture}[scale=.4]
					\draw (0,0) rectangle (5,5);
					
					\draw (0,4) rectangle (1,5);
					\draw (1,4) rectangle (2,5);
					\draw (2,4) rectangle (3,5);
					\draw (1,2) rectangle (3,3);
					\draw[line width = .05ex] (2,2) --(2,3);
					
					\filldraw (0.5,3.5) circle (.5ex);
					\draw[line width = .2ex] (0.5,0) -- (0.5,3.5) -- (5,3.5);
					\filldraw (1.5,1.5) circle (.5ex);
					\draw[line width = .2ex] (1.5,0) -- (1.5,1.5) -- (5,1.5);
					\filldraw (2.5,0.5) circle (.5ex);
					\draw[line width = .2ex] (2.5,0) -- (2.5,0.5) -- (5,0.5);
					\filldraw (3.5,4.5) circle (.5ex);
					\draw[line width = .2ex] (3.5,0) -- (3.5,4.5) -- (5,4.5);
					\filldraw (4.5,2.5) circle (.5ex);
					\draw[line width = .2ex] (4.5,0) -- (4.5,2.5) -- (5,2.5);
		\end{tikzpicture}}}.\]
		From this diagram, $\ell(w) = 5$, $\mathcal{E}ss(w) = \{(1,3), (3,3)\}$ and
		\(I_{v,w} = \left\langle x_{11}, \,x_{12}, \, x_{13}, \, x_{22},\,x_{23}  \right\rangle.\)
		\hfill \qedsymbol		
	\end{ex}
	
	Every Schubert determinantal ideal is a Kazhdan-Lusztig ideal.
	To see this, let $v$ be a permutation in $S_{2n}$ such that 
	\(v(i) = i + n \,\, \text{and} \,\, v(i+n) = i \,\, \text{for $1 \leq i \leq n$},\)
	then the matrix $\mathbf{X}^{(v)}$ is of the form 
	\[\mathbf{X}^{(v)} = 
	\left(\begin{array}{c|c}
		(x_{ij}) & I_n  \\\hline
		I_n & 0 
	\end{array}\right),
	\]
	which only involves the variables $x_{ij}$ for $1 \leq i,j \leq n$.
	Given $w \in S_n$, let $w \times 1_n \in S_{2n}$ be the standard embedding into $S_{2n}$, where
	\((w \times 1_n)(i) = w(i) \,\, \text{and} \,\, (w \times 1_n)(i+n) = i+n \,\, \text{for $1 \leq i \leq n$}.\)
	The Kazhdan-Lusztig ideal $I_{v,w\times 1_n}$ is equal to the Schubert determinantal ideal $I_w$. 
	Further explanation can be found in \cite{woo2012grobner}.
	
	\subsection{Schubert patch ideals}\label{subsec:patch}	
	Fix an arbitrary permutation $v \in S_n$ and define a specialized generic matrix $\mathbf{Z}^{(v)}$ of size $n \times n$ as follows: 
	for all $j$, set
	\(\mathbf{Z}^{(v)}_{v(j),j} = 1,\)
	and, for all $j$, set
	$\mathbf{Z}^{(v)}_{a,j} = 0$ for 
	$a > v(j)$.
	For all other coordinates $(i,j)$, set
	\(\mathbf{Z}^{(v)}_{i,j} = z_{ij}.\)
	Let $\mathbf{z}^{(v)} \subseteq \mathbf{z}$ denote the remaining variables appearing in $\mathbf{Z}^{(v)}$.
	If we define the set
	$$\mathbf{y}^{(v)} := \{z_{v(j),b} \in \mathbf{z}^{(v)}\,\,|\,\, b > j, \,\, 1 \leq j \leq n\},$$ 
	then the resulting set from substituting $x_{ij}$ for $z_{ij}$ in $\mathbf{z}^{(v)} \setminus \mathbf{y}^{(v)}$ coincides with the set $\mathbf{x}^{(v)}$ defined for Kazhdan-Lusztig ideals at the beginning of Subsection \ref{subsec:kaz-lus}.
	Let $\tilde{\mathbf{x}}^{(v)} := \mathbf{z}^{(v)} \setminus \mathbf{y}^{(v)}$.
	
	\begin{defnnota}\label{defnnotazmax}
		Fix the lexicographic term order on $\mathbb{K}[\tilde{\mathbf{x}}^{(v)}]$ and $\mathbb{K}[\mathbf{y}^{(v)}]$ which is induced from the following lexicographic order on $\mathbb{K}[\mathbf{z}]$: \(z_{ij} > z_{i'j'}\) if either $j > j'$, or $j = j'$ and $i < i'$.
	The term order $\succ$ on $\mathbb{K}[\mathbf{z}^{(v)}]$ is defined as: 
	\(\tilde{\mathbf{x}}^{(v)} \succ \mathbf{y}^{(v)},\) i.e., the variables in $\tilde{\mathbf{x}}^{(v)}$ are ordered above the variables in $\mathbf{y}^{(v)}$.
	Since $\mathbf{y}^{(v)} = \emptyset$ for Kazhdan-Lusztig ideals, it follows that the term order defined here coincides, up to setting the variables $\mathbf{y}^{(v)}$ to zero, with the term order in Definition \ref{defnnota:xlast}.
	We denote by $z_{\text{max}}$ the variable in $\mathbf{z}^{(v)}$ that is maximal with respect to the term order $\succ$. 
	This variable $z_{\text{max}}$ is the same as the variable $x_{\text{last}}$ in Notation \ref{defnnota:xlast} after the relabelling $z_{ij} \mapsto x_{ij}$.
	\end{defnnota}
	
	\begin{ex}\label{ex:coordinate}
		Revisiting Example \ref{ex:coordinateK}, we have the following here:
		\[ 
		\mathbf{Z}^{(v)} = 
		\begin{pmatrix}
			{\color{blue} z_{11}} & {\color{blue} z_{12}} & {\color{blue} z_{13}} & 1 & {\color{red} z_{15}}\\
			{\color{blue} z_{21}} & {\color{blue} z_{22}} & {\color{blue} z_{23}} & 0 & 1\\
			1 & {\color{red} z_{32}} & {\color{red} z_{33}} & 0 & 0\\
			0 & 1 & {\color{red} z_{43}} & 0 & 0\\
			0 & 0 & 1 & 0 & 0
		\end{pmatrix}
		.\]
		$\mathbf{z}^{(v)} = \{z_{11}, z_{12}, z_{13}, z_{15}, z_{21}, z_{22}, z_{23}, z_{32}, z_{33}, z_{43}\} \subseteq \mathbf{z}$, 
		$\tilde{\mathbf{x}}^{(v)} = \{z_{11}, z_{12}, z_{13}, z_{21}, z_{22}, z_{23}\}$ (as indicated in blue) and $\mathbf{y}^{(v)} = \{z_{15}, z_{32}, z_{33}, z_{43}\}$ (as indicated in red). 
		Furthermore, we have
		\(z_{13} \succ z_{23} \succ z_{12} \succ z_{22} \succ z_{11} \succ z_{21} \succ z_{15} \succ z_{33} \succ z_{43} \succ z_{32}\)
		and $z_{\text{max}} = z_{13}$.
		\hfill \qedsymbol
	\end{ex}
	
	\begin{lem}\label{lem:loczmax}
		Let $v$ be an arbitrary permutation in $S_n$ and $b$ be the last descent of $v$.
		The variable $z_{\text{max}}$ is located at the intersection of row $v(b+1)$ and column $b$ of the matrix $\mathbf{Z}^{(v)}$.
		\begin{proof}
		The variable $z_{\text{max}}$ is the rightmost, then uppermost variable in the resulting matrix from setting to zero the variables $\mathbf{y}^{(v)}$ in the matrix $\mathbf{Z}^{(v)}$.
		So it must be the variable immediately to the left of the 1 in column $b+1$ of $\mathbf{Z}^{(v)}$, since $b$ is the last descent of $v$.
		This 1 on column $b+1$ of $\mathbf{Z}^{(v)}$ is located at row $v(b+1)$, and so $z_{\text{max}}$ is located at position $(v(b+1),b)$.
		\end{proof}
	\end{lem}
	
	\begin{defn}\label{defn:patch ideal}
		Let $v, w \in S_n$. The \textbf{Schubert patch ideal} ${Q}_{v,w} \subseteq \mathbb{K}[\mathbf{z}^{(v)}]$ is defined as follows:
		\[{Q}_{v,w} = \langle \text{minors of size $1+ \rank(w_{p \times q})$ in $\mathbf{Z}^{(v)}_{p \times q}$,\,\, $1 \leq p,q \leq n$} \rangle,\]
		where $\mathbf{Z}^{(v)}_{p \times q}$ is the $p \times q$ northwest  submatrix of $\mathbf{Z}^{(v)}$.
	\end{defn}
	
	It follows from \cite[Lemma 3.10]{fulton1992flags} that for any permutation $w \in S_n$, the (essential) minors of size $1 + \rank(w_{p \times q})$ in $\mathbf{Z}^{(v)}_{p \times q}$, $(p,q) \in \mathcal{E}ss(w)$, are sufficient to generate $Q_{v,w}$.
	Let 
	\[\mathcal{M}_{v,w}:= \text{Spec}\big( \mathbb{K}[\mathbf{z}^{(v)}]/Q_{v,w} \big)\]
	be the associated affine scheme.

	\begin{prop}\label{prop:patch-variety-cell}
	For permutations $v,w \in S_n$, the Schubert patch scheme $\mathcal{M}_{v,w}$ is isomorphic to the intersection of the Schubert variety $X_w=B_{-} \backslash \overline{B_{-}wB_{+}}$ with the permuted opposite big cell $\Omega_{v_0}^\circ v_0v$, i.e., 
	$\mathcal{M}_{v,w} \cong X_w \cap \big(\Omega_{v_0}^\circ v_0v\big),$
	where $\Omega_v^{\circ} = B_{-} \backslash {B_{-}vB_{-}}$ is the opposite Schubert cell associated to $v$, and $v_0 \in S_n$ is the long word permutation $v_0(i) = n-i+1$.
	\begin{proof}
		We proceed in the same way as the proof of \cite[Proposition 3.1]{woo2008governing}.
		Let $\pi \,:\, G \rightarrow B_{-}\backslash G$ be the natural quotient map and consider the map 
		\(\sigma\,:\, \Omega_{v_0}^\circ v_0 v \rightarrow G,\)
		where $\sigma(F_{\boldsymbol{\bigcdot}})$ is the unique matrix representative of $F_{\boldsymbol{\bigcdot}}$, which, for all $j$, has 1s at $(v(j),j)$, and for all $j$, has 0s at $(a,j)$ for $a > v(j)$.
		It follows that
		\(X_w \cap \big(\Omega_{v_0}^\circ v_0v\big) \cong \pi^{-1}\big(X_w\big) \cap \sigma \big(\Omega_{v_0}^\circ v_0v\big)\)
		since $\sigma$ is a local section of $\pi$.
		One coordinate ring for $\text{GL}_n(\mathbb{K})$ is $\mathbb{K}[\mathbf{z},\det^{-1}(\mathbf{Z})]$, where $\mathbf{z}:= \{z_{ij}\,\, |\,\,1 \leq i,j \leq n\}$, are the entries of a generic matrix $\mathbf{Z}$.
		With these coordinates, the defining ideal for $\sigma\big(\Omega_{v_0}^\circ v_0v\big)$, denoted by $J_v$, is generated by the polynomials \({z}_{v(j),j} - 1,\)
		and monomials of the form ${z}_{a,j}$ for $a > v(j)$.
		Fulton \cite[Lemma 3.10]{fulton1992flags} showed that the defining ideal $I_w$ for $\pi^{-1}(X_w)$ is generated by minors of size $1 + \rank(w_{p \times q})$ in $\mathbf{Z}^{(v)}_{p \times q}$, for all $(p,q) \in \mathcal{E}ss(w)$.
		Therefore,
		\(X_w \cap \big(\Omega_{v_0}^\circ v_0v\big) \cong \text{Spec}\big( \mathbb{K}[\mathbf{z}]/\big(I_{w} + J_v\big) \big) 
		,\)
		noting that $\det(\mathbf{Z}) = \pm 1$ by $J_v$.
		To reduce the number of variables, instead of working in a generic matrix $\mathbf{Z}$, we first quotient by $J_v$ and then work in the generic matrix $\mathbf{Z}^{(v)}$.
		The image of $I_w$ in $\mathbb{K}[\mathbf{z}^{(v)}]$ under this quotient by $J_v$ is precisely $Q_{v,w}$, and hence, $\text{Spec}\big( \mathbb{K}[\mathbf{z}^{(v)}]/Q_{v,w} \big) \cong X_w \cap \big(\Omega_{v_0}^\circ v_0v\big)$.
	\end{proof}
	\end{prop}
The following result is an immediate consequence of Proposition \ref{prop:patch-variety-cell}.
\begin{cor}\label{cor:patch-prime}
	Let $v, w \in S_n$ for which $w \leq v$ in Bruhat order. 
	Then $Q_{v,w}$ is prime of codimension $\ell(w)$.
\end{cor}

	Example \ref{ex:2431} shows that by setting the variables in $\mathbf{y}^{(v)}$ to zero in a Schubert patch ideal, we obtain a Kazhdan-Lusztig ideal.
	
	\begin{ex}\label{ex:2431}
		Let $v\! =\! 34512$ and $w\! =\! 24513$, as\hspace{-0.1em} in\hspace{-0.1em} Examples \ref{ex:coordinate} and  \ref{ex:2431K}, respectively. Then
		\[Q_{v,w} = \left\langle z_{11}, \,z_{12}, \, z_{13},\,\begin{vmatrix}
			z_{21} & z_{22}\\
			1 & z_{32}
		\end{vmatrix},\,\begin{vmatrix}
			z_{21} & z_{23}\\
			1 & z_{33}
		\end{vmatrix},\,\begin{vmatrix}
			z_{22} & z_{23}\\
			z_{32} & z_{33}
		\end{vmatrix}  \right\rangle.\]
	We observe here that setting the variables $z_{32},\,z_{33} \in \mathbf{y}^{(v)} \subseteq \mathbf{z}^{(v)}$ to zero in the Schubert patch ideal $Q_{v,w}$ above, the resulting ideal is the Kazhdan-Lusztig ideal $I_{v,w}$ in Example \ref{ex:2431K}.
		\hfill \qedsymbol
	\end{ex}

\subsection{Torus actions and multigradings of Schubert patch ideals and Kazhdan-Lusztig ideals}
	We describe a positive multigrading of the coordinate ring $\mathbb{K}[\mathbf{z}^{(v)}]$ of $\mathbf{Z}^{(v)}$.
Following the same idea in \cite{woo2012grobner} but with different conventions, consider the right action (torus action) of $T$ on $\Omega^{\circ}_{v_0}v_0v$, where $T$ is the subgroup of diagonal matrices in $G$.
For each $v \in S_n$, this action independently scales the columns of $\mathbf{Z}^{(v)}$, to obtain $\mathbf{Z}^{(v)} \cdot \mathbf{t}$, $\mathbf{t} \in T$.
We then carefully choose a matrix $\mathbf{b} \in B_{-}$ such that the matrix $\mathbf{b}\cdot \mathbf{Z}^{(v)} \cdot \mathbf{t}$ has 1 in position $(v(j),j)$, $1 \leq j \leq n$.
By so doing, we dependently rescaled the rows of the matrix $\mathbf{Z}^{(v)} \cdot \mathbf{t}$.	

\begin{ex}\label{ex:multi}
	Let $v = 34512$ (as in Example \ref{ex:coordinate}).
	Setting 
	\[\mathbf{t} := \begin{pmatrix}
		t_1 & 0 & 0 & 0 & 0\\
		0 & t_2 & 0 & 0 & 0\\
		0 & 0 & t_3 & 0 & 0\\
		0 & 0 & 0 & t_4 & 0\\
		0 & 0 & 0 & 0 & t_5
	\end{pmatrix} \quad \text{and} \quad
	\mathbf{b} := \begin{pmatrix}
		t_4^{-1} & 0 & 0 & 0 & 0\\
		0 & t_5^{-1} & 0 & 0 & 0\\
		0 & 0 & t_1^{-1} & 0 & 0\\
		0 & 0 & 0 & t_2^{-1} & 0\\
		0 & 0 & 0 & 0 & t_3^{-1}
	\end{pmatrix},	
	\]
	we then obtain
	\[\mathbf{b} \cdot \mathbf{Z}^{(v)} \cdot \mathbf{t} = \begin{pmatrix}
		z_{11}t_1t_4^{-1} & {z_{12}}t_2t_4^{-1} & {z_{13}}t_3t_4^{-1} & 1 & {z_{15}}t_5t_4^{-1}\\
		z_{21}t_1t_5^{-1} & {z_{22}}t_2t_5^{-1} & z_{23}t_3t_5^{-1} & 0 & 1\\
		1 & {z_{32}}t_2t_1^{-1} & z_{33}t_3t_1^{-1} & 0 & 0\\
		0 & 1 & z_{43}t_3t_2^{-1} & 0 & 0\\
		0 & 0 & 1 & 0 & 0
	\end{pmatrix}.
	\]
	Writing weights additively, we can assign to the variables $z_{11}$, $z_{12}$, $z_{13}$, $z_{15}$, $z_{21}$, $z_{22}$, $z_{23}$, $z_{32}$, $z_{33}$ and $z_{43}$ in $\mathbf{Z}^{(v)}$ the weights $e_4-e_1$, $e_4-e_2$, $e_4-e_3$, $e_4 - e_5$, $e_5 - e_1$, $e_5 - e_2$, $e_5-e_3$, $e_1-e_2$, $e_1 - e_3$ and $e_2-e_3$, respectively, where $e_i$ is the $i$\text{th} standard basis vector in $\mathbb{Z}^5$.
	\qed
\end{ex}

In general, for a fixed $v \in S_n$, the action described above gives the variable $z_{ij}$ in matrix $\mathbf{Z}^{(v)}$ the weight $e_{v^{-1}(i)} - e_j$. This action therefore yields a $\mathbb{Z}^n$-grading of the coordinate ring $\mathbb{K}[\mathbf{z}^{(v)}]$ of $\mathbf{Z}^{(v)}$ and this multigrading is positive, i.e., the only polynomials in $\mathbb{K}[\mathbf{z}^{(v)}]$ of degree $\boldsymbol{0}$ are the elements of $\mathbb{K}$.
In \cite[Lemma 5.2]{woo2008governing}, Woo and Yong showed that every Kazhdan-Lusztig ideal $I_{v,w}$ is homogeneous with respect to the positive multigrading of the underlying ring by $\mathbb{Z}^n$.
This same proof can also be used to show the following result.

\begin{lem}\label{lem:homo-multi}
	Let $e_1,\ldots,e_n$ be the standard basis vectors of $\mathbb{Z}^n$. Under the multigrading where the variable  $z_{ij}$ has degree $e_{v^{-1}(i)} - e_j$, every Schubert patch ideal $Q_{v,w}$ is homogeneous.	
\end{lem}
 
Recall from the beginning of Subsection \ref{subsec:patch}, that for each non-zero variable $z_{ij}$ in the matrix $\mathbf{Z}^{(v)}$, we have the inequality $i < v(j)$.

 \begin{lem}\label{lem:positive-multi}
	Let $e_1,\ldots,e_n$ be the standard basis vectors of $\mathbb{Z}^n$.
	The multigrading where each variable $z_{ij}$ in $\mathbf{Z}^{(v)}$ has degree $e_{v^{-1}(i)} - e_j$ is positive.
	
	\begin{proof}
			Set $R:=\mathbb{K}[\mathbf{z}^{(v)}]$. It suffices to show that no product $\boldsymbol{m}:=z_{i_1,j_1}z_{i_2,j_2}z_{i_3,j_3}...z_{i_{_\ell},j_{_\ell}} \in R$ has degree $\boldsymbol{0}$.
		To this end, suppose there exists a monomial $\boldsymbol{m} := z_{i_1,j_1}z_{i_2,j_2}z_{i_3,j_3}...z_{i_{_\ell},j_{_\ell}} \in R$  (the variables involved are not necessarily distinct)  such that
		$\deg(\boldsymbol{m}) = \boldsymbol{0}$, i.e. \begin{equation}\label{eq:confg}
			\sum_{k=1}^\ell e_{v^{-1}(i_k)} - e_{j_k} = 0.
		\end{equation}
		
		Equation (\ref{eq:confg}) implies complete cancellations of vectors: $e_{v^{-1}(i_k)} = e_{j_{k'}}$, for some 
		$1 \leq k\neq k' \leq \ell$ i.e., $v^{-1}(i_k) = j_{k'}$, for some $k \neq k'$.
		If $v^{-1}(i_k) = j_{k'}$, then $i_k = v(j_{k'})$.
		Thus, $v(j_{k'}) < v(j_k)$, since by definition, $i_k < v(j_k)$ for $z_{i_k,j_k} \in \mathbf{Z}^{(v)}$.
		In summary, as a result of equation (\ref{eq:confg}), for any vector $e_{v^{-1}(i_k)}$ that cancels out the vector $e_{j_{k'}}$ ($k \neq k'$), we have $v(j_{k'}) < v(j_k)$. 
		Without loss of generality, suppose the variables involved in $\boldsymbol{m}$ are distinct.
		
		If $\ell = 2$, then $e_{v^{-1}(i_2)} = e_{j_1}$ and $e_{v^{-1}(i_1)} = e_{j_2}$. Therefore, by the above argument, it follows that $v(j_1) < v(j_2)$ and $v(j_2) < v(j_1)$, which is a contradiction.
		Hence, no product $z_{i_1,j_1}z_{i_2,j_2}$ in $R$ with degree $\boldsymbol{0}$.
		
		Let $\ell > 2$ and consider, for instance, the variable $z_{i_1,j_1}$ with degree $e_{v^{-1}(i_1)} - e_{j_1}$.
		Equation (\ref{eq:confg}) then implies that among the variables that appear in $\boldsymbol{m}$, there exists two other variables, say, $z_{i_{k_1},j_{k_1}}$ (with degree $e_{v^{-1}(i_{k_1})} - e_{j_{k_1}}$) and $z_{i_{k_2},j_{k_2}}$ (with degree $e_{v^{-1}(i_{k_2})} - e_{j_{k_2}}$), both different from $z_{i_1,j_1}$, such that $e_{v^{-1}(i_1)} = e_{j_{k_1}}$ and $e_{v^{-1}(i_{k_2})} = e_{j_1}$. 
		Then, by the argument in the second paragraph, it follows that 
		\begin{equation}\label{ineq:12}
			v(j_{k_1}) < v(j_1) \qquad \text{and} \qquad v(j_1) < v(j_{k_2}).
		\end{equation}
		If $k_1 = k_2$, then inequalities (\ref{ineq:12}) become a contradiction.
		Therefore, $k_1 \neq k_2$, and equation (\ref{eq:confg}) becomes
		\begin{equation}\label{eq:confg2a}
			\sum_{k \neq 1,k_1,k_2} e_{v^{-1}(i_k)} - e_{j_k} = e_{j_{k_2}} - e_{v^{-1}(i_{k_1})}.
		\end{equation}			
		
		Vectors $e_{v^{-1}(i_{k_1})}$ and $e_{j_{k_2}}$ are canceled out by some other vectors in equation (\ref{eq:confg2a}), thereby canceling out degrees of the variables $z_{i_{k_1},j_{k_1}}$ and $z_{i_{k_2},j_{k_2}}$ completely from equation (\ref{eq:confg2a}) (and hence from equation (\ref{eq:confg})), in addition to the degree of the variable $z_{i_1,j_1}$ that has been canceled out.
		Precisely, if ${v^{-1}(i_{k_1})} = {j_{k'_1}}$ and ${v^{-1}(i_{k'_2})} = {j_{k_2}}$, for some $k_1', k_2'$, with $1 < k'_1, k'_2 \leq \ell$ and different from $k_1,k_2$, then by the argument in the second paragraph, it follows that  
		\begin{equation}\label{ineq:123}
			v(j_{k'_1}) < v(j_{k_1}) \qquad \text{and} \qquad v(j_{k_2}) < v(j_{k'_2}).
		\end{equation}
		If $k'_1 = k'_2$, then combining inequalities (\ref{ineq:12}) and (\ref{ineq:123}), we have $v(j_{k_1}) < v(j_1) < v(j_{k_2}) < v(j_{k'_2}) = v(j_{k'_1}) < v(j_{k_1}),$ which is a contradiction.		
		Therefore $k'_1 \neq k'_2$, and equation (\ref{eq:confg2a}) becomes
		\begin{equation}\label{eq:confg2b}
			\sum_{k \neq 1,k_1,k_2,k'_1,k'_2} e_{v^{-1}(i_k)} - e_{j_k} = e_{j_{k'_2}} - e_{v^{-1}(i_{k'_1})}.
		\end{equation}
		%
		Observe that, so far, degrees of three variables have been completely eliminated from equation (\ref{eq:confg}).
		We can continue this way by finding next some other vectors that cancel out vectors $e_{v^{-1}(i_{k'_1})}$ and $e_{j_{k'_2}}$ on the right hand side of equation (\ref{eq:confg2b}).
		Continuing in this way, we will always arrive at a contradiction.
		Hence, the claim follows since $\ell$ is finite.
	\end{proof}	
 \end{lem}

\begin{ex}
	Let $v = 34512$ and $w = 24513$. The ideal $Q_{v,w}$ is given in Example \ref{ex:2431}.
	With respect to the above multigrading, we have 
	\(		\deg(z_{22})  =  e_5 - e_2\) and \( 
		\deg(z_{21}z_{32})  =  (e_5-e_1)+(e_1-e_2) \,=\,e_5-e_2.
		\)
Thus, the generator $\begin{vmatrix}
	z_{21} & z_{22}\\
	1 & z_{32}
\end{vmatrix}$ of $Q_{v,w}$ is homogeneous.
Similarly, the generators $\,\begin{vmatrix}
	z_{21} & z_{23}\\
	1 & z_{33}
\end{vmatrix}$ and $\begin{vmatrix}
	z_{22} & z_{23}\\
	z_{32} & z_{33}
\end{vmatrix}$ of $Q_{v,w}$ are also homogeneous.
	Therefore, $Q_{v,w}$ is homogeneous with respect to the positive multigrading of $\mathbb{K}[\mathbf{z}^{(v)}]$ by $\mathbb{Z}^5$.
	\qed
\end{ex}

\section{Gr\"obner Basis via Linkage and its analog in the Multigraded Setting}\label{sec:multigraded-set}
In this section, we briefly discuss some standard terminology as well as key technique we use in giving a proof for the main result of this paper.

	\subsection{Liaison Background}
	The primary reference for this subsection is \cite{gorla2013grobner}.
	Let $R = \mathbb{K}[\mathbf{z}]$ be a polynomial ring over a field $\mathbb{K}$, in $n$ variables $\mathbf{z} = z_1,\ldots,z_n$.
	Let $V_1, V_2, X \subseteq \mathbb{P}^n$ be subschemes defined by saturated homogeneous ideals $I_{V_1}$, $I_{V_2}$ and $I_X$ of $R$, respectively, and assume that $X$ is arithmetically Gorenstein. If $I_X \subseteq I_{V_1} \cap I_{V_2}$ and if $[I_X:I_{V_1}] = I_{V_2}$ and $[I_X:I_{V_2}] = I_{V_1}$, then $V_1$ and $V_2$ are \textbf{directly algebraically $G$-linked} by $X$, and we write $I_{V_1} \sim I_{V_2}$.
	If there is a sequence of links $V_1 \sim \cdots \sim V_k$ for some $k \geq 2$, then we say that $V_1$ and $V_k$ are in the same \textbf{$G$-liaison class (or Gorenstein liaison class)} and that they are \textbf{$G$-linked} in $k-1$ steps.
	If $V_k$ is a complete intersection, then we say $V_1$ is in the \textbf{Gorenstein liaison class of a complete intersection (shortened glicci).}
		Let $I \subseteq R$ be a homogeneous, saturated ideal. $I$ is said to be \textbf{Gorenstein in codimension $\boldsymbol{\leq c}$}, often written as $I$ is $G_c$, if the localization $(R/I)_\mathfrak{p}$ is a Gorenstein ring for any prime ideal $\mathfrak{p}$ of $R/I$ of height smaller than or equal to $c$. An ideal $I$ that is Gorenstein in codimension 0 is called \textbf{generically Gorenstein,} or $G_0$.
		Of utmost importance to us in our study is the terminology \emph{elementary G-biliaison}.
		
			\begin{defn}\label{def:elementary-G-biliaison}
			Let $R =  \mathbb{K}[\mathbf{z}]$ and $I,J \subseteq R$ be homogeneous, {saturated} and unmixed ideals such that $\text{ht}\,I = \text{ht}\, J = c$. $I$ is said to be obtained by an \textbf{elementary biliaison} of height $h$ from $J$ if there exists a Cohen-Macaulay ideal $N$ in $R$ of height $c-1$ such that $N \subseteq I \cap J$ and $I/N \cong (J/N)(-h)$ as $R/N$-modules. If in addition the ideal $N$ is $G_0$, then $I$ is obtained from $J$ via an \textbf{elementary G-biliaison}. If $h > 0$, we have an \textbf{ascending} elementary G-biliaison.
		\end{defn}
	
		\begin{thm}[\cite{kleppe2001gorenstein}, \cite{hartshorne2007generalized}]\label{connector}
		Let $I_1$ be obtained by an elementary G-biliaison from $I_2$. Then $I_2$ is G-linked to $I_1$ in two steps.
	\end{thm}

		\begin{defn}
			Let $C \subset B \subset R$ be homogeneous ideals such that $\text{ht}(C) = \text{ht}(B) - 1$ and $R/C$ is Cohen-Macaulay. 
			Let $f \in R_d$ be a homogeneous element of degree $d$ such that
			$C : \langle f \rangle = C$.
			The ideal $A := C + fB$ is called a \textbf{basic double link} of degree $d$ of $B$ on $C$. 
			If moreover $C$ is $G_0$ and $B$ is unmixed, then $A$ is a \textbf{basic double $G$-link} of $B$ on $C$.
		\end{defn}

\subsection{The Key Lemma}
In Section \ref{sec:grobner-s-d-i-k-l-i}, we will show that the essential minors of a Schubert patch ideal form a Gr\"{o}bner basis with respect to a certain term order.
We achieve this through the linkage-theoretic approach of Gorla, Migliore, and Nagel \cite{gorla2013grobner}.
This approach involves an important lemma that gives a sufficient condition for a set of polynomials to form a Gr\"{o}bner basis for the ideal it generates.
The standard graded version of this lemma was originally given in \cite[Lemma 1.12]{gorla2013grobner} and can also be found in \cite[Lemma 3.1]{fieldsteel2019gr}.
In Lemma \ref{suffCondGbasis2} below, we provide the multigraded version.

Let $R := \mathbb{K}[\mathbf{z}^{(v)}]$. 
Suppose an ideal $I$ is homogeneous with respect to the positive grading of $R$ by $\mathbb{Z}^n$.
Let $M:=R/I$.
Then $M$ is $\mathbb{Z}^n$-graded and equals 
 $\bigoplus_{\boldsymbol{e} \in \mathbb{Z}^n}M_{\boldsymbol{e}}$, where   $M_{\boldsymbol{e}}$ is a finite dimensional vector space over $\mathbb{K}$.
 The dimension $H_I(\boldsymbol{e}):=\dim_{\mathbb{K}} M_{\boldsymbol{e}} = \dim_{\mathbb{K}}\big( (R/I)_{\boldsymbol{e}}\big)$
 is the Hilbert function of $M$ evaluated at 
$\boldsymbol{e} \in \mathbb{Z}^n$.

\begin{lem}\label{suffCondGbasis2}
	Let $R$ be a positively $\mathbb{Z}^d$-graded polynomial ring over an arbitrary field $\mathbb{K}$. 
	Let $I, J$ and $N$ be homogeneous ideals with respect to the positive grading of $R$ by $\mathbb{Z}^d$ such that $N \subseteq I \cap J$.
	Let $A, B$ and $C$ be monomial ideals of $R$ such that $A \subseteq \text{in}_{\sigma}(I)$, $B = \text{in}_{\sigma}(J)$ and $C = \text{in}_{\sigma}(N)$ for some term order $\sigma$. 
	Suppose that there exists $\boldsymbol{e} \in \mathbb{Z}^d$ such that $(I/N)_{\boldsymbol{\ell}} \cong (J/N)_{\boldsymbol{\ell} - \boldsymbol{e}}$ and that $(A/C)_{\boldsymbol{\ell}} \cong (B/C)_{\boldsymbol{\ell} - \boldsymbol{e}}$ for all $\boldsymbol{\ell} \in \mathbb{Z}^d$.
	Then $A = \text{in}_{\sigma}(I)$.
	
	\begin{proof}
		We proceed exactly as in the proof of Lemma 3.1 in \cite{fieldsteel2019gr}, replacing the scalars in their proof with appropriate elements in $\mathbb{Z}^d$.
		For an arbitrary $\boldsymbol{\ell} \in \mathbb{Z}^d$, we have the following:
		\[\begin{array}{lllll}
			H_I(\boldsymbol{\ell}) & = & H_N(\boldsymbol{\ell}) - \dim_{\mathbb{K}}\big( (I/N)_{\boldsymbol{\ell}}\big) & & \big(0 \rightarrow I/N \rightarrow R/N \rightarrow R/I \rightarrow 0\big)\\
			& = & H_N(\boldsymbol{\ell}) - \dim_{\mathbb{K}}\big( (J/N)_{\boldsymbol{\ell}-\boldsymbol{e}}\big) & & \big((I/N)_{\boldsymbol{\ell}} \cong (J/N)_{\boldsymbol{\ell} - \boldsymbol{e}}\big)\\
			& = & H_N(\boldsymbol{\ell}) - \big(H_N(\boldsymbol{\ell} - \boldsymbol{e}) - H_J(\boldsymbol{\ell} - \boldsymbol{e})\big) & & \big(0 \rightarrow J/N \rightarrow R/N \rightarrow R/J \rightarrow 0\big)\\
			& = & H_C(\boldsymbol{\ell}) - \big(H_C(\boldsymbol{\ell} - \boldsymbol{e}) - H_B(\boldsymbol{\ell} - \boldsymbol{e})\big) & & \big(B = \text{in}_{\sigma}(J) \,\text{and}\, C = \text{in}_{\sigma}(N)\big)\\
			& = & H_C(\boldsymbol{\ell}) - \dim_{\mathbb{K}}\big( (B/C)_{\boldsymbol{\ell}-\boldsymbol{e}}\big) & & \big(0 \rightarrow B/C \rightarrow R/C \rightarrow R/B \rightarrow 0\big)\\
			& = & H_C(\boldsymbol{\ell}) - \dim_{\mathbb{K}}\big( (A/C)_{\boldsymbol{\ell}}\big) & & \big((A/C)_{\boldsymbol{\ell}} \cong (B/C)_{\boldsymbol{\ell} - \boldsymbol{e}}\big)\\
			& = & H_A(\boldsymbol{\ell}) & & \big(0 \rightarrow A/C \rightarrow R/C \rightarrow R/A \rightarrow 0\big).
		\end{array}\]
		Since $A \subseteq \text{in}_{\sigma}(I)$ and $H_I(\boldsymbol{\ell}) = H_A(\boldsymbol{\ell})$ for all $\boldsymbol{\ell} \in \mathbb{Z}^d$, it follows that $A = \text{in}_{\sigma}(I)$.
	\end{proof}
	
\end{lem}

	\section{Gr\"obner Basis via Linkage for Schubert patch ideals, Kazhdan-Lusztig Ideals and Schubert Determinantal Ideals}\label{sec:grobner-s-d-i-k-l-i}

		In this section, a proof will be given of the fact that the essential minors form Gr\"{o}bner bases for Schubert patch ideals under the term order $\succ$ as in Definition \ref{defnnotazmax}.
		This proof can be easily adapted to give a new proof of the known fact that the essential minors form Gr\"{o}bner bases for Kazhdan-Lusztig ideals.
		
\subsection{Gr\"{o}bner basis for Schubert patch ideals}
In what follows, we say $b$ is the last descent of $v \in S_n$ if
	$ b = \max \{i\,\,|\,\, v(i) > v(i+1), \,\, 1 \leq i < n\}.$
Given a permutation $v \in S_n$, the matrix $\mathbf{Z}^{(v)} s_b$ is the resulting matrix from swapping columns $b$ and $b+1$ of $\mathbf{Z}^{(v)}$.

\begin{ex}\label{ex:diagramZvsbFromZv}
	Let $v = 45312 \in S_5$ with last descent $b=3$.
	$\mathbf{Z}^{(v)}$ and $\mathbf{Z}^{(v)}s_3$ are given below:
	\[ 
	\mathbf{Z}^{(v)} = 
	\begin{pmatrix}
		{ z_{11}} & { z_{12}} & { z_{13}} & 1 & { z_{15}}\\
		{ z_{21}} & z_{22} & { z_{23}} & 0 & 1\\
		z_{31} & z_{32} & {1} & 0 & 0\\
		1 & z_{42} & 0 & 0 & 0\\
		0 & 1 & 0 & 0 & 0
	\end{pmatrix} \quad \text{and} \qquad 	\mathbf{Z}^{(v)}s_3 = \begin{pmatrix}
			{ z_{11}} & { z_{12}} & 1 & { z_{13}} & { z_{15}}\\
		{ z_{21}} & z_{22} & 0 & { z_{23}} & 1\\
		z_{31} & z_{32} & 0 & 1 & 0\\
		1 & z_{42} & 0 & 0 & 0\\
		0 & 1 & 0 & 0 & 0
	\end{pmatrix}
	.\vspace{-0.5cm}\]
	\qed
\end{ex}

\begin{defn}\label{defn:J-ideal}
	Let $v, w \in S_n$ and $b$ be the last descent of $v$. 
	Define an ideal $T_{vs_b,w} \subseteq \mathbb{K}[\mathbf{z}^{(v)}]$ as follows:
	\[{T}_{vs_b,w} = \langle \text{minors of size $1+ \rank(w_{p \times q})$ in $\big(\mathbf{Z}^{(v)}s_b\big)_{p \times q}$,\,\, $1 \leq p,q \leq n$} \rangle.\]
\end{defn}

\begin{ex}\label{ex:diagramZvsbFromZvideal}
	Let $v = 45312$ as in Example \ref{ex:diagramZvsbFromZv} and $w = 12543$.
	The diagram $D(w)$ and ideal $T_{vs_3,w}$ are:
	\[ 
D(w) = \vcenter{\hbox{
		\begin{tikzpicture}[scale=.45]
			\draw (0,0) rectangle (5,5);
			
			\draw (2,1) -- (2,3) -- (4,3) -- (4,2) -- (3,2) -- (3,1) -- (2,1);
			\draw[line width = .05ex] (3,2) --(3,3);
			\draw[line width = .05ex] (2,2) --(3,2);
			
			\filldraw (0.5,4.5) circle (.5ex);
			\draw[line width = .2ex] (0.5,0) -- (0.5,4.5) -- (5,4.5);
			\filldraw (1.5,3.5) circle (.5ex);
			\draw[line width = .2ex] (1.5,0) -- (1.5,3.5) -- (5,3.5);
			\filldraw (2.5,0.5) circle (.5ex);
			\draw[line width = .2ex] (2.5,0) -- (2.5,0.5) -- (5,0.5);
			\filldraw (3.5,1.5) circle (.5ex);
			\draw[line width = .2ex] (3.5,0) -- (3.5,1.5) -- (5,1.5);
			\filldraw (4.5,2.5) circle (.5ex);
			\draw[line width = .2ex] (4.5,0) -- (4.5,2.5) -- (5,2.5);
\end{tikzpicture}}}
	 \quad \text{and} \quad  T_{vs_3,w} = \left \langle \begin{vmatrix}
	z_{21} & z_{22}\\
	z_{31} & z_{32}
\end{vmatrix}, \begin{vmatrix}
	z_{21} & z_{23}\\
	z_{31} & 1
\end{vmatrix}, \begin{vmatrix}
	z_{22} & z_{23}\\
	z_{32} & 1
\end{vmatrix}, \begin{vmatrix}
	z_{21} & z_{22}\\
	1 & z_{42}
\end{vmatrix}, \begin{vmatrix}
	z_{31} & z_{32}\\
	1 & z_{42}
\end{vmatrix} \right\rangle. \vspace{-0.5cm}\]
	\qed
\end{ex}

\begin{defn}\label{defn:moving-from-T-Q}
		Let $v \in S_n$ and $b$ be the last descent of $v$.
		Define a map $\varphi \,:\, \mathbb{K}[\mathbf{z}^{(vs_b)}] \rightarrow \mathbb{K}[\mathbf{z}^{(v)}]$ by $z_{j,b+1} \mapsto z_{j,b}$, $z_{j,b} \mapsto z_{j,b+1}$ and $z_{j,i} \mapsto z_{j,i}$, if $i \neq b, b+1$, i.e., for $f \in \mathbb{K}[\mathbf{z}^{(vs_b)}]$, $\varphi(f)$ is obtained by simultaneous substitution of $z_{j,b}$ for $z_{j,b+1}$ and $z_{j,b+1}$ for $z_{j,b}$, while other variables remain unchanged.
\end{defn}

\begin{ex}\label{ex:diagramZvsbFromZvideal2}
	Continuing with Examples \ref{ex:diagramZvsbFromZv} and \ref{ex:diagramZvsbFromZvideal}, the ideal $Q_{vs_3,w}$ is given below:
	\[ Q_{vs_3,w} = \left \langle \begin{vmatrix}
		z_{21} & z_{22}\\
		z_{31} & z_{32}
	\end{vmatrix}, \begin{vmatrix}
		z_{21} & z_{24}\\
		z_{31} & 1
	\end{vmatrix}, \begin{vmatrix}
		z_{22} & z_{24}\\
		z_{32} & 1
	\end{vmatrix}, \begin{vmatrix}
		z_{21} & z_{22}\\
		1 & z_{42}
	\end{vmatrix}, \begin{vmatrix}
		z_{31} & z_{32}\\
		1 & z_{42}
	\end{vmatrix} \right\rangle. \]
Note that $T_{vs_3,w}$ (as in Example \ref{ex:diagramZvsbFromZvideal}) is the ideal obtained by applying the substitution map in Definition \ref{defn:moving-from-T-Q} to each of the generators of $Q_{vs_3,w}$. 
	\qed
\end{ex}

\begin{remark}\label{rem:trans-Q-T}
Let $v, w \in S_n$ and $b$ be the last descent of $v$.
It is easy to see that $f$ is a generator of $Q_{vs_b,w}$ (resp. $Q_{vs_b,ws_b}$) if and only if $\varphi(f)$ is a generator of $T_{vs_b,w}$ (resp. $T_{vs_b,ws_b}$).
\end{remark}

	\begin{lem}\label{lem:descent-ascent-equal-ideal}
		Let $v, w \in S_n$  
		and $b$ be the last descent of $v$.
		If $b$ is an ascent of $w$, then the ideals $Q_{v,w}$ and $T_{vs_b,w}$ are equal.
		
		\begin{proof}
			Consider all locations of essential boxes in $D(w)$.
			Since $b$ is an ascent of $w$, it follows that there is no essential box in column $b$ of $D(w)$.
			For any essential box $(p,q)$
			in $D(w)$, up to sign, set of minors of size $1+ \rank(w_{p \times q})$ in $\mathbf{Z}^{(v)}_{p \times q}$ is equal to set of minors of size $1+ \rank(w_{p \times q})$ in  $\big(\mathbf{Z}^{(v)}s_b\big)_{p \times q}$,
			i.e., both ideals $Q_{v,w}$ and $T_{vs_b,w}$ have the same essential minors, up to sign, for the locations of essential boxes in $D(w)$.
			This is true since the matrices $\mathbf{Z}^{(v)}$ and $\mathbf{Z}^{(v)}s_b$ are the same up to rearranging columns, and there is no essential box in column $b$ of $D(w)$.
			Precisely, columns $b$ (resp. $b+1$) in $\mathbf{Z}^{(v)}$ is the same as column $b+1$ (resp. $b$) in $\mathbf{Z}^{(v)}s_b$.
		\end{proof}
	\end{lem}

\begin{defn}
	Let $M$ be an $m \times n$ matrix. The $(m-\ell_1) \times (n-\ell_2)$ submatrix of $M$ formed by deleting rows $i_1,\ldots,i_{\ell_1}$ and columns $j_1,\ldots,j_{\ell_2}$ of $M$ will be denoted by $c_{i_1,\ldots,i_{\ell_1};j_1,\ldots,j_{\ell_2}}(M)$.
	If only rows $i_1,\ldots,i_{\ell_1}$ (resp. only columns $j_1,\ldots,j_{\ell_2}$) of $M$ are deleted, then we write  $c_{i_1,\ldots,i_{\ell_1};\,}(M)$ (resp. $c_{\,;j_1,\ldots,j_{\ell_2}}(M)$) to represent the resulting matrix.
\end{defn}

The following result aids the proof of Lemma \ref{lem:IvwToIvsbw}, Lemma \ref{lem:IvwToIvsbwsb}, Lemma \ref{lem:N-portion}, and Lemma \ref{lem:kazdan-to-schubert2} in this paper. These lemmas are fundamental to the proof of the main result (Theorem \ref{lem:gb-patch-multigrading2}) of this paper.

\begin{lem}\label{lem:cofactor-exp}
	Let $v \in S_n$, $b$ be the last descent of $v$ and
	$M$ be the $\alpha \times \beta$ northwest  submatrix of $\mathbf{Z}^{(v)}$ (or $\mathbf{Z}^{(v)}s_b$), for some $\alpha, \beta$. 
	Set $t:=1+\text{rank}(w_{\alpha \times \beta})$ and let $I$ be the ideal generated by the minors of size $t$ in $M$.
	For some $\ell < t \leq \alpha$, 
	we have the following:
	\begin{enumerate}
		\item\label{it:cofactor-exp1} The ideal $I$ is equal to the ideal generated by the minors of size $t - \ell$ in \linebreak[4] $c_{1,\ldots,{\ell};v^{-1}(1),\ldots,v^{-1}(\ell)}(M)$.
		\item\label{it:cofactor-exp2} If there are $m$ number of minors of size $t$ in $M$, then there exists a generating set $\mathcal{G}$ for $I$ of fewer than $m$ minors of size $t$ in $M$, and this generating set consists of minors that do not involve the variables in the first $\ell$ rows of $M$.
	\end{enumerate}
	 
	\begin{proof}
		Observe that the $\ell$ 1s in the submatrix $c_{\ell+1,\ldots,{\alpha};}(M)$ of $M$ are at positions $(i,v^{-1}(i))$, $1 \leq i \leq \ell$.
		So, considering the structure of these 1s, by cofactor expansion about the rows and columns where these 1s are located in $M$, starting from the uppermost 1,  we have that
		the ideal $I$ is equal to the ideal $I'$ generated by the minors of size $t - \ell$ in $G:=c_{1,\ldots,\ell;v^{-1}(1),\ldots,v^{-1}(\ell)}(M)$.
		Hence, part (\ref{it:cofactor-exp1}) is verified.
		Furthermore, given any minor of size $t-\ell$ that involves rows $i_1,\ldots,i_{t-\ell}$ and columns $j_1,\ldots,j_{t-\ell}$ in $G$, the minor that involves rows $1,\ldots,\ell,i_1,\ldots,i_{t-\ell}$ and columns $v^{-1}(1),\ldots,v^{-1}(\ell),j_1,\ldots,j_{t-\ell}$ in $M$ is a minor of size $t$ in $M$.
		Hence, each minor of size $t-\ell$ in $G$ is also a minor of size $t$ in $M$.
		In addition, none of these minors of size $t-\ell$ in $G$ involves the variables in rows $1,\ldots,\ell$ and columns $v^{-1}(1),\ldots,v^{-1}(\ell)$ of $M$, and so their corresponding minors of size $t$ in $M$ do not also involve these variables. 
				Hence, part (\ref{it:cofactor-exp2}) is verified.
	\end{proof}
\end{lem}

\begin{lem}\label{lem:IvwToIvsbwOneSide}
	Let $v, w \in S_n$ for which the last descent of $v$ is a descent of $w$.
	Let $b$ be the last descent of $v$ and $(\alpha,b)$, $\alpha \geq v(b+1)$, be a location for an essential box in column $b$, on or below row $v(b+1)$ of $D(w)$.
	Then the ideal generated by minors of size $\text{rank}({w}_{\alpha \times b})$ in $c_{v(b+1);b}\big(\mathbf{Z}_{\alpha \times b}^{(v)}\big)$ is equal to the ideal generated by minors of size $1+\text{rank}({w}_{\alpha \times b})$ in $\big(\mathbf{Z}^{(v)}s_b\big)_{\alpha \times b}$.
	\begin{proof}
		Set $a := v(b+1)$.
		Let $I$ be the ideal generated by minors of size $t:=1+\text{rank}({w}_{\alpha \times b})$ in $\big(\mathbf{Z}^{(v)}s_b\big)_{\alpha \times b}$ and $I'$ be the ideal generated by minors of size $t-1$ in $c_{a;b}\big(\mathbf{Z}_{\alpha \times b}^{(v)}\big)$. 
		We claim that $I = I'$.
		Set $M:= \big(\mathbf{Z}^{(v)}s_b\big)_{\alpha \times b}$.
		Then we observe that $M_{ab} = 1$, $M_{ib} = z_{i,b+1}$ for all $i < a$ and  $M_{ib} = 0$ for all $i > a$.
			As a result, there are $a$ 1s in the submatrix formed by rows $1,\ldots,a$ and columns $v^{-1}(1),\ldots,v^{-1}({a-1}),b$
			of $M$.
		Then, by part (\ref{it:cofactor-exp1}) of Lemma \ref{lem:cofactor-exp}, 
		the ideal $I$ is equal to the ideal generated by minors of size $t - a$ in $G:=c_{1,\ldots,a;v^{-1}(1),\ldots,v^{-1}(a-1),b}(M)$.
		Observe that the matrix $M$ has one row and one column more than the matrix $c_{a,b}(M')$, where $M':=\mathbf{Z}_{\alpha \times b}^{(v)}$; precisely, $c_{a;b}(M') = c_{a;b}(M)$.
		So there are $(a-1)$ 1s in the submatrix formed by rows $1,\ldots,a-1$ and columns $v^{-1}(1),\ldots,v^{-1}(a-1)$ of $c_{a;b}(M')$.
		Therefore, by part (\ref{it:cofactor-exp1}) of Lemma \ref{lem:cofactor-exp},
		the ideal $I'$ is equal to the ideal generated by minors of size $(t-1) - (a-1)  = t-a$ in $G':=c_{1,\ldots,a-1;v^{-1}(1),\ldots,v^{-1}(a-1)}\big(c_{a;b}(M')\big) = c_{1,\ldots,a;v^{-1}(1),\ldots,v^{-1}(a-1),b}(M')$.
		In summary, $I$ is the ideal generated by minors of size $t - a$ in $G$ and $I'$ is the ideal generated by minors of size $t-a$ in $G'$.
		Since $G  = G'$, it follows that $I = I'$.
		Note that if $t = a$, then both $I$ and $I'$ are unit ideal.
	\end{proof}
\end{lem}

\begin{ex}\label{ex:IvwToIvsbwOneSide}
Let $v = 45312$ and $w=12543$ as in Examples \ref{ex:diagramZvsbFromZv} and \ref{ex:diagramZvsbFromZvideal}.
The last descent of $v$ is $b=3$, which is a descent for $w$, and $v(b+1) = 1$.
	For the essential box $(4,3)$ in $D(w)$, 
	observe that the ideal generated by minors of size $\text{rank}({w}_{4 \times 3}) = 2$ in $c_{1;4}\big(\mathbf{Z}_{4 \times 3}^{(v)}\big)$ is equal to the ideal generated by minors of size $1+\text{rank}({w}_{4 \times 3}) = 3$ in $\big(\mathbf{Z}^{(v)}s_3\big)_{4 \times 3}$.
		\qed
\end{ex}

The following result gives a connection between the ideals $Q_{v,w}$ and $T_{vs_b,w}$ in terms of their generators.

\begin{lem}\label{lem:IvwToIvsbw}
	Let $v, w \in S_n$ for which the last descent $b$ of $v$ is a descent of $w$. If we write 
\begin{equation}\label{eq:Qvw-generator}
	Q_{v,w} = \left\langle z_{\text{max}}g_1 + r_1, \ldots, z_{\text{max}}g_k + r_k,h_1,\ldots,h_\ell\right \rangle,
\end{equation}
where the set $\mathcal{G}:=\{z_{\text{max}}g_1 + r_1, \ldots, z_{\text{max}}g_k + r_k,h_1,\ldots,h_\ell \}$ is a complete list of all essential minors in $Q_{v,w}$ and $z_{\text{max}}$ does not divide any term of $g_i$ or $r_i$ for any $1 \leq i \leq k$ nor any $h_j$ for $1 \leq j \leq \ell$,
then 
\(T_{vs_b,w} = \left\langle g_1, \ldots, g_k,h_1,\ldots,h_\ell\right \rangle.\)
	\begin{proof}
	Set $a := v(b+1)$ and $y:=z_{\text{max}}$. The integer $a$ is the row where $y$ is located in $\mathbf{Z}^{(v)}$, by Lemma \ref{lem:loczmax}.
	Given an arbitrary essential box $(\alpha,\beta)$ in $D(w)$, the minors of size $1+ \rank(w_{\alpha \times \beta})$ in $\mathbf{Z}^{(v)}_{\alpha \times \beta}$ are in $\mathcal{G}$, and for each of these essential minors $yg_i + r_i$ or $h_j$ of size $1+ \rank(w_{\alpha \times \beta})$ in $\mathcal{G}$, there exists corresponding minors $g_i$ or $h_j$ in the set $\mathcal{G}':=\{g_1, \ldots,g_k,h_1,\ldots,h_\ell \}$.
	We wish to show that each corresponding minor $g_i$ or $h_j$ belongs to $T_{vs_b,w}$, i.e., 
	we will first show that $\left\langle g_1, \ldots, g_k,h_1,\ldots,h_\ell\right \rangle$ is contained in $T_{vs_b,w}$.
	To this end, we will consider all possible locations of essential boxes in $D(w)$.	
	
	Let $(\alpha,\beta)$ be an essential box in $D(w)$ for which $\beta \neq b$.
	
	If $\beta < b$, then for any $\alpha$, the essential minors of size $1+ \rank(w_{\alpha \times \beta})$ in $\mathbf{Z}^{(v)}_{\alpha \times \beta}$ do not involve the variable $y$, and so they are some of the $h_i$ in $Q_{v,w}$. 
	These essential minors $h_i$ belong to $Q_{v,w}$ if and only if they belong to $T_{vs_b,w}$, since $\mathbf{Z}^{(v)}_{\alpha \times \beta} = \big(\mathbf{Z}^{(v)}s_b\big)_{\alpha \times \beta}$ in this case.
	
	Suppose $\beta > b$ and $\alpha < a$.
	This is similar to the previous case. Observe that for such $\alpha$ and $\beta$, $(\mathbf{Z}^{(v)}s_b\big)_{\alpha \times \beta}$ is the resulting matrix from swapping columns $b$ and $b+1$ of $\mathbf{Z}^{(v)}_{\alpha \times \beta}$.
	
	Suppose $\beta > b$, $\alpha \geq a$ and set $M:=\mathbf{Z}^{(v)}_{\alpha \times \beta}$.
	Then we observe that $M_{a,b+1} = 1$, $M_{i,b+1} = z_{i,b+1}$ for all $i < a$ and  $M_{i,b+1} = 0$ for all $i > a$.
	So, there are $a$ 1s in the submatrix formed by rows $1,\ldots,a$ and columns $v^{-1}(1),\ldots,v^{-1}(a-1),b$ of $M$.
	Then, by part (\ref{it:cofactor-exp2}) of Lemma \ref{lem:cofactor-exp},  
	the ideal $I$ generated by the minors of size $1+\text{rank}({w}_{\alpha \times \beta})$ in $M$ can be generated by a set $\mathcal{G}$ consisting of some essential minors of size $1+\text{rank}({w}_{\alpha \times \beta})$ in $M$, none of which involves, in particular, the variable $y$.
	Therefore, these essential minors in $\mathcal{G}$ are also some of the $h_i$ in ${Q_{v,w}}$.
	These essential minors $h_i$ belong to $Q_{v,w}$ if and only if they belong to ${T_{vs_b,w}}$, since, in this case,  $\big(\mathbf{Z}^{(v)}s_b\big)_{\alpha \times \beta}$ is the resulting matrix from swapping columns $b$ and $b+1$ of $\mathbf{Z}^{(v)}_{\alpha \times \beta}$. 	
	
	On the other hand, let $(\alpha,\beta)$ be an essential box in $D(w)$ for which $\beta = b$.
	
	If $\alpha < a$, then up to rearranging columns, the submatrix $\mathbf{Z}_{\alpha \times b-1}^{(v)} = \big(\mathbf{Z}^{(v)}s_b\big)_{\alpha \times b-1}$ contains an $\alpha \times \alpha$ upper triangular matrix with 1s on the diagonal; reason being that each variable $z_{ib}$, $1 \leq i \leq \alpha$, has 1 to its left in any of these submatrices.
	Hence, there exists a minor corresponding to this essential box $(\alpha , b)$ in $Q_{v,w}$ and $T_{vs_b,w}$ that is equal to 1.
	
	If $\alpha\geq a$, i.e., $(\alpha,b)$ is an essential box in column $b$, on or below row $a$ of $D(w)$, and $D$ is any of the essential minors in $Q_{v,w}$ associated to $(\alpha,b)$, then $D$ either involves row $a$, but not column $b$ of $\mathbf{Z}^{(v)}$, or involves column $b$, but not row $a$ of $\mathbf{Z}^{(v)}$, or involves both row $a$ and column $b$ of $\mathbf{Z}^{(v)}$, or involves neither row $a$ nor column $b$ of $\mathbf{Z}^{(v)}$.
	These four subcases are considered below, where $s:=1+\text{rank}({w}_{\alpha \times b})$:
	
	We combine the first and last subcases as one.
	Suppose $M$ is an $s \times s$ submatrix of $\mathbf{Z}^{(v)}_{\alpha \times b}$ that does not involve column $b$ of $\mathbf{Z}^{(v)}$, and set $D := \det(M)$.
	Then $M$ is also a submatrix of $\big(\mathbf{Z}^{(v)}s_b\big)_{\alpha \times b}$, and so $D$ is an essential minor in $Q_{v,w}$ if and only if it is an essential minor in $T_{vs_b,w}$; precisely, $D = h_i$, for some $i$, since it does not involve $y$. 
	
	Suppose $M$ is an $s \times s$ submatrix of $\mathbf{Z}^{(v)}_{\alpha \times b}$ that involves only column $b$, but not row $a$ of $\mathbf{Z}^{(v)}$, and set $D := \det(M)$.
	Note that $D$ is one of the $h_i$ in ${Q_{v,w}}$ since it does not involve $y$. 
	Expanding $D$ along this column $b$ of $\mathbf{Z}^{(v)}$ in $M$, it follows that $D$ belongs to the ideal generated by minors of size $s-1=\text{rank}({w}_{\alpha \times b})$ in $c_{\,;b}(M)$.
	Since $c_{\,;b}(M)$ is a submatrix of $c_{a;b}\big(\mathbf{Z}_{\alpha \times b}^{(v)}\big)$, it follows that $D$ belongs to the ideal generated by minors of size $s-1$ in $c_{a;b}\big(\mathbf{Z}_{\alpha \times b}^{(v)}\big)$.
	Consequently, by Lemma \ref{lem:IvwToIvsbwOneSide}, $D$ belongs to the ideal generated by minors of size $1+\text{rank}({w}_{\alpha \times b})$ in $\big(\mathbf{Z}^{(v)}s_b\big)_{\alpha \times b}$, i.e., $D$ belongs to $T_{vs_b,w}$.
	
	Next, suppose $M$ is an $s \times s$ submatrix of $\mathbf{Z}^{(v)}_{\alpha \times b}$ that involves both row $a$ and column $b$ of $\mathbf{Z}^{(v)}$, and set $D := \det(M)$.		
	In other words, $D$ involves the variable $y$, and so can be written in the form $D = yg+r$, where $y$ does not divide any term of $g$ or $r$.
	In this case, since $D$ belongs to the ideal generated  by the minors of size $s=1+\text{rank}({w}_{\alpha \times b})$ in $\mathbf{Z}^{(v)}_{\alpha\times b}$, it follows that $g$ belongs to the ideal generated by the minors of size $s-1=\text{rank}({w}_{\alpha \times b})$ in $c_{a;b}\big(\mathbf{Z}^{(v)}_{\alpha\times b}\big)$. 
	Consequently, by Lemma \ref{lem:IvwToIvsbwOneSide}, $g$ belongs to the ideal generated by the minors of size $1+\text{rank}({w}_{\alpha \times b})$ in $\big(\mathbf{Z}^{(v)}s_b\big)_{\alpha\times b}$, and so $g$ belongs to $T_{vs_b,w}$.
	
	So far, we have been able to show that the essential minors $h_i$ in $Q_{v,w}$ remain the same in $T_{vs_b,w}$ and $g_i$ belongs to $T_{vs_b,w}$ for each essential minor $yg_i+r_i$ in $Q_{v,w}$, where $y$ does not divide any term of $g_i$ or $r_i$. 
	Hence, $\left\langle g_1, \ldots, g_k,h_1,\ldots,h_\ell\right \rangle \subseteq T_{vs_b,w}$.
	
	To show the other inclusion $T_{vs_b,w} \subseteq \left \langle g_1, \ldots, g_k,h_1,\ldots,h_\ell \right \rangle$, we will consider all possible locations of essential boxes in $D(w)$, and compare their corresponding minors in $T_{vs_b,w}$ to the ones in $Q_{v,w}$. 
	To do this, it suffices to consider the last two cases above --
	we will consider any essential minor $D$ in $T_{vs_b,w}$ associated to an essential box $(\alpha,b)$ in column $b$, on or below row $a$ of $D(w)$ ($\alpha \geq a$) and for which $D$ either involves column $b$, but not row $a$ of $\mathbf{Z}^{(v)}s_b$, or involves both row $a$ and column $b$ of $\mathbf{Z}^{(v)}s_b$.
	This is enough to be shown since other cases will have same setup as above.
	Set $s:=1+\text{rank}({w}_{\alpha \times b})$.
	
	First, suppose $M$ is an $s \times s$ submatrix of $\big(\mathbf{Z}^{(v)}s_b\big)_{\alpha \times b}$ that involves only column $b$, but not row $a$ of $\mathbf{Z}^{(v)}s_b$, and set $D := \det(M)$.
	We claim that $D$ belongs to the ideal generated by the $h_i$.
	Assume $M$ uses at least one row from the rows above row $a$ of $\mathbf{Z}^{(v)}s_b$; otherwise, $D = 0$ since the entries below the 1 at position $(a,b)$ of $\mathbf{Z}^{(v)}s_b$ are all zero.
	Since $M$ is a submatrix of $M' := c_{a;\,}\big( \big(\mathbf{Z}^{(v)}s_b\big)_{\alpha \times b} \big)$, it follows that $D$ belongs to the ideal $I$ generated by minors of size $s$ in $M'$.
	Observe that each variable $z_{i,b+1}$, $1 \leq i < a$, above the 1 at position $(a,b)$ of $\mathbf{Z}^{(v)}s_b$ has 1 to its left in $\mathbf{Z}^{(v)}s_b$.
	So, there are $(a-1)$ 1s in the submatrix formed by rows $1,\ldots,(a-1)$ and columns $v^{-1}(1),\ldots,v^{-1}(a-1)$ of $M'$.
	Therefore, by part (\ref{it:cofactor-exp1}) of Lemma \ref{lem:cofactor-exp},
	the ideal $I$ is equal to the ideal generated by the minors of size $s - (a-1)$ in $c_{1,\ldots,a-1;v^{-1}(1),\ldots,v^{-1}(a-1)}(M') = c_{1,\ldots,a-1;v^{-1}(1),\ldots,v^{-1}(a-1)}\left( c_{a;\,}\big( \big(\mathbf{Z}^{(v)}s_b\big)_{\alpha \times b} \big) \right) = c_{1,\ldots,a;v^{-1}(1),\ldots,v^{-1}(a-1)} \big(\big(\mathbf{Z}^{(v)}s_b\big)_{\alpha \times b} \big)$.
	Even more, this ideal $I$ is equal to the ideal generated by the minors of size $s - (a-1)$ in $G:=c_{1,\ldots,a;v^{-1}(1),\ldots,v^{-1}(a-1),b} \big(\mathbf{Z}^{(v)}s_b\big)_{\alpha \times b} \big)$, since the entries left in the last column of 
	$c_{1,\ldots,a;v^{-1}(1),\ldots,v^{-1}(a-1)} \big(\mathbf{Z}^{(v)}s_b\big)_{\alpha \times b} \big)$
	are all zero.
	Now, we consider the ideal $I'$ generated by minors of size $s$ in $M'':=c_{a;\,}\big(\mathbf{Z}^{(v)}_{\alpha \times b}\big)$.
	The aforementioned $(a-1)$ 1s in $\mathbf{Z}^{(v)}s_b$ remain at the same locations in $\mathbf{Z}^{(v)}$ as $\mathbf{Z}^{(v)}s_b$.
	Therefore, by part (\ref{it:cofactor-exp1}) of Lemma \ref{lem:cofactor-exp},
the ideal $I'$ is equal to the ideal generated by the minors of size $s - (a-1)$ in $G':=c_{1,\ldots,a-1;v^{-1}(1),\ldots,v^{-1}(a-1)}\big(M''\big) =c_{1,\ldots,a;v^{-1}(1),\ldots,v^{-1}(a-1)}\big(\mathbf{Z}^{(v)}_{\alpha \times b}\big)$.
Observe that $G$ is a submatrix of $G'$, which implies $I \subseteq I'$, by the definitions of $I$ and $I'$.
Hence, $D$ belongs to $I'$.
Even more, by part (\ref{it:cofactor-exp2}) of Lemma \ref{lem:cofactor-exp},
	the ideal $I'$ can be generated by some (essential) minors of size $s$ in $M''$, and none of these minors involves $y$, since row $a$ of $\mathbf{Z}^{(v)}_{\alpha \times b}$ is not involved in the submatrix $M''$.
	Therefore, $D$ belongs to the ideal generated by the $h_i$.

	Lastly, suppose $M$ is an $s \times s$ submatrix of $\big(\mathbf{Z}^{(v)}s_b\big)_{\alpha \times b}$ that involves both row $a$ and column $b$ of $\mathbf{Z}^{(v)}s_b$, and set $D := \det(M)$.
	We claim that $D$ belongs to the ideal generated by the $g_i$s.
	Since $M$ is a submatrix of $\big(\mathbf{Z}^{(v)}s_b\big)_{\alpha \times b}$, it follows that $D$ belongs to the ideal $I$ generated by minors of size $s$ in $\big(\mathbf{Z}^{(v)}s_b\big)_{\alpha \times b}$. 
	There are $a$ 1s in the submatrix formed by rows $1,\ldots,a$ and columns $v^{-1}(1),\ldots,v^{-1}(a-1),b$ of $\big(\mathbf{Z}^{(v)}s_b\big)_{\alpha \times b}$.
	Therefore, by part (\ref{it:cofactor-exp1}) of Lemma \ref{lem:cofactor-exp},
	the ideal $I$ is equal to the ideal generated by the minors of size $s - a$ in $M':=c_{1,\ldots,a;v^{-1}(1),\ldots,v^{-1}(a-1),b}\big(\big(\mathbf{Z}^{(v)}s_b\big)_{\alpha \times b}\big)$.
	Now, let $M''$ be any $(s-a) \times (s-a)$ submatrix of $M'$.
	Since $M'$ is a submatrix of $G':=c_{1,\ldots,a-1;v^{-1}(1),\ldots,v^{-1}(a-1)}\big(\mathbf{Z}^{(v)}_{\alpha \times b}\big)$, it follows that $M''$ is also a submatrix of 
	$G'$.
	So, if \(M''' = 
	\begin{pmatrix}
		\ast_1 & y\\
		M'' & \ast_2
	\end{pmatrix},
	\) is an $(s-a+1) \times (s-a+1)$ matrix
	with entries $\ast_1$ and $\ast_2$ chosen such that $M'''$ is a submatrix of 
	$G'$,
	then using the 1s in the submatrix formed by row $1,\ldots,a-1$ and columns $v^{-1}(1),\ldots,v^{-1}(a-1)$ of $\mathbf{Z}^{(v)}_{\alpha \times b}$, we have from part (\ref{it:cofactor-exp1}) of Lemma \ref{lem:cofactor-exp} that the ideal $I'$ generated by minors of size $s$ in $\mathbf{Z}^{(v)}_{\alpha \times b}$ is equal to the ideal generated by minors of size $s-a+1$ in $G'$.
	Consequently, $\det(M''')$ belongs to $I'$, and hence, belongs belongs to $Q_{v,w}$.
	Since $\det(M''')$ belongs to $Q_{v,w}$ and $\det(M''') = y\det(M'') + r$, with none of the terms of $\det(M'')$ and $r$ divisible by $y$, it follows that $\det(M'')$ is one of the $g_i$s.
	Hence, the claim follows.
	\end{proof}
\end{lem}

\begin{ex}\label{ex:IvwToIvsbw}
	Continuing with Example \ref{ex:diagramZvsbFromZv} and Example \ref{ex:diagramZvsbFromZvideal}, where $v = 45312$ and $w=12543$, the last descent of $v$ occurs at $b=3$, which is a descent for $w$.
	The ideal $Q_{v,w}$ is:
	$Q_{v,w} = \left \langle \begin{vmatrix}
		z_{21} & z_{22}\\
		z_{31} & z_{32}
	\end{vmatrix}, \begin{vmatrix}
		z_{21} & z_{23}\\
		z_{31} & 1
	\end{vmatrix}, \begin{vmatrix}
		z_{22} & z_{23}\\
		z_{32} & 1
	\end{vmatrix}, \begin{vmatrix}
		z_{11} & z_{12} & z_{13}\\
		z_{21} & z_{22} & z_{23}\\
		1 & z_{42} & 0
	\end{vmatrix}, \begin{vmatrix}
		z_{11} & z_{12} & z_{13}\\
		z_{31} & z_{32} & 1\\
		1 & z_{42} & 0
	\end{vmatrix} \right \rangle$.
If we write the last two generators of $Q_{v,w}$ as $z_{13}g_1 + r_1, z_{13}g_2 + r_2$, where $z_{13}$ does not divide any term of $g_i$ or $r_i$, $1 \leq i \leq 2$, then the ideal generated by $g_1, g_2$ and the first three generators of $Q_{v,w}$ is exactly the ideal $T_{vs_3,w}$, given in Example \ref{ex:diagramZvsbFromZvideal}.
	\qed
\end{ex}

By Lemma \ref{lem:homo-multi}, $Q_{v,w}$, $v,w \in S_n$, is homogeneous with respect to the positive multigrading of $R:=\mathbb{K}[\mathbf{z}^{(v)}]$ by $\mathbb{Z}^n$, where 
each variable $z_{ij}$ has degree $e_{v^{-1}(i)}-e_{j}$.
In the following result, $T_{vs_b,w}$ is shown to be homogeneous with respect to this positive multigrading of $R$ by $\mathbb{Z}^n$.

\begin{lem}\label{lem:J-Homog}
Let $v,w \in S_n$ and 
$b$ be the last descent of $v$.
Let $e_1, \ldots,e_n$ be the standard basis vectors of $\mathbb{Z}^n$.
	Under the multigrading where the variable $z_{ij}$ has degree $e_{v^{-1}(i)}-e_{j}$, the ideal $T_{vs_b,w}$ is homogeneous.
	\begin{proof}
	Following Lemma \ref{lem:IvwToIvsbw}, it suffices to show that each $g_i \in {T_{vs_b,w}}$ is homogeneous with respect to the required multigrading.
	This is indeed true since each $z_{\text{max}}g_i + r_i$ in $Q_{v,w}$ is homogeneous with respect to the required multigrading by Lemma \ref{lem:homo-multi}.
		\end{proof}
\end{lem}

The following result gives a connection between the ideals $Q_{v,w}$ and $T_{vs_b,ws_b}$ in terms of their generators.

\begin{lem}\label{lem:IvwToIvsbwsb}
	Under the same hypotheses as Lemma \ref{lem:IvwToIvsbw}, we have
	\(T_{vs_b,ws_b} = \left\langle h_1,\ldots,h_\ell\right \rangle.\)
	\begin{proof}
		Set $a:=v(b+1)$, $w'\!:=\!ws_b$ and $y\!:=\!z_{\text{max}}$.
		Recall that $y$ is at the position $(a,b)$ of $\mathbf{Z}^{(v)}$.
		We will first show that $T_{vs_b,w'} \subseteq \left\langle h_1,\ldots,h_\ell\right \rangle$ by examining all possible locations of the essential boxes in $D(w')$ and their corresponding minors.
		Both $D(w)$ and $D(w')$ agree in all columns except for columns $b$ and $b+1$, by Lemma \ref{lem6.5}.
		
		If there is an essential box in column $b$ of $D(w')$, then it must be at position $(\alpha,b)$ in $D(w')$, for some $\alpha < w'(b)$. Since $w'(b) = w(b+1)$, it follows that $\alpha < w(b+1)$ and so this box $(\alpha,b)$ in $D(w')$ is the same box at location $(\alpha,b)$ in $D(w)$, by Lemma \ref{lem6.5}.
		If $(\alpha,b)$ is an essential box in $D(w)$, then it follows, by definition of essential box, that there is no box at location $(\alpha,b+1)$ of $D(w)$.
		So at the location $(\alpha,b+1)$ of the Rothe diagram $D(w)$, we expect either a dot $\Large \bigcdot$ to be at this location or a line to pass down through it.
		Since $b$ is a descent of $w$, it follows that there is a dot $\Large \bigcdot$ in column $b+1$ of $D(w)$ at row $w(b+1)$, and consequently, there cannot be another dot $\Large \bigcdot$ in column $b+1$ of $D(w)$.
		Also, there cannot be a line passing down through the location $(\alpha,b+1)$ of $D(w)$.
		Therefore, there cannot be an essential box in column $b$ of $D(w')$.
		
		Now, we consider the essential boxes that are on column $b+1$ of $D(w')$ strictly below row $w(b+1)$.
		For any essential box $(\alpha,b)$ in column $b$ of $D(w)$ that is strictly below row $w(b+1)$ (i.e., $\alpha > w(b+1)$), we have $\text{rank}({w'}_{\alpha \times (b+1)}) = 1+\text{rank}({w}_{\alpha \times b})$. 
		Then, since $\mathbf{Z}^{(v)}_{\alpha \times b}$ is a submatrix of $\mathbf{Z}^{(vs_b)}_{\alpha \times (b+1)}$, one observes that the ideal generated by minors of size $t:=1+\text{rank}({w'}_{\alpha \times (b+1)})$ in $\big(\mathbf{Z}^{(v)}s_b\big)_{\alpha \times (b+1)}$ is contained in the ideal generated by minors of size $t-1=1+\text{rank}({w}_{\alpha \times b})$ in $\mathbf{Z}^{(v)}_{\alpha \times b}$, and consequently, all minors corresponding to the essential boxes on column $b+1$ of $D(w')$ strictly below row $w(b+1)$ belong to $Q_{v,w}$.
			In particular, the first $a$ 1s in $M:=\big(\mathbf{Z}^{(v)}s_b\big)_{\alpha \times (b+1)}$ (counting from the top) are in columns $v^{-1}(1),\ldots,v^{-1}(a-1),b$ of $M$.
			So, by part (\ref{it:cofactor-exp2}) of Lemma \ref{lem:cofactor-exp}, 
			none of the minors of size $t-a$ in $G:=c_{1,\ldots,a;v^{-1}(1),\ldots,v^{-1}(a-1),b}(M)$ involves, in particular, $y=z_{ab}$.			
Recall that $M':=\mathbf{Z}^{(v)}_{\alpha \times b}$ is a submatrix of $M$; precisely, $M$ and $M'$ differ by column $b$ of $M$.
	The matrix $M'$ therefore has $(a-1)$ 1s in its submatrix formed by rows $1,\ldots,(a-1)$ and columns $v^{-1}(1),\ldots,v^{-1}(a-1)$.
So by cofactor expansion about these rows and corresponding columns  in $M'$, we obtain minors of size $1+\text{rank}({w}_{\alpha \times b}) - (a-1)  = (t-1) -a +1 = t-a$ in $G':=c_{1,\ldots,a-1;v^{-1}(1),\ldots,v^{-1}(a-1)}(M')$.
Observe that $G'$ differs from $G$ by its first row, and so, $G$ is a submatrix of $G'$.
Consequently, the set of all minors of size $t-a$ in $G$ is contained in the set of all minors of size $t-a$ in $G'$.
Therefore, these minors of size $t-a$ in $G$ are essential minors in $Q_{v,w}$ and since they do not involve $y$, we conclude that they are some of the $h_i$ in $Q_{v,w}$.
		
		Next, we consider any essential box, say $(\alpha,\beta)$, in $D(w')$ that is different from any of the essential boxes on column $b+1$ strictly below row $w(b+1)$.
		For any such essential box $(\alpha,\beta)$ in $D(w')$, consider the matrices $\mathbf{Z}^{(v)}_{\alpha \times \beta}$ and $\big(\mathbf{Z}^{(v)}s_b\big)_{\alpha \times \beta}$. Observe that
		$\big(\mathbf{Z}^{(v)}s_b\big)_{\alpha \times \beta}$ is the resulting matrix from swapping columns $b$ and $b+1$ of $\mathbf{Z}^{(v)}_{\alpha \times \beta}$, and 
		$\text{rank}({w}_{\alpha \times \beta}) = \text{rank}({w'}_{\alpha \times \beta})$.
		If $\beta < b$ with any $\alpha$ or $\beta > b$ and $\alpha < a$, the essential minors of size $1+ \rank(w'_{\alpha \times \beta})=1+ \rank(w_{\alpha \times \beta})$ in $\big(\mathbf{Z}^{(v)}s_b\big)_{\alpha \times \beta}=\mathbf{Z}^{(v)}_{\alpha \times \beta}$ do not involve the variable $y$, and so they are some of the $h_i$ in $Q_{v,w}$. 
		If $\beta > b$ and $\alpha \geq a$, then just like this case in the proof of Lemma \ref{lem:IvwToIvsbw}, by cofactor expansion, the ideal generated by the minors of size $1+\text{rank}({w}_{\alpha \times \beta})$ (and hence, in $\big(\mathbf{Z}^{(v)}s_b\big)_{\alpha \times \beta}$ since $\big(\mathbf{Z}^{(v)}s_b\big)_{\alpha \times \beta}$ can be obtained from $\mathbf{Z}^{(v)}_{\alpha \times \beta}$ by swapping its columns $b$ and $b+1$)
		can be generated by a set consisting of some essential minors in $\mathbf{Z}^{(v)}_{\alpha \times \beta}$, none of which involves, in particular, the variable $y$, and hence, are the some of the $h_i$ in $Q_{v,w}$.
		Consequently, these essential minors $h_i$ belong to $Q_{v,w}$ if and only if they belong to $T_{vs_b,w'}$.
		Therefore, in all, we obtain $T_{vs_b,ws_b} \subseteq \left\langle h_1,\ldots,h_\ell\right \rangle$.
		
	To show the other inclusion  $ \left\langle h_1,\ldots,h_\ell\right \rangle \subseteq T_{vs_b,ws_b}$, we will consider all essential minors in $Q_{v,w}$ that do not involve $y$.
	To do this, it suffices to consider the essential boxes in column $b$, on or below row $a$ of $D(w)$, whose (essential) minors do not involve $y$. 
	The diagrams below show typical locations of boxes in columns $b-1$ and $b$ of $D(w)$, and their corresponding locations in $D(ws_b)$.
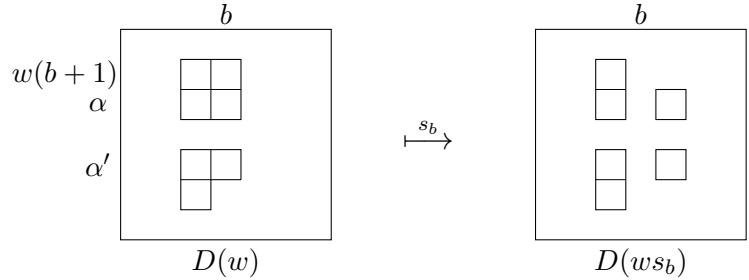
	\[
	\vcenter{\hbox{
			\begin{tikzpicture}[scale=.4]
				\draw (0,0) rectangle (7,7);
				\node at (3.5,7.5) {\small $b$};
	`			\node at (3.5,-0.75) {\small $D(w)$};
				\node at (-0.75,2.5) {\small $\alpha'$};
				\node at (-0.75,4.5) {\small $\alpha$};
				\node at (-1.85,5.5) {\small $w(b+1)$};
				
				\draw (2,4) rectangle (4,6);
				
				\draw[line width = .05ex] (2,5) -- (4,5);
				\draw[line width = .05ex] (3,4) --(3,6);

				\draw (4,2) -- (4,3) -- (2,3) -- (2,1) -- (3,1) -- (3,2) -- (4,2);
				
				\draw[line width = .05ex] (3,2) -- (2,2);
				\draw[line width = .05ex] (3,2) --(3,3);

	\end{tikzpicture}}}
\qquad  \stackrel{s_b}{\longmapsto} \qquad   
\vcenter{\hbox{
\begin{tikzpicture}[scale=.4]
\draw (0,0) rectangle (7,7);
	\node at (3.5,7.5) {\small $b$};
	`\node at (3.5,-0.75) {\small $D(ws_b)$};

	\draw (2,4) rectangle (3,6);
	\draw (4,4) rectangle (5,5);
	
	\draw[line width = .05ex] (2,5) -- (3,5);

	\draw (2,1) rectangle (3,3);
\draw (4,2) rectangle (5,3);
		
	\draw[line width = .05ex] (2,2) -- (3,2);

\end{tikzpicture}}}\vspace{-0.25cm}\]
Assume $w(b+1) \geq a$ and let $(\alpha,b)$ be the first esential box, from the top, in column $b$ of $D(w)$.
First, we claim that any essential minor in $Q_{v,w}$ associated to this essential box $(\alpha,b) \in D(w)$ that does not involve $y$ belongs to the set of all essential minors in $T_{vs_b,ws_b}$ associated to the boxes $(\alpha,b-1), (\alpha,b+1) \in D(ws_b)$.
To see this, we first observe that there are $(a-1)$ 1s in the submatrix formed by rows
$1,\ldots,(a-1)$ and 
columns $v^{-1}(1),\ldots,v^{-1}(a-1)$ 
of $M:=\mathbf{Z}^{(v)}_{\alpha \times b}$, 
and so, by part (\ref{it:cofactor-exp1}) of Lemma \ref{lem:cofactor-exp}, the minors of size $t:=1+\text{rank}({w}_{\alpha \times b})$ in $M$ equals minors of size $t - (a-1)  = t-a+1$ in $G:=c_{1,\ldots,a-1;v^{-1}(1),\ldots,v^{-1}(a-1)}(M)$.

The essential box $(\alpha,b+1)$ is an essential box in column $b+1$ of $D(w')$ strictly below row $w(b+1)$.
Consequently, by cofactor expansion about the first $a$ rows of $M':=\big(\mathbf{Z}^{(v)}s_b\big)_{\alpha \times (b+1)}$ and the columns where 1s are located in these first $a$ rows, we have that the ideal $I'$ generated by minors of size $1+\text{rank}({w'}_{\alpha \times (b+1)}) = 1 + (1+\text{rank}({w}_{\alpha \times b})) = 1+t$ in $M'$ is equal to the ideal generated by minors of size $(1+t)-a = t-a+1$ in $G':=c_{1,\ldots,a;v^{-1}(1),\ldots,v^{-1}(a-1),b}(M')$.
Here, $G$ differs from $G'$ by its first row, which is row $a$ of $M$.
Consequently, any minor of size $t-a+1$ in $G$ (which is an essential minor of size $t$ in $M$, by part (\ref{it:cofactor-exp2}) of Lemma \ref{lem:cofactor-exp}) that does not involve row $a$ (i.e., does not involve $y$, and hence, it is one of the $h_i$) is one of the minors of size $t-a+1$ in $G'$, and hence, belongs to $I'\subseteq T_{vs_b,ws_b}$.

Furthermore, for the essential box $(\alpha,b\!-\!1)$ in $D(w)$ and $D(w')$, we have that $\mathbf{Z}^{(v)}_{\alpha \times b-1}=\big(\mathbf{Z}^{(v)}s_b\big)_{\alpha \times b\!-\!1}$ and $ 1+\rank(w'_{\alpha \times b-1})=1+\rank(w_{\alpha \times b\!-\!1}) = 1+\rank(w_{\alpha \times b}) = t$. 
Set $M'':=\mathbf{Z}^{(v)}_{\alpha \times b-1}=\big(\mathbf{Z}^{(v)}s_b\big)_{\alpha \times b-1}$.
Using the $(a\!-\!1)$ 1s in the first $a\!-\!1$ rows of $M''$, it follows from part (\ref{it:cofactor-exp1}) of Lemma \ref{lem:cofactor-exp} that
the ideal $I''$ generated by minors of size $t$ in $M''$ is equal to the ideal generated by minors of size $t\!-\!a+1$ in $G'':=c_{1,\ldots,a\!-1;j_1,\ldots,j_{a\!-1}}(M'')$.
Here, $G$ differs from $G''$ by its last column, which is column $b$ of $M$.
Consequently, any minor of size $t-a+1$ in $G$ (which is an essential minor of size $t$ in $M$, by part (\ref{it:cofactor-exp2}) of Lemma \ref{lem:cofactor-exp}) that does not involve column $b$ (i.e., does not involve $y$, and hence, is one of the $h_i$) is one of the minors of size $t\!-\!a+1$ in $G''$, and hence, belongs to $I'' \subseteq T_{vs_b,ws_b}$.
Hence, the first claim follows.

Second, we claim that any essential minor in $Q_{v,w}$ associated to the essential box $(\alpha',b) \in D(w)$ that does not involve $y$ belongs to the set of all minors in $T_{vs_b,ws_b}$ associated to the boxes $(\alpha',b-1), (\alpha',b+1) \in D(ws_b)$.
This is shown exactly the same way as the first claim.
Repeating this same procedure for all other essential boxes in column $b$ of $D(w)$, we obtain the desired inclusion:  $ \left\langle h_1,\ldots,h_\ell\right \rangle \subseteq T_{vs_b,ws_b}$.
Therefore, $ T_{vs_b,ws_b} = \left\langle h_1,\ldots,h_\ell\right \rangle$.
	\end{proof}
\end{lem}

\begin{ex}\label{ex:IvwToIvsbwsb}
	Continuing with Examples \ref{ex:diagramZvsbFromZv} and \ref{ex:diagramZvsbFromZvideal}, where $v = 45312$ and $w=12543$, the last descent $b=3$ of $v$ is a descent for $w$.
	$Q_{v,w}$ is given in Example \ref{ex:IvwToIvsbw}.
	The ideal $T_{vs_b,ws_b}$ is generated by the first three generators of $Q_{v,w}$, since they do not involve $z_{\text{max}} = z_{13}$.
	\qed
\end{ex}

The following result is an immediate consequence of Lemma \ref{lem:IvwToIvsbw} and Lemma \ref{lem:IvwToIvsbwsb}.

\begin{cor}\label{cor:N-contained-I-int-J}
	Let $v,w \in S_n$ for which the last descent of $v$ is a descent of $w$.
	If $b$ is the last descent of $v$,
	then $T_{vs_b,ws_b} \subseteq Q_{v,w} \cap T_{vs_b,w}$.
	\end{cor}

By Lemma \ref{lem:homo-multi}, the ideal $Q_{v,w}$ is homogeneous with respect to a positive multigrading of $\mathbb{K}[\mathbf{z}^{(v)}]$ by $\mathbb{Z}^n$.
If we set $y:=z_{\text{max}}$ and write $Q_{v,w}$ in the form \(Q_{v,w} = \left\langle yg_1 + r_1, \ldots, yg_k + r_k,\right.$ $\left.h_1,\ldots,h_\ell\right \rangle\), where 
the set $\{yg_1 + r_1, \ldots, yg_k + r_k,h_1,\ldots,h_\ell \}$ is a complete list of all essential minors in $Q_{v,w}$
and $y$ does not divide any term of $g_i$ or $r_i$ for any $1 \leq i \leq k$ nor any $h_j$ for $1 \leq j \leq \ell$, then each generator $h_j$ is homogeneous with respect to this positive multigrading of $\mathbb{K}[\mathbf{z}^{(v)}]$ by $\mathbb{Z}^n$.
Hence, the following result is a direct consequence of Lemma \ref{lem:homo-multi} and Lemma \ref{lem:IvwToIvsbwsb}.

\begin{cor}\label{cor:N-Homog}
	Let $v,w \in S_n$ and 
	$b$ be the last descent of $v$.
	Let $e_1, \ldots,e_n$ be the standard basis vectors of $\mathbb{Z}^n$.
	Under the multigrading where the variable $z_{ij}$ has degree $e_{v^{-1}(i)}-e_{j}$, the ideal $T_{vs_b,ws_b}$ is homogeneous.
	\end{cor}

Let $P$ be a $k \times n$ matrix.
For any ordered sequence $1 \leq i_i < \cdots < i_k \leq n$ of $k$ integers between $1$ and $n$, let $P_{[\,;i_1,\ldots,i_k]}$ be the determinant of the $k \times k$ matrix whose columns are the $i_i,\ldots,i_k$ columns of $P$.
For any two ordered sequences $i_1 < i_2 < \cdots < i_{k-1}$ and $j_1 < j_2 < \cdots < j_{k+1}$ of positive integers $1 \leq i_{\ell}, j_m \leq n$, the following equation is valid:
\begin{eqnarray}\label{eq:plucker relation}
	\sum_{\ell=1}^{k+1} (-1)^{\ell} P_{[\,;i_1,\ldots,i_{k-1},j_{\ell}]}  P_{[\,;j_1,\ldots,\widehat{j}_{\ell},\ldots,j_{k+1}]} = 0,
	\vspace{-0.25cm}
\end{eqnarray}
where $j_1,\ldots,\widehat{j}_{\ell},\ldots,j_{k+1}$ denotes the sequence $j_1,\ldots,j_{k+1}$ with the term $j_\ell$ omitted.
Equation (\ref{eq:plucker relation}) is called a \textbf{Grassmann-Pl\"{u}cker relation}.

The following result can be easily verified. It aids the proof of Lemma \ref{lem:dif-belongs-to-N}.

\begin{lem}\label{lem:MtoP}
Let $P$ be a $k \times (m+k)$ block matrix formed by horizontal concatenation of a $k \times m$ matrix $M$ and a $k \times k$ identity matrix $I_k$.
For $\ell \leq k$, let
$M_{[p_1,\ldots,p_\ell;q_1,\ldots,q_\ell]}$ denote an $\ell \times \ell$ minor of $M$ that involves rows $1 \leq p_1 < \cdots < p_\ell \leq k$ and columns $1 \leq q_1 < \cdots < q_\ell \leq m$.
Then for any ordered sequence $m+1 < q_{\ell + 1} < \cdots < q_k \leq m+k$, where each $q_{i}-m$, $\ell < i \leq k$, belongs to  $[k] \setminus \{p_1,\ldots,p_{\ell}\}$ $([k]:=\{1,\ldots,k\})$, we have $M_{[p_1,\ldots,p_\ell;q_1,\ldots,q_\ell]} = P_{[\,;q_1,\ldots,q_\ell,q_{\ell+1},\ldots,q_{k}]}$.
\end{lem}

\begin{defn}\label{defn:NalphaBeta}
	Let $v, w \in S_n$ for which the last descent $b$ of $v$ is a descent of $w$, and let $(\alpha,b)$ be a location for an essential box in column $b$ for which $\alpha \geq a$, where $a:=v(b+1)$.
	Define $N_{(\alpha,b)}$ to be the ideal generated by the minors of size $1+\text{rank}(w_{\alpha \times b})$ in $c_{a;\,}(\mathbf{Z}^{(v)}_{\alpha \times b})$ and $c_{\,;b}(\mathbf{Z}^{(v)}_{\alpha \times b})$.
\end{defn}

\begin{lem}\label{lem:dif-belongs-to-N}
	Let $v, w \in S_n$ for which the last descent $b$ of $v$ is a descent of $w$.
	Let $(\alpha,b)$ and $(\alpha',b)$, where $\alpha' \geq \alpha \geq a$, $a:=v(b+1)$, be locations of any two essential boxes in column $b$, on or below row $a$ of $D(w)$.
	Set $M:=\mathbf{Z}^{(v)}$, $t:=1+\text{rank}(w_{\alpha \times b})$ and $t':=1+\text{rank}(w_{\alpha' \times b})$.
	Then for any 
	$1 \leq p_1 < \cdots < p_{i-1} < a < p_{i+1} < \cdots < p_t \leq \alpha$ and $1 \leq p'_1 < \cdots < p'_{i'-1} < a < p'_{i'+1} < \cdots < p'_{t'} \leq \alpha'$, we have that
		\begin{eqnarray}\nonumber
		M_{[p_1,\ldots,p_{i-1},p_{i+1},\ldots,p_{t};q_1,\ldots,q_{t-1}]}\cdot M_{[p'_1,\ldots,p'_{i'-1},a,p'_{i'+1},\ldots,p'_{t'};q'_1,\ldots,q'_{t'-1},b]} & & \\ \label{eq:dif-belongs-to-N}
		& \!\!\!\!\!\!\!\!\!\!\!\!\!\!\!\!\!\!\!\!\!\!\!\!\!\!\!\!\!\!\!\!\!\!\!\!\!\!\!\!\!\!\!\!\!\!\!\!\!\!\!\!\!\!\!\!\!\!\!\!\!\!\!\!\!\!\!\!\!\!\!\!\!\!\!\!\!\!\!\!\!\!\!\!\!\!\!\!\!\!\!\!\!\!\!\!\!\!\!\!\!\!\!\!\!\!\!\!\!\!\!\!- \,\,  M_{[p'_1,\ldots,p'_{i'-1},p'_{i'+1},\ldots,p'_{t'};q'_1,\ldots,q'_{t'-1}]}\cdot M_{[p_1,\ldots,p_{i-1},a,p_{i+1},\ldots,p_{t};q_1,\ldots,q_{t-1},b]} &
	\end{eqnarray}
belongs to $N_{(\alpha,b)} + N_{(\alpha',b)}$, 
where $1 \leq q_1 < \cdots < q_{t-1} < b$ and $1 \leq q'_1 < \cdots < q'_{t'-1} < b$.
\begin{proof}
		Set $k:=\alpha'$, $\mathbf{p} := \{p_1,\ldots,p_{i-1},p_{i+1},\ldots,p_{t}\}$ and $\mathbf{p'}: = \{p'_1,\ldots,p'_{i'-1},p'_{i'+1},\ldots,p'_{t'}\}$. 
	Recall that $M=\mathbf{Z}^{(v)}$.
	Let $M_{k \times b}$ be the northwest $k \times b$ submatrix of $M$.
	Let $P$ be the $k \times (b+k)$ block matrix formed by horizontal concatenation of $M_{k \times b}$ and the $k \times k$ identity matrix $I_k$. 
	If $\{q_{t},q_{t+1}\ldots,q_{k-1}\} := \{b+r\,|\, 1 \leq r \leq k,\, r \notin \big(\mathbf{p} \cup \{a\}\big)\}$ and $\{q'_{t'},q'_{t'+1},\ldots,q'_{k-1}\}:=\{b+r'\,|\, 1 \leq r' \leq k,\, r' \notin \big(\mathbf{p'} \cup \{a\}\big) \}$, then we have the following equality by Lemma \ref{lem:MtoP}:
	\[
	\begin{array}{l}
		f_1 :=  M_{[p_1,\ldots,p_{i-1},p_{i+1},\ldots,p_{t};q_1,\ldots,q_{t-1}]}  =  P_{[\,;q_1,\ldots,q_{t-1},b+a,q_t,q_{t+1} ,\ldots,q_{k-1}]},\\
		f_2:= M_{[p_1,\ldots,p_{i-1},a,p_{i+1},\ldots,p_{t};q_1,\ldots,q_{t-1},b]} = P_{[\,;q_1,\ldots,q_{t-1},b,q_t,q_{t+1},\ldots,q_{k-1}]},\\
		f'_1:= M_{[p'_1,\ldots,p'_{i'-1},p'_{i'+1},\ldots,p'_{t'};q'_1,\ldots,q'_{t'-1}]} = P_{[\,;q'_1,\ldots,q'_{t'-1},b+a,q'_{t'},q'_{t'+1} ,\ldots,q'_{k-1}]},\\
		f'_2:= M_{[p'_1,\ldots,p'_{i'-1},a,p'_{i'+1},\ldots,p'_{t'};q'_1,\ldots,q'_{t'-1},b]} = P_{[\,;q'_1,\ldots,q'_{t'-1},b,q'_{t'},q'_{t'+1},\ldots,q'_{k-1}]}.
	\end{array}
	\]
	Assume $q_1 < \cdots < q_{t-1} < q_t < q_{t+1} < \cdots < q_{k-1}$ and $q'_1 < \cdots < q'_{t'-1} < b < b+a < q'_{t'} < q'_{t'+1} < \cdots < q'_{k-1}$.
	Then we have the following Grassmann-Pl\"{u}cker relations:
	\begin{equation}\label{eq:plucker relation2}
		\sum_{{\ell}=1}^{k+1} (-1)^{\ell} P_{[\,;q_1,\ldots,q_{t-1},q_{t},q_{t+1} ,\ldots,q_{k-1},q'_{j_{_\ell}}]}  \,P_{[\,;\underbrace{q'_1,\ldots,q'_{t'-1},b,b+a,q'_{t'},q'_{t'+1},\ldots,q'_{k-1}}_{\text{$q'_{j_{_\ell}}$ omitted}}]} = 0
	\end{equation}
	From the equation above, observe that when $\ell=t'$, we have 
	\[
	\begin{array}{ll}
		& (-1)^{t'} P_{[\,;q_1,\ldots,q_{t-1},q_{t},q_{t+1} ,\ldots,q_{k-1},b]}  \,P_{[\,;{q'_1,\ldots,q'_{t'-1},b+a,q'_{t'},q'_{t'+1},\ldots,q'_{k-1}}]}	\\
		= & (-1)^{t'} (-1)^{k-t} P_{[\,;q_1,\ldots,q_{t-1},b,q_{t},q_{t+1} ,\ldots,q_{k-1}]}  \,P_{[\,;{q'_1,\ldots,q'_{t'-1},b+a,q'_{t'},q'_{t'+1},\ldots,q'_{k-1}}]}	= (-1)^{k+t'-t} f_2f'_1,
	\end{array}
\vspace{-0.25cm}
	\]
	and when $\ell=t'+1$, we have
	\[
	\begin{array}{ll}
		& (-1)^{t'+1} P_{[\,;q_1,\ldots,q_{t-1},q_{t},q_{t+1} ,\ldots,q_{k-1},b+a]}  \,P_{[\,;{q'_1,\ldots,q'_{t'-1},b,q'_{t'},q'_{t'+1},\ldots,q'_{k-1}}]}	\\
		= & (-1)^{t'+1} (-1)^{k-t} P_{[\,;q_1,\ldots,q_{t-1},b+a,q_{t},q_{t+1} ,\ldots,q_{k-1}]}  \,P_{[\,;{q'_1,\ldots,q'_{t'-1},b,q'_{t'},q'_{t'+1},\ldots,q'_{k-1}}]}	= (-1)^{k+t'-t+1} f_1f'_2.
	\end{array}
\vspace{-0.15cm}
	\]
	Therefore, for $q'_{j_{_\ell}} \!\in\! \{q'_1,\ldots,q'_{t'-1},q'_{t'},q'_{t'+1},\ldots,q'_{k-1}\}$ with 
	$q'_{j_{_\ell}} \!\!\neq\! b,b+a$, $f_1f'_2 - f_2f'_1$ can be isolated from equation (\ref{eq:plucker relation2}) and then written as sum of the products $h\big(j'_{_{\ell}}\big) \!\cdot\! h'\big(j'_{_{\ell}}\big)$, where  
	$h\big(j'_{_{\ell}}\big):=\! P_{[\,;q_1,\ldots,q_{t-1},q_{t},q_{t+1} ,\ldots,q_{k-1},q'_{j_{_\ell}}]}$ and  $h'\big(j'_{_{\ell}}\big)\!:=\!P_{[\,;\underbrace{q'_1,\ldots,q'_{t'-1},b,b+a,q'_{t'},\ldots,q'_{k-1}}_{\text{$q'_{j_{_\ell}}$ omitted}}]}$.
	
	First, we claim that if $q'_{j_{_\ell}}$ belongs to $\{q'_1,\ldots,q'_{t'-1}\}$, then $h\big(j'_{_{\ell}}\big)$ is one of the minors of size $t$ in $c_{\,;b}(\mathbf{Z}^{(v)}_{\alpha \times b})$.
	To see this, we will first rewrite the minors of size $t$ in $c_{\,;b}(\mathbf{Z}^{(v)}_{\alpha \times b})$ as minors of size $k$ in $P$.
	An arbitrary minor of size $t$ in $c_{\,;b}(\mathbf{Z}^{(v)}_{\alpha \times b})$ is of the form $M_{[\mathbf{p}'';\mathbf{q}'']}$, where $\mathbf{p}''$ is some $t$-element subset of $[\alpha]=\{1,\ldots,\alpha\}$ and $\mathbf{q}''$ is some $t$-element subset of $[b-1]$.
	If we set $\mathcal{C}:=\{b+r\,|\, 1 \leq r \leq k,\, r \notin \mathbf{p}''\}$, then
	it follows from Lemma \ref{lem:MtoP} that $M_{[\mathbf{p}'';\mathbf{q}'']} = P_{[\,;\mathbf{q}'' \cup \mathcal{C}]}$.
	Since $\mathbf{p}''$ and $\mathbf{q}''$ are arbitrary, it follows that the set of minors of size $t$ in $c_{\,;b}(\mathbf{Z}^{(v)}_{\alpha \times b})$ is equal to the set of all determinants $P_{[\,;\mathbf{q}'' \cup \{b+r\,|\, 1 \leq r \leq k,\, r \notin \mathbf{p}''\}]}$ in $P$, where $\mathbf{p}''$ and $\mathbf{q}''$ are any $t$-element subsets  of $[\alpha]$ and $[b-1]$ respectively.
	Now, suppose $q'_{j_{_\ell}}$ belongs to $\{q'_1,\ldots,q'_{t'-1}\}$.
	Then $q'_{j_{_\ell}} \leq b-1$, since $1 \leq q'_1 < \cdots < q'_{t'-1} < b$.
	Therefore, in this case, the set $\{q_1,\ldots,q_{t-1},q'_{j_{_\ell}}\}$ is a $t$-element subset of $[b-1]$.
	In addition, by definition, the set $\{q_{t},q_{t+1}\ldots,q_{k-1}\}$ equals $\{b+r\,|\, 1 \leq r \leq k,\, r \notin \big(\mathbf{p} \cup \{a\}\big)\}$.
	Since $1 \leq p_1 < \cdots < p_{i-1} < a < p_{i+1} < \cdots < p_t \leq \alpha$, it follows that $\mathbf{p} \cup \{a\}$ is a $t$-element subsets of $[\alpha]$.
 	Hence, the determinant 	$h\big(j'_{_{\ell}}\big)= \pm P_{[\,;q_1,\ldots,q_{t-1},q'_{j_{_\ell}},q_{t},q_{t+1} ,\ldots,q_{k-1}]}$ is one of the minors of size $t$ in $c_{\,;b}(\mathbf{Z}^{(v)}_{\alpha \times b})$.
	
	Conversely, we claim that if $q'_{j_{_\ell}}$ belongs to $\{q'_{t'},\ldots,q'_{k-1}\}$, then $h'\big(j'_{_{\ell}}\big)$ is one of the minors of size $t'$ in $c_{a;\,}(\mathbf{Z}^{(v)}_{\alpha' \times b})$.
	Just like the previous paragraph, by Lemma \ref{lem:MtoP}, the set of minors of size $t'$ in $c_{a;\,}(\mathbf{Z}^{(v)}_{\alpha' \times b})$ is equal to the set of all determinants $P_{[\,;\mathbf{q}'' \cup \{b+r'\,|\, 1 \leq r' \leq k,\, r' \notin \mathbf{p}''\}]}$ in $P$, where $\mathbf{p}''$ and $\mathbf{q}''$ are any $t'$-element subsets  of $[\alpha'] \setminus \{a\}$ and $[b]$ respectively.
	Suppose $q'_{j_{_\ell}}$ belongs to $\{q'_{t'},\ldots,q'_{k-1}\}$ and consider the set $\{b+a,q'_{t'},\ldots,\widehat{q'}_{j_{_\ell}}, \ldots,q'_{k-1}\}$ with $k-t'$ elements, where for some $j_{\ell}$, $\widehat{q'}_{j_{_\ell}}$ means ${q'}_{j_{_\ell}}$ is omitted.
	Corresponding to this set by translation is the set $\mathcal{R}\!:=\!\{a,q'_{t'}\!-\!b,\ldots,\widehat{q'}_{j_{_\ell}}, \ldots,q'_{k-1}\!-b\}$ which is a subset of $[k]$, by definition.
	Consequently, the set $[k]\!\setminus \!\mathcal{R}$ is a $t'$-element subset of $[\alpha']\! \setminus\! \{a\}$, since $k\!=\!\alpha'$.\!
	In addition, the set $\{q'_1,\ldots,q'_{t'-1},b\}$ is a $t'$-element subset of $[b]$, since $1\! \leq\! q'_1\! <\! \cdots\! <\! q'_{t'-1} \!<\! b$.
	Hence, the determinant 	$h'\big(j'_{_{\ell}}\big)\!=\!  P_{[\,;q'_1,\ldots,q'_{t'-1},b,b+a,q'_{t'},\ldots,\widehat{q'}_{j_{_\ell}}, \ldots,q'_{k-1}]}$ is one of the minors of size $t'$ in $c_{a;\,}(\mathbf{Z}^{(v)}_{\alpha' \times b})$.
\end{proof}
\end{lem}

\begin{ex}\label{ex:dif-belongs-to-N}
Continuing with Example \ref{ex:N-portion}, let $\alpha = 2$ and $\alpha' = 4$, so that $t:=1+\text{rank}(w_{\alpha \times b}) = 2$ and $t':=1+\text{rank}(w_{\alpha' \times b}) = 3$, where $b=4$ is the last descent of $v$.
	Set $a:=v(b+1) = 1$, $f_1:=M_{[p_1;q_1]}=M_{[2;1]}$, $f_2:=M_{[a,p_1;q_1,b]}=M_{[1,2;1,4]}$ $f'_1:= M_{[p'_1,p'_2;q'_1,q'_2]} = M_{[3,4;1,2]}$ and $f'_2:=M_{[a,p'_1,p'_2;q'_1,q'_2,b]}=M_{[1,3,4;1,2,4]}$, where $M = \mathbf{Z}^{(v)}$.
	We wish to write the difference  
	$f_1f'_2 - f_2f'_1$ as sum of some products of determinants.
	We proceed as follows: set $k := \alpha' = 4,$ $\mathbf{p}:=\{2\}$ and $\mathbf{p}':=\{3,4\}$.
	If we set
	\[\small P :=  \begin{pmatrix}
		{ z_{11}} & { z_{12}} & { z_{13}} & z_{14} & 1 & 0 & 0 & 0\\
		{ z_{21}} & z_{22} & { z_{23}} & z_{24} & 0 & 1 & 0 & 0\\
		z_{31} & z_{32} & z_{33} & 1 & 0 & 0 & 1 & 0\\
		z_{41} & z_{42} & 1 & 0 & 0 & 0 & 0 & 1
\end{pmatrix},\vspace{-0.10cm}\]
$\{q_{t},q_{t+1}\ldots,q_{k-1}\} = \{q_{2},q_{3}\} := \{b+r\,|\, 1 \leq r \leq k,\, r \notin \big(\mathbf{p} \cup \{a\}\big)\} = \{7,8\}$ and $\{q'_{t'},q'_{t'+1}, \linebreak[4]\ldots,q'_{k-1}\} = \{q'_{3}\} := \{b+r'\,|\, 1 \leq r' \leq k,\, r' \notin \big(\mathbf{p'} \cup \{a\}\big) \} = \{6\}$, then we have the following:
	\[
	\begin{array}{l}
		f_1 =  M_{[2;1]}  =  P_{[\,;q_1,b+a,q_2,q_{3}]} = P_{[\,;1,5,7,8]},\qquad 	f_2= M_{[1,2;1,4]} = P_{[\,;q_1,b,q_2,q_{3}]} = P_{[\,;1,4,7,8]},\\
		f'_1= M_{[3,4;1,2]} = P_{[\,;q'_1,q'_{2},b+a,q'_{3}]} = P_{[\,;1,2,5,6]}, \quad 
		f'_2= M_{[1,3,4;1,2,4]} = P_{[\,;q'_1,q'_{2},b,q'_{3}]} = P_{[\,;1,2,4,6]}.
	\end{array}
	\]
		\begin{equation*}
		\sum_{{\ell}=1}^{4+1} (-1)^{\ell} P_{[\,;1,7,8,q'_{j_{_\ell}}]}  \,P_{[\,;\underbrace{1,2,4,5,6}_{\text{$q'_{j_{_\ell}}$ omitted}}]} = 0, \qquad \text{$\text{where $q'_{j_{_\ell}}$} \in \{1,2,4,5,6\}$}, 
	\end{equation*}
i.e., 
\[
\begin{array}{c}
0 = - P_{[\,;1,7,8,1]}  \,P_{[\,;{2,4,5,6}]} +  P_{[\,;1,7,8,2]}  \,P_{[\,;{1,4,5,6}]} - P_{[\,;1,7,8,4]}  \,P_{[\,;{1,2,5,6}]}\\ + \linebreak[4]P_{[\,;1,7,8,5]}  \,P_{[\,;{1,2,4,6}]}  - P_{[\,;1,7,8,6]}  \,P_{[\,;{1,2,4,5}]}
\end{array}
\]
i.e.,
\(\qquad
\begin{array}{c}
P_{[\,;1,5,7,8]}  \,P_{[\,;{1,2,4,6}]} - P_{[\,;1,4,7,8]}  \,P_{[\,;{1,2,5,6}]} =  -  P_{[\,;1,2,7,8]}  \,P_{[\,;{1,4,5,6}]}  + P_{[\,;1,6,7,8]}  \,P_{[\,;{1,2,4,5}]}.
\end{array}
\)	

By the latter arguments in the proof of Lemma \ref{lem:dif-belongs-to-N}, we expect $P_{[\,;1,7,8,1]}$ and $P_{[\,;1,7,8,2]}$ to be one of the minors of size $t = 2$ in $c_{\,;b}(\mathbf{Z}^{(v)}_{\alpha \times b})$; here, $q'_{j_{_\ell}} = 1,2 \leq b-1$.
While the first determinant is zero, the second determinant is equal to $M_{[1,2;1,2]},$ which is one of the minors of size $t=2$ in $c_{\,;4}(\mathbf{Z}^{(v)}_{2 \times 4})$. 
Furthermore, we expect $P_{[\,;1,2,4,5]}$ to be one of the minors of size $t' = 3$ in $c_{a;\,}(\mathbf{Z}^{(v)}_{\alpha' \times b})$; here, $q'_{j_{_\ell}} = 6 > b+a$.
The determinant $P_{[\,;1,2,4,5]}$ is equal to $M_{[2,3,4;1,2,4]}$, which is one of the minors of size $t'=3$ in $c_{1;\,}(\mathbf{Z}^{(v)}_{4 \times 4})$. 
\end{ex}

The following result aids the proof of Lemma \ref{lem:kazdan-to-schubert2} in this paper.

\begin{lem}\label{lem:N-portion}
	Let $v, w \in S_n$ for which the last descent $b$ of $v$ is a descent of $w$.
	If $(\alpha,b)$ is an essential box for which $\alpha \geq a$, where $a:=v(b+1)$, 
	then $N_{(\alpha,b)} \subseteq T_{vs_b,ws_b}$.
	\begin{proof}
		Set $a:=v(b+1)$ and let $(\alpha,b)$, $\alpha \geq a$, be an essential box in column $b$, on or below row $a$ of $D(w)$.
		Set $M:=\mathbf{Z}^{(v)}_{\alpha \times b}$ and $t:=1+\text{rank}(w_{\alpha \times b})$.
		The submatrix $c_{\,;b}(M)$ of $M$ equals $\mathbf{Z}^{(v)}_{\alpha \times b-1} = \big(\mathbf{Z}^{(v)}s_b\big)_{\alpha \times b-1}$.
		If there is a box immediately to the left of essential box $(\alpha,b)$ in $D(w)$, then we have that $\text{rank}(w_{\alpha \times b-1}) = \text{rank}(w'_{\alpha \times b-1})$, where $w':=ws_b$.
		Hence, the set of minors of size $t$ in $c_{\,;b}(M) = \mathbf{Z}^{(v)}_{\alpha \times b-1}$ is equal to the set of minors of size $1+\text{rank}(w'_{\alpha \times b})$ in $\big(\mathbf{Z}^{(v)}s_b\big)_{\alpha \times b-1}$, and so these minors belong to $T_{vs_b,ws_b}$.
		Furthermore, if there is an essential box at location $(\alpha,b+1)$ in $D(w')$, then we have that $\text{rank}(w'_{\alpha \times b+1}) = t$.
		Set $M' := \big(\mathbf{Z}^{(v)}s_b\big)_{\alpha \times (b+1)}$ and $t':=1+\text{rank}(w'_{\alpha \times b+1})=1+t$.
		The matrix $M'$ has $a$ 1s in its submatrix formed rows $1,\ldots,a$ and columns $v^{-1}(1),\ldots,v^{-1}(b-1),b$.
		Therefore, by parts (\ref{it:cofactor-exp1}) and (\ref{it:cofactor-exp2}) of Lemma \ref{lem:cofactor-exp}, the ideal generated by minors of size $t'$ in $M'$ is equal to the ideal generated by 
		minors of size $t'-a$ in	
		$G':=c_{1,\ldots,a;v^{-1}(1),\ldots,v^{-1}(a-1),b}(M')$ and each of these minors of size $t'-a$ in $G'$ can actually be realized as a minor of size $t'$ in $M'$, and hence, belongs to $T_{vs_b,ws_b}$.
		Similarly, by cofactor expansion about rows $1,\ldots,(a-1)$ of $M$ and the columns where 1s are located in these rows, we have that 
		the ideal generated by minors of size $t$ in $M$ is equal to the ideal generated by minors of size $t-(a-1)$ in $c_{1,\ldots,a-1;v^{-1}(1),\ldots,v^{-1}(a-1)}(M)$.
		Consequently, we have that the ideal generated by the minors of size $t$ in $c_{a;\,}(M)$ is equal to the ideal generated by the minors of size $t-(a-1)$ in $G:=c_{a;\,}\big(c_{1,\ldots,a-1;v^{-1}(1),\ldots,v^{-1}(a-1)}(M)\big)=c_{1,\ldots,a;v^{-1}(1),\ldots,v^{-1}(a-1)}(M)$, and each of these minors of size $t-(a-1)$ in $G$ can be realized as minors of size $t$ in $c_{a;\,}(M)$.
		Since $M'$ differs from $M$ by its column $b$; precisely, $M$ is the resulting matrix from deleting column $b$ of $M'$, it follows that $G' = G$.
		We also have that $t-(a-1) = t'-1 - (a-1) = t'-a$.
		Hence, each of the minors of size $t-(a-1) = t'-a$ in $G=G'=c_{1,\ldots,a;v^{-1}(1),\ldots,v^{-1}(a-1),b}(M')$ can be realized as minors of size $t'$ in $M'$, and hence, belongs to $T_{vs_b,ws_b}$.
	\end{proof}
\end{lem}

\begin{ex}\label{ex:N-portion}
	\!Let $v \!=\! 654312$ and $w\! = \!136524$.\! The last descent $b\!=\!4$ of $v$ is a descent of $w$.
	\[
	\stackrel{ 
		{\small \begin{pmatrix}
				{ z_{11}} & { z_{12}} & { z_{13}} & z_{14} & 1 & { z_{16}}\\
				{ z_{21}} & z_{22} & { z_{23}} & z_{24} & 0 & 1\\
				z_{31} & z_{32} & z_{33} & 1 & 0 & 0\\
				z_{41} & z_{42} & 1 & 0 & 0 & 0\\
				z_{51} & 1 & 0 & 0 & 0 & 0\\
				1 & 0 & 0 & 0 & 0 & 0
		\end{pmatrix}},}{\mathbf{Z}^{(v)}}  \,\,\,
	\stackrel{ 
		{\small \begin{pmatrix}
				{ z_{11}} & { z_{12}} & { z_{13}} & 1 & z_{14}  & { z_{16}}\\
				{ z_{21}} & z_{22} & { z_{23}} & 0 & z_{24} & 1\\
				z_{31} & z_{32} & z_{33} & 0 & 1 & 0\\
				z_{41} & z_{42} & 1 & 0 & 0 & 0\\
				z_{51} & 1 & 0 & 0 & 0 & 0\\
				1 & 0 & 0 & 0 & 0 & 0
		\end{pmatrix}},}{\mathbf{Z}^{(v)}s_b}  \,\,\,
	\stackrel{ 
		\vcenter{\hbox{
				\begin{tikzpicture}[scale=.45]
					\draw (0,0) rectangle (6,6);
					
					\draw (1,4) rectangle (4,5);
					\draw[line width = .05ex] (2,4) -- (2,5);
					\draw[line width = .05ex] (3,4) --(3,5);

					\draw (2,1) -- (2,3) -- (4,3) -- (4,2) -- (3,2) -- (3,1) -- (2,1);
					\draw[line width = .05ex] (3,2) --(3,3);
					\draw[line width = .05ex] (2,2) --(3,2);
					
					\filldraw (0.5,5.5) circle (.5ex);
					\draw[line width = .2ex] (0.5,0) -- (0.5,5.5) -- (6,5.5);
					\filldraw (1.5,3.5) circle (.5ex);
					\draw[line width = .2ex] (1.5,0) -- (1.5,3.5) -- (6,3.5);
					\filldraw (2.5,0.5) circle (.5ex);
					\draw[line width = .2ex] (2.5,0) -- (2.5,0.5) -- (6,0.5);
					\filldraw (3.5,1.5) circle (.5ex);
					\draw[line width = .2ex] (3.5,0) -- (3.5,1.5) -- (6,1.5);
					\filldraw (4.5,4.5) circle (.5ex);
					\draw[line width = .2ex] (4.5,0) -- (4.5,4.5) -- (6,4.5);
					\filldraw (5.5,2.5) circle (.5ex);
					\draw[line width = .2ex] (5.5,0) -- (5.5,2.5) -- (6,2.5);
		\end{tikzpicture}}},}{D(w)}
	\,\,\, 
	\stackrel{
		\vcenter{\hbox{
				\begin{tikzpicture}[scale=.45]
					\draw (0,0) rectangle (6,6);
					
					\draw (1,4) rectangle (3,5);
					\draw[line width = .05ex] (2,4) -- (2,5);

					\draw (2,1) rectangle (3,3);
					\draw[line width = .05ex] (2,2) -- (3,2);
					
					\draw (4,2) rectangle (5,3);
					
					\filldraw (0.5,5.5) circle (.5ex);
					\draw[line width = .2ex] (0.5,0) -- (0.5,5.5) -- (6,5.5);
					\filldraw (1.5,3.5) circle (.5ex);
					\draw[line width = .2ex] (1.5,0) -- (1.5,3.5) -- (6,3.5);
					\filldraw (2.5,0.5) circle (.5ex);
					\draw[line width = .2ex] (2.5,0) -- (2.5,0.5) -- (6,0.5);
					\filldraw (3.5,4.5) circle (.5ex);
					\draw[line width = .2ex] (3.5,0) -- (3.5,4.5) -- (6,4.5);
					\filldraw (4.5,1.5) circle (.5ex);
					\draw[line width = .2ex] (4.5,0) -- (4.5,1.5) -- (6,1.5);
					\filldraw (5.5,2.5) circle (.5ex);
					\draw[line width = .2ex] (5.5,0) -- (5.5,2.5) -- (6,2.5);
		\end{tikzpicture}}}.}{D(ws_b)}
	\]
	Set $a:=v(b+1)=1$, the row at which the variable $z_{\text{max}} = z_{14}$ is located in $\mathbf{Z}^{(v)}$.
	For a box $(\alpha,\beta)$ in $D(ws_b)$, let $N'_{(\alpha,\beta)}$ denote the ideal generated by minors of size $1+\text{rank}((ws_b)_{\alpha \times \beta})$ in $\big(\mathbf{Z}^{(v)}s_b\big)_{\alpha \times \beta}$.
	For the essential box $(2,4)$ in $D(w)$, while the minors of size $t:=1+\text{rank}(w_{2 \times 4}) = 2$ in $c_{a;\,}\big(\mathbf{Z}^{(v)}_{2 \times 4}\big)$ equals zero, the minors of size $t$ in $c_{\,;b}\big(\mathbf{Z}^{(v)}_{2 \times 4}\big)$ equals $N'_{(2,3)} \subseteq T_{vs_b,ws_b}$. 
	Furthermore, for the essential box $(4,4)$ in $D(w)$, the minors of size $t':=1+\text{rank}(w_{4 \times 4}) = 3$ in $c_{a;\,}\big(\mathbf{Z}^{(v)}_{4 \times 4}\big)$ equals $N'_{(4,5)} \subseteq T_{vs_b,ws_b}$ and the minors of size $t'$ in $c_{\,;b}\big(\mathbf{Z}^{(v)}_{4 \times 4}\big)$ equals to $N'_{(4,3)} \subseteq T_{vs_b,ws_b}$. 
	Hence, $N_{(2,4)} + N_{(4,4)}$ is contained in $T_{vs_b,ws_b}$.
	\qed
\end{ex}

\begin{lem}\label{lem:kazdan-to-schubert2}
	Let $v, w \in S_n$ for which the last descent of $v$ is also a descent of $w$, and consider the positive multigrading of $R := \mathbb{K}[\mathbf{z}^{(v)}]$ by $\mathbb{Z}^n$.	
	Let $b$ be the last descent of $v$ and assume the variable $z_{\text{max}}$ does not belong to $\langle \text{in}_{\succ}(\mathcal{G}_I) \rangle$, where $\text{in}_{\succ}(\mathcal{G}_I)$ is the set of initial terms of essential minors generating the ideal $I:=Q_{v,w}$.
	Set $J:=T_{vs_b,w}$, $N:=T_{vs_b,ws_b}$, $A:=\langle \text{in}_{\succ}(\mathcal{G}_I) \rangle$, $B:=\langle \text{in}_{\succ}(\mathcal{G}_J) \rangle$ and $C:=\langle \text{in}_{\succ}(\mathcal{G}_N) \rangle$.
	Then there exists  $\boldsymbol{e} \in \mathbb{Z}^n$ such that there is an $R/N$-module isomorphism $I/N \cong (J/N)(-\boldsymbol{e})$ and an $R/C$-module isomorphism $A/C \cong (B/C)(-\boldsymbol{e})$.
\end{lem}

\begin{proof}
	Set $y:=z_{\text{max}}$ and let $\boldsymbol{e} \in \mathbb{Z}^n$ be the degree of $y$.
	First, we will show that $I/N \cong (J/N)(- \boldsymbol{e})$, as $R/N$-modules.
		Write $I$ in the form \(\left\langle yg_1 + r_1, \ldots, yg_k + r_k,h_1,\ldots,h_\ell\right \rangle\),  where 
		the set $\{yg_1 + r_1, \ldots, yg_k + r_k,h_1,\ldots,h_\ell \}$ is a complete list of all essential minors in $I$ and $y$ does not divide any term of $g_i$ or $r_i$ for any $1 \leq i \leq k$ nor any $h_j$ for $1 \leq j \leq \ell$, then it follows from Lemmas \ref{lem:IvwToIvsbw} and \ref{lem:IvwToIvsbwsb} that \(J = \left\langle g_1, \ldots, g_k,h_1,\ldots,h_\ell\right \rangle\) and \(N = \left\langle h_1,\ldots,h_\ell\right \rangle\).	
	Consequently, the mapping $I/N \longrightarrow (J/N)(-\boldsymbol{e})$ defined by 
	$\overline{f} \longmapsto \frac{g}{yg + r}\cdot \overline{f}$, for some essential minor $yg + r$ in $I$, is an $(R/N)$-module isomorphism.
	Indeed, if $yg' + r'$ is an arbitrary essential minor in $I$ that involves $y$, then $\frac{g}{yg + r} \cdot (yg' + r') = g'$, since $g(yg'+r') - (yg+r)g'$ belongs to $N$, by Lemmas \ref{lem:dif-belongs-to-N} and \ref{lem:N-portion}.
	Since $N$ is prime, it follows that neither $g$ nor $yg + r$ is a zero-divisor in $R/N$.
	
		Next, we will show that $A/C \cong (B/C)(- \boldsymbol{e})$, as $R/C$-modules.
	Observe that 
	\begin{eqnarray*}
		\langle \text{in}_{\succ}(\mathcal{G}_N) \rangle + y \cdot \langle \text{in}_{\succ}(\mathcal{G}_J) \rangle & = & \langle \text{in}_{\succ}(h_1), \ldots, \text{in}_{\succ}(h_\ell) \rangle + \langle y\, \text{in}_{\succ}(g_1), \ldots, y\, \text{in}_{\succ}(g_k),  y\, \text{in}_{\succ}(h_1), \ldots,\\
		& & \qquad\qquad\qquad\qquad\qquad\qquad\qquad\quad\qquad\qquad\qquad\qquad y\, \text{in}_{\succ}(h_\ell) \rangle\\
		& = & \langle \text{in}_{\succ}(h_1), \ldots, \text{in}_{\succ}(h_\ell) \rangle + \langle y\, \text{in}_{\succ}(g_1), \ldots, y\, \text{in}_{\succ}(g_k) \rangle + \\
		& & 				\qquad\qquad\qquad\qquad\qquad\qquad\qquad\!\!\quad\qquad \langle y\, \text{in}_{\succ}(h_1), \ldots, y\, \text{in}_{\succ}(h_\ell) \rangle\\
		& = & \langle \text{in}_{\succ}(h_1), \ldots, \text{in}_{\succ}(h_\ell) \rangle + \langle y\, \text{in}_{\succ}(g_1), \ldots, y\, \text{in}_{\succ}(g_k) \rangle\\
		& = & \langle \text{in}_{\succ}(h_1), \ldots, \text{in}_{\succ}(h_\ell), y\, \text{in}_{\succ}(g_1), \ldots, y\, \text{in}_{\succ}(g_k) \rangle\\
		& = & \langle \text{in}_{\succ}(\mathcal{G}_I) \rangle,
	\end{eqnarray*}
	i.e., $C + y \cdot B = A$. We claim that
	\[A/C = (y \cdot B + C)/C \cong (B/C)(- \boldsymbol{e})\]
	is a graded $R/C$-module isomorphism.
	To verify this claim, observe that the generators of $C$ do not involve $y$ and the map $B \rightarrow (y \cdot B + C)/C$ defined by $b \mapsto yb + C$ is a surjective homomorphism with kernel $C$. 
	For the kernel, observe that $b \mapsto C$ if and only if $yb + C = C$ if and only if $yb \in C$ if and only if $b \in C$.
\end{proof}

\begin{ex}\label{ex:descent-ascent}
	Let $v = 34512$ (as in Example \ref{ex:coordinate}) and $w=12354$. The last descent $b=3$ of $v$ is an ascent of $w$, and so we set both $J$ and $N$ to be $T_{vs_3,{w}}$. We have
	\[\mathbf{Z}^{(v)}s_3 = 
	\begin{pmatrix}
		{ z_{11}} & { z_{12}} & 1 & { z_{13}} & { z_{15}}\\
		{ z_{21}} & { z_{22}} & 0 & { z_{23}} & 1\\
		1 & { z_{32}} & 0 & {z_{33}} & 0\\
		0 & 1 & 0 & {z_{43}} & 0\\
		0 & 0 & 0 & 1 & 0
	\end{pmatrix} \quad \text{and} \quad 	D(w) =  \vcenter{\hbox{
			\begin{tikzpicture}[scale=.5]
				\draw (0,0) rectangle (5,5);
				
				\draw (3,1) rectangle (4,2);

				\filldraw (0.5,4.5) circle (.5ex); \draw[line width = .2ex] (0.5,0) --(0.5,4.5) --(5,4.5);
				\filldraw (1.5,3.5) circle (.5ex); \draw[line width = .2ex] (1.5,0) --(1.5,3.5) --(5,3.5);
				\filldraw (2.5,2.5) circle (.5ex); \draw[line width = .2ex] (2.5,0) --(2.5,2.5) --(5,2.5);
				\filldraw (3.5,0.5) circle (.5ex); \draw[line width = .2ex] (3.5,0) --(3.5,0.5) --(5,0.5);
				\filldraw (4.5,1.5) circle (.5ex); \draw[line width = .2ex] (4.5,0) --(4.5,1.5) --(5,1.5);
	\end{tikzpicture}}}
	.\]
	
	Here, $T_{vs_3,{w}} = Q_{v,w}$, and so if we set $I := Q_{v,{w}}$, then $I/N \cong J/N$.
	\qed
\end{ex}

\begin{ex}\label{ex:descent-descent}
	Let $v = 34512$ (as in Example \ref{ex:coordinate}) and $w=14325$.
	The last descent of $v$ is $b=3$ which is a descent of $w$.
	$\mathbf{Z}^{(v)}s_3$ is given in Example \ref{ex:descent-ascent}. 
	Here, we have
	\[	D(w) =  \vcenter{\hbox{
			\begin{tikzpicture}[scale=.5]
				\draw (0,0) rectangle (5,5);
				
				\draw (1,2) --(1,4) -- (3,4) -- (3,3) -- (2,3) --(2,2) -- (1,2);
				\draw (1,3) --(2,3);
				\draw (2,3) --(2,4);
				
							\filldraw (0.5,4.5) circle (.5ex); \draw[line width = .2ex] (0.5,0) --(0.5,4.5) --(5,4.5);
				\filldraw (1.5,1.5) circle (.5ex); \draw[line width = .2ex] (1.5,0) --(1.5,1.5) --(5,1.5);
				\filldraw (2.5,2.5) circle (.5ex); \draw[line width = .2ex] (2.5,0) --(2.5,2.5) --(5,2.5);
				\filldraw (3.5,3.5) circle (.5ex); \draw[line width = .2ex] (3.5,0) --(3.5,3.5) --(5,3.5);
				\filldraw (4.5,0.5) circle (.5ex); \draw[line width = .2ex] (4.5,0) --(4.5,0.5) --(5,0.5);
	\end{tikzpicture}}} \quad \text{and} \quad 
	D(ws_3) =  \vcenter{\hbox{
			\begin{tikzpicture}[scale=.5]
				\draw (0,0) rectangle (5,5);

				\draw (1,2) rectangle (2,4);
				\draw (1,3) --(2,3);

				\filldraw (0.5,4.5) circle (.5ex); \draw[line width = .2ex] (0.5,0) --(0.5,4.5) --(5,4.5);
				\filldraw (1.5,1.5) circle (.5ex); \draw[line width = .2ex] (1.5,0) --(1.5,1.5) --(5,1.5);
				\filldraw (3.5,2.5) circle (.5ex); \draw[line width = .2ex] (3.5,0) --(3.5,2.5) --(5,2.5);
				\filldraw (2.5,3.5) circle (.5ex); \draw[line width = .2ex] (2.5,0) --(2.5,3.5) --(5,3.5);
				\filldraw (4.5,0.5) circle (.5ex); \draw[line width = .2ex] (4.5,0) --(4.5,0.5) --(5,0.5);
	\end{tikzpicture}}}
	.\]
	\[\text{Set}\,\,J:=T_{vs_3,w} = \left \langle 
	z_{21},\,
	z_{22},\,
	\begin{vmatrix}
		z_{11} & z_{12}\\
		1 & z_{32}
	\end{vmatrix}
	\right \rangle \, \text{and} \,  N:=T_{vs_3,ws_3} = \left \langle 
	\begin{vmatrix}
		z_{11} & z_{12}\\
		z_{21} & z_{22}
	\end{vmatrix},
	\begin{vmatrix}
		z_{11} & z_{12}\\
		1 & z_{32}
	\end{vmatrix},
	\begin{vmatrix}
		z_{21} & z_{22}\\
		1 & z_{32}
	\end{vmatrix}
	\right \rangle.\]
	If we set \(I:= Q_{v,w} = \left \langle 
	\begin{vmatrix}
		z_{11} & z_{12}\\
		z_{21} & z_{22}
	\end{vmatrix},
	\begin{vmatrix}
		z_{11} & z_{13}\\
		z_{21} & z_{23}
	\end{vmatrix},
	\begin{vmatrix}
		z_{12} & z_{13}\\
		z_{22} & z_{23}
	\end{vmatrix},
	\begin{vmatrix}
		z_{11} & z_{12}\\
		1 & z_{32}
	\end{vmatrix},
	\begin{vmatrix}
		z_{21} & z_{22}\\
		1 & z_{32}
	\end{vmatrix}
	\right \rangle,\)
	then  $I/N \cong J/N$ via the map $\overline{g} \mapsto \frac{f_1}{f_2} \cdot \overline{g}$, where, for instance, $f_1 = z_{21}$ and $f_2 = \begin{vmatrix}
		z_{11} & z_{13}\\
		z_{21} & z_{23}
	\end{vmatrix}$.
	\qed
\end{ex}

	The following result is a fact that aids the proof of the main result of this paper.

\begin{lem}\label{lem:base-case}
	Let $v \in S_n$ for which $\ell(v) = 1$.
	For any $w$, if $Q_{v,w}$ is a proper ideal, then $v$ equals $w$ and $Q_{v,w} = \langle z_{bb} \rangle$, where $b$ is the last (and only) descent of $v$.
	\begin{proof}
		Since $\ell(v) = 1$, it follows that $v$ has only one descent, say $b$, and $v$ is the simple transposition $s_b$.
		In this case, the variable $z_{\max}$ is $z_{bb}$.
		Fix a $w \in S_n$ and consider all possible locations of essential boxes relative to the $b$\text{th} row and $b$\text{th} column of $D(w)$.
		See below for an illustration; all possible cases are shown in one diagram.
		\vspace{-0.25cm}
		\[
		 \vcenter{\hbox{
				\begin{tikzpicture}[scale=.35]
					\node at (-1.5,3.5) {$b$};
					\node at (2.5,7.5) {$b$};

					\draw (-1,0) -- (6,0) -- (6,7) -- (-1,7) -- (-1,0);

						\node at (2.5,3.5) { $\diamond$};
					\draw (2,3) -- (2,4) -- (3,4) -- (3,3) -- (2,3);
					
					\node at (0.5,5.5) {\scriptsize $\dagger$};
					\draw (0,5) -- (0,6) -- (1,6) -- (1,5) -- (0,5);
					
					\node at (4.5,5.5) {\tiny $\unrhd$};
					\draw (4,5) rectangle (5,6);
					
					\node at (0.5,1.5) {\tiny $\unlhd$};
					\draw (0,1) rectangle (1,2);
					
						\node at (4.5,1.5) {\scriptsize $\ddagger$};
					\draw (4,1) -- (4,2) -- (5,2) -- (5,1) -- (4,1);
					
					\node at (2.5,1.5) {\scriptsize $\triangleleft$};
					\draw (2,1) -- (2,2) -- (3,2) -- (3,1) -- (2,1);
					
						\node at (4.5,3.5) {$\triangleright$};
					\draw (4,3) -- (4,4) -- (5,4) -- (5,3) -- (4,3);
				
		\end{tikzpicture}}} 
		\] 
			Suppose that there is an essential box at location $(i,j)$ in $D(w)$ where either $i < b$ and $j < b$ (see {\scriptsize $\dagger$}), or $i < b$ and $j \geq b$ (see {\tiny $\unrhd$}), or $i \geq b$ and $j < b$ (see {\tiny  $\unlhd$}), or $i = b$ and $j > b$ (see $\triangleright$), or $i > b$ and $j = b$ (see { $\triangleleft$}), or $i > b$ and $j > b$ (see {\scriptsize $\ddagger$}).
			For any of these pairs $(i,j)$, the corresponding  submatrix $\mathbf{Z}^{(v)}_{i \times j}$ of $\mathbf{Z}^{(v)}$ contains a square submatrix, of size $\min(i,j)$, whose main diagonal entries are all 1s, up to rearranging rows or columns. So $Q_{v,w} = \langle 1 \rangle$ in this case, since $1 + \text{rank}(w_{i \times j}) \leq \min(i,j)$.
			Suppose that there is an essential box at location $(b,b)$ in $D(w)$ (see { $\diamond$}).
			The first $b-1$ entries on the main diagonal of the submatrix $\mathbf{Z}^{(v)}_{b \times b}$ of $\mathbf{Z}^{(v)}$ are all equal to 1.
			So any ideal generated by minors of size $t \times t$ in $\mathbf{Z}^{(v)}_{b \times b}$, where $t \leq b-1$, equals $\langle 1 \rangle$.
			If $t = b$, then by cofactor expansion, minor of size $t \times t$ in $\mathbf{Z}^{(v)}_{b \times b}$ equals $z_{bb}$.
			Thus, $Q_{v,w}$ is a proper ideal provided $D(w)$ has a unique essential box at $(b,b)$ and $\text{rank}(w_{b\times b})  = b-1$. In this case, we have $t = 1+ \text{rank}(w_{b\times b}) = b$, $Q_{v,w} = \langle z_{bb} \rangle$ and $w = s_b = v$.
	\end{proof}
\end{lem}

In what follows, let $\succ_v$ be the term order that we use on $\mathbb{K}[\mathbf{z}^{(v)}]$ (as in Definition \ref{defnnotazmax}) and $\succ_{vs_b}$ be the term order on $\mathbb{K}[\mathbf{z}^{(vs_b)}]$, where $b$ is the last descent of $v$.
	The substitution map $\varphi$ defined in Definition \ref{defn:moving-from-T-Q} is not order-preserving, in the sense that, if $z_{ij} \succ_{vs_b} z_{i'j'}$ in $\mathbb{K}[\mathbf{z}^{(vs_b)}]$, then $\varphi(z_{ij}) \succ_{v} \varphi(z_{i'j'})$ does not necessarily hold in $\mathbb{K}[\mathbf{z}^{(v)}]$.
	For example, in $\mathbb{K}[\mathbf{z}^{(vs_b)}]$, where $v = 51423$ and $b = 3$, we have $z_{14} \succ_{vs_b} z_{24}$, but $\varphi(z_{14}) = z_{13} \nsucc_{v} z_{23} = \varphi(z_{24})$.
However,  restricting the term order $\succ_{vs_b}$ to those variables in $\mathbb{K}[\mathbf{z}^{(vs_b)}]$ that are different from the variables on or above (i.e., weakly above) row $v(b+1)$, and on columns $b$ and $b+1$ of $\mathbf{Z}^{(vs_b)}$, the map $\varphi$ preserves the term orders $\succ_{vs_b}$ and $\succ_{v}$.
Recall that the variable $z_{\text{max}}$ is at position $(v(b+1),b)$ in $\mathbf{Z}^{(v)}$.
We therefore have the following result which can be easily verified. It aids the proof of our main result (Theorem \ref{lem:gb-patch-multigrading2}).

\begin{lem}\label{lem:moving-from-T-Q}
	Let $v, w \in S_n$, $b$ be the last descent of $v$ and $\varphi$ be the substitution map defined in Definition \ref{defn:moving-from-T-Q}. Then we have the following:
	\begin{enumerate}
		\item\label{lem:moving-from-T-Q-pt1} $\text{in}_{\succ_v} \big(\varphi(f)\big) = \varphi \big( \text{in}_{\succ_{vs_b}}(f) \big)$, for any $f \in \mathbb{K}[\mathbf{z}^{(vs_b)}]$ that does not involve the variables $z_{ij} \in \mathbf{z}^{(vs_b)}$, where $1 \leq i \leq v(b+1)$ and $b \leq j \leq b+1$.
		\item 	The essential minors of $Q_{vs_b,w}$ (resp. $Q_{vs_b,ws_b}$) form a Gr\"{o}bner basis with respect to $\succ_{vs_b}$ if and only if the essential minors of $T_{vs_b,w}$ (resp. $T_{vs_b,ws_b}$) form a Gr\"{o}bner basis with respect to $\succ_v$.
	\end{enumerate}
\end{lem}

Given a simplicial complex $\triangle$ on vertex $[n]$, define the \textbf{link} of a face $\sigma \in \triangle$ to be
\[\text{link}_\triangle(\sigma) = \{\tau \in \triangle\,|\,\tau \cup \sigma \in \triangle \,\, \text{and} \,\, \tau \cap \sigma = \emptyset\}\]
and the \textbf{deletion} of $\sigma \in \triangle$ to be
\[\text{del}_\triangle(\sigma) = \{\tau \in \triangle\,|\,\tau \cap \sigma = \emptyset\}.\]
The simplicial complex $\triangle$ on the vertex set $[n]$ is \textbf{vertex decomposable} if $\triangle$ is pure (i.e., if dimension of the maximal faces of $\triangle$, with respect to inclusion, are all equal) and either (i) $\triangle$ is a simplex, or (ii) $\triangle = \emptyset$, or (iii) for some vertex $\sigma \in \triangle$, both $\text{del}_\triangle(\sigma)$ and $\text{link}_\triangle(\sigma)$ are vertex-decomposable, and
\(\dim(\triangle) = \dim(\text{del}_\triangle(\sigma)) = \dim(\text{link}_\triangle(\sigma)) + 1.\)

Below is the main result of this paper.
Before stating this result, we want to point out that only the Gr\"{o}bner basis part of this result is new.
Knutson showed in his paper \cite{knutson2008schubert} that Schubert patches degenerate to a Stanley-Reisner scheme whose underlying simplicial complex is a subword complex, which are vertex decomposable \cite[Theorem E]{knutson2005grobner}.

\begin{thm}\label{lem:gb-patch-multigrading2}
Let $v, w \in S_n$.
Under the term order $\succ_{v}$ on $\mathbb{K}[\mathbf{z}^{(v)}]$, as in Definition \ref{defnnotazmax}, the essential minors form a Gr\"{o}bner basis for $Q_{v,w} \subseteq \mathbb{K}[\mathbf{z}^{(v)}]$.
In addition, the initial ideal $\text{in}_{\succ_{v}}(Q_{v,w})$ of $Q_{v,w}$ with respect to $\succ_{v}$ is squarefree and the simplicial complex 
associated to $\text{in}_{\succ_{v}}(Q_{v,w})$ is vertex decomposable.
\begin{proof}
	Let $v \in S_n$ be fixed and $R := \mathbb{K}[\mathbf{z}^{(v)}]$.
	We proceed by induction on $\ell(v)$.
	If $\ell(v) = 0$, then $v = \text{id}$; the identity permutation, and so $Q_{v,w}$ is the unit  ideal.
	If $\ell(v) = 1$, then by Lemma \ref{lem:base-case}, $Q_{v,w}$ is generated by an indeterminate.
	
	Suppose the hypotheses are true for all Schubert patch ideals $Q_{v',w}$, $v',w \in S_n$ with $v' \leq v$ in Bruhat order.
	If $b$ is the last descent of $v$, then $vs_b \leq v$ in Bruhat order and $\ell(vs_b) = \ell(v) -1$.
	Set $I:=Q_{v,w}$, $J:=T_{vs_b,{w}}$ and $N:=T_{vs_b,{w}s_b}$.
	Under the positive grading of $R$ by $\mathbb{Z}^n$, the ideals $I$, $J$ and $N$ are homogeneous by Lemma \ref{lem:homo-multi}, Lemma \ref{lem:J-Homog} and Corollary \ref{cor:N-Homog} respectively.
	Let $\mathcal{G}_I$, $\mathcal{G}_J$, $\mathcal{G}_N$, $\mathcal{G}'_{Q_{vs_b,w}}$ and $\mathcal{G}'_{Q_{vs_b,ws_b}}$ be the sets of essential minors generating $I$, $J$, $N$, $Q_{vs_b,w}$ and $Q_{vs_b,ws_b}$ respectively.
	
	Suppose the variable $y:=z_{\text{max}}$ belongs to $I$.
	Then $J$ is the unit ideal and $I = N + \langle y \rangle$.
	By the induction hypothesis, $\mathcal{G}'_{Q_{vs_b,ws_b}}$ is a Gr\"{o}bner basis for the ideal $Q_{vs_b,ws_b}$ with respect to the term order $\succ_{vs_b}$.
	Consequently, by Lemma \ref{lem:moving-from-T-Q}, $\mathcal{G}_N$ is a Gr\"{o}bner basis for the ideal $N$ with respect to the term order $\succ_{v}$.
	Since essential minors in $\mathcal{G}_N$ do not involve $y$, it follows that the set $\mathcal{G}_I = \{y\} \cup \mathcal{G}_N$ of essential minors generating $I$ is a Gr\"{o}bner basis for $I$. 
	
	On the other hand, suppose the variable $y$ does not belong to $I$.
	Then $J$ is a proper ideal. 
	By the induction hypothesis, the sets $\mathcal{G}'_{Q_{vs_b,w}}$ and $\mathcal{G}'_{Q_{vs_b,ws_b}}$ are respectively Gr\"{o}bner bases for the ideals $Q_{vs_b,w}$ and $Q_{vs_b,ws_b}$ with respect to the term order $\succ_{vs_b}$.
	Consequently, by Lemma \ref{lem:moving-from-T-Q}, $\mathcal{G}_J$ and $\mathcal{G}_N$ are, respectively, Gr\"{o}bner bases for the ideals $J$ and $N$ with respect to the term order $\succ_{v}$.
	Therefore, if $B$ and $C$ are the ideals generated by the initial terms, with respect to $\succ_{v}$, of the elements of 
	$\mathcal{G}_J$ and $\mathcal{G}_N$ respectively, i.e. if \(B := \langle \text{in}_{\succ_{v}}(\mathcal{G}_J)\rangle\) and \(C := \langle \text{in}_{\succ_{v}}(\mathcal{G}_N)\rangle\), then $B = \text{in}_{\succ_{v}}(J)$ and $C = \text{in}_{\succ_{v}}(N)$.
	Furthermore, if \(A := \langle \text{in}_{\succ_{v}}(\mathcal{G}_I) \rangle\), then it follows from Lemma \ref{lem:kazdan-to-schubert2} that $I/N \cong (J/N)(- \boldsymbol{e})$ as $R/N$-modules and $A/C \cong (B/C)(- \boldsymbol{e})$ as $R/C$-modules, where $\boldsymbol{e} \in \mathbb{Z}^n$ is the degree of the variable $y$.
	Hence,  
	by Lemma \ref{suffCondGbasis2}, $A = \text{in}_{\succ_{v}}(I)$, i.e., the set  $\mathcal{G}_I$ of essential minors generating $I$ is a Gr\"{o}bner basis for $I$. 
	
	From the proof of Lemma \ref{lem:kazdan-to-schubert2}, we establish that $A = C + y \cdot B$.
	Therefore, $A$ is squarefree since both $B$ and $C$ are squarefree by the induction hypotheses.
	Lastly, let $\triangle_A$, $\triangle_B$ and $\triangle_C$ be the simplicial complexes associated to $A$, $B$ and $C$, respectively, and $\gamma \in \triangle_A$ be the vertex corresponding to the variable $y \in A$.
	Since $A = C + y \cdot B$, it follows that $\triangle_B = \text{link}_{\triangle_A}(\gamma)$ and $\triangle_C = \text{del}_{\triangle_A}(\gamma)$.
	Therefore, $\triangle_A$ is vertex decomposable, since both $\triangle_B$ and $\triangle_C$ are vertex decomposable by the induction hypotheses.
\end{proof}
\end{thm}

\subsection{Gr\"obner bases for Kazhdan-Lusztig ideals and Schubert determinantal ideals}

In this subsection, under the term order defined in Definition \ref{defnnota:xlast}, we give a new proof of the known result that essential minors form Gr\"{o}bner basis for Kazhdan-Lusztig ideals, and hence, for Schubert determinantal ideals. 
	
	\begin{lem}\cite[Lemma 6.13]{woo2012grobner}\label{lem:moving-v-vsi}
		Given $f \in \mathbb{K}[\mathbf{x}^{(vs_b)}]$, let $f'$ be obtained by the substitution $x_{j,b+1} \mapsto x_{j,b}$.
		Then $f' \in \mathbb{K}[\mathbf{x}^{(v)}]$.
	\end{lem}

Let $v \in S_n$, $b$ be the last descent of $v$ and 
$\mathbf{W}^{(vs_b)}$ be obtained from $\mathbf{X}^{(vs_b)}$ by the substitution $x_{i,b+1} \mapsto x_{i,b}$. 
	Below is an analogue of Definition \ref{defn:J-ideal} for Kazhdan-Lusztig ideals.
	
	\begin{defn}\label{defn:J-idealKL}
		Let $v, w \in S_n$ and $b$ be the last descent of $v$. 
		Define an ideal $L_{vs_b,w} \subseteq \mathbb{K}[\mathbf{x}^{(v)}]$ as follows:
		\[{L}_{vs_b,w} = \langle \text{minors of size $1+ \rank(w_{p \times q})$ in $\mathbf{W}^{(vs_b)}_{p \times q}$,\,\, $1 \leq p,q \leq n$} \rangle.\]
	\end{defn}

\begin{ex}
		In Example \ref{ex:diagramZvsbFromZvideal} (resp. Example \ref{ex:diagramZvsbFromZvideal2}), setting to zero in $T_{vs_3,w}$ (resp. $Q_{vs_3,w}$) the set of variables $\{z_{13},z_{15},z_{42}\}$ (resp. $\{z_{14},z_{15},z_{42}\}$) that are due east of the 1s in $\mathbf{Z}^{(v)}s_3$ (resp. $\mathbf{Z}^{(vs_3)}$) and relabeling remaining variables $z_{ij} \mapsto x_{ij}$, the resulting ideal is $L_{vs_3,w}$ (resp. $I_{vs_3,w}$).
		\qed
\end{ex}

The result below is an immediate consequence of Lemma \ref{lem:descent-ascent-equal-ideal}, after setting to zero the variables in $\mathbf{y}^{(v)} \subseteq \mathbf{z}^{(v)}$ defined at the beginning of Subsection \ref{subsec:patch}. Observe that while the ideal $I_{v,w}$ is contained in $\mathbb{K}[\mathbf{x}^{(v)}]$, the ideal $L_{vs_b,w}$ is contained in $\mathbb{K}[\mathbf{x}^{(v)}\backslash \{x_{\text{last}}\}]$.

	\begin{cor}\label{cor:descent-ascent-equal-ideal}
			Let $v,w \in S_n$ and $b$ be the last descent of $v$.
		If $b$ is an ascent of $w$, then $I_{v,w} = L_{vs_b,w}\mathbb{K}[\mathbf{x}^{(v)}]$.
	\end{cor}
	
Define a map 
\(
	\phi\,:\,\mathbb{K}[\mathbf{z}^{(v)}] \rightarrow \mathbb{K}[\mathbf{x}^{(v)}]
\)
by $z_{ij} \mapsto 0$, for each $z_{ij} \in \mathbf{y}^{(v)}$, and $z_{ij} \mapsto x_{ij}$, for other variables.
It follows that if $D$ is an essential minor of $T_{vs_b,w}$ (resp. $T_{vs_b,ws_b}$), then $\phi(D)$ is an essential minor of $L_{vs_b,w}$ (resp. $L_{vs_b,ws_b}$).
Consequently, we have that $L_{vs_b,ws_b} \subseteq  L_{vs_b,w}$, since $T_{vs_b,ws_b} \subseteq  T_{vs_b,w}$ by Corollary \ref{cor:N-contained-I-int-J}.
Similarly, if $D$ is an essential minor of $Q_{v,w}$, then $\phi(D)$ is an essential minor of $I_{v,w}$.
Consequently, we have that $L_{vs_b,ws_b} \subseteq  I_{v,w}$, since $T_{vs_b,ws_b} \subseteq  Q_{v,w}$ by Corollary \ref{cor:N-contained-I-int-J}.
	Hence, the following result follows.
	
	\begin{cor}\label{cor:stdHompartaK}
	Let $v,w \in S_n$ for which the last descent of $v$ is a descent of $w$.
	If $b$ is the last descent of $v$,
	then ${L}_{vs_b,ws_b} \subseteq {I}_{v,w} \cap {L}_{vs_b,w}$.
	\end{cor}
	
	By changing the underlying ring $\mathbb{K}[\mathbf{z}^{(v)}]$ for Schubert patch ideals in Lemma \ref{lem:kazdan-to-schubert2} and Theorem \ref{lem:gb-patch-multigrading2} to the corresponding ring $\mathbb{K}[\mathbf{x}^{(v)}]$ for Kazhdan-Lusztig ideals, we obtain proofs for the following results.
	

	\begin{cor}\label{cor:kazdan-to-schubert2}
		Let $v, w \in S_n$ for which the last descent $b$ of $v$ is also a descent of $w$, and consider the positive multigrading of $R := \mathbb{K}[\mathbf{x}^{(v)}]$ by $\mathbb{Z}^n$.	
		Assume the variable $x_{\text{last}}$ does not belong to $\langle \text{in}_{\succ}(\mathcal{G}_I) \rangle$, where $\text{in}_{\succ}(\mathcal{G}_I)$ is the set of initial terms of essential minors generating the ideal $I:=I_{v,w}$.
		Set $J:=L_{vs_b,w}$, $N:=L_{vs_b,ws_b}$,  $A:=\langle \text{in}_{\succ}(\mathcal{G}_I) \rangle$, $B:=\langle \text{in}_{\succ}(\mathcal{G}_J) \rangle$ and $C:=\langle \text{in}_{\succ}(\mathcal{G}_N) \rangle$.
		Then there exists  $\boldsymbol{e} \in \mathbb{Z}^n$ such that there is an $R/N$-module isomorphism $I/N \cong (J/N)(-\boldsymbol{e})$ and an $R/C$-module isomorphism $A/C \cong (B/C)(-\boldsymbol{e})$.
	
	\end{cor}

	\begin{proof}
Same as the proof of Lemma \ref{lem:kazdan-to-schubert2}, replacing mainly Corollary \ref{cor:N-contained-I-int-J} 
with Corollary \ref{cor:stdHompartaK}.
Note that an analogue of Lemma \ref{lem:IvwToIvsbw} for Kazhdan-Lusztig ideals is easily obtained by setting to zero the variables $\mathbf{y}^{(v)} \subseteq \mathbf{z}^{(v)}$.
\end{proof}

In what follows, let $\succ_v$ be the term order that we use on $\mathbb{K}[\mathbf{x}^{(v)}]$ in Definition \ref{defnnota:xlast} and $\succ_{vs_b}$ be the term order on $\mathbb{K}[\mathbf{x}^{(vs_b)}]$.
The following result is an analogue of Remark \ref{rem:trans-Q-T} and Lemma \ref{lem:moving-from-T-Q}
for Kazhdan-Lusztig ideals.
\begin{cor}\label{cor:moving-from-T-Q}
	Let $v, w \in S_n$ and $b$ be the last descent of $v$. Given $f \in \mathbb{K}[\mathbf{x}^{(vs_b)}]$, let $f'$ be as in Lemma \ref{lem:moving-v-vsi}. Then we have the following:
	\begin{enumerate}
		\item $\text{in}_{\succ_v} \big(f'\big) = \big( \text{in}_{\succ_{vs_b}}(f) \big)'$, for any $f \in \mathbb{K}[\mathbf{x}^{(vs_b)}]$.
		\item $f$ is a generator of $I_{vs_b,w}$ (resp. $I_{vs_b,ws_b}$) if and only if $f'$ is a generator of $L_{vs_b,w}$ (resp. $L_{vs_b,ws_b}$).
		\item 	The essential minors of $I_{vs_b,w}$ (resp. $I_{vs_b,ws_b}$) form a Gr\"{o}bner basis with respect to $\succ_{vs_b}$ if and only if the essential minors of $L_{vs_b,w}$ (resp. $L_{vs_b,ws_b}$) form a Gr\"{o}bner basis with respect to $\succ_v$.
	\end{enumerate}
\end{cor}

Recall that every Schubert determinantal ideal is a Kazhdan-Lusztig ideal.
The next result is the Gr\"{o}bner basis result for Kazhdan-Lusztig ideals {\cite[Theorem 2.1]{woo2012grobner}}, and hence, a Gr\"{o}bner basis result for Schubert determinantal ideals {\cite[{ Theorem B}]{knutson2005grobner}}.
		
	\begin{cor}\label{cor:gb-result-KL}
		Let $w \in S_n$ be an arbitrary permutation and $v \in S_n$ be fixed.
		Under the term order $\succ_v$, the essential minors form a Gr\"{o}bner basis for $I_{v,w} \subseteq \mathbb{K}[\mathbf{x}^{(v)}]$.
		In addition, the initial ideal $\text{in}_{\succ_{v}}(I_{v,w})$ of $I_{v,w}$ with respect to $\succ_{v}$ is squarefree and the simplicial complex 
		associated to $\text{in}_{\succ_{v}}(I_{v,w})$ is vertex decomposable.
	\end{cor}
\begin{proof}
		Same as the proof of Theorem \ref{lem:gb-patch-multigrading2}, replacing mainly Lemma \ref{lem:moving-from-T-Q}, Corollary \ref{cor:N-contained-I-int-J} and Lemma \ref{lem:kazdan-to-schubert2} by Corollary \ref{cor:moving-from-T-Q}, Corollary \ref{cor:stdHompartaK} and Corollary \ref{cor:kazdan-to-schubert2} respectively.
		Note that Kazhdan-Lusztig ideals are homogeneous under the positive grading of $\mathbb{K}[\mathbf{x}^{(v)}]$ by $\mathbb{Z}^n$ (see \cite[Lemma 5.2]{woo2008governing}).
	\end{proof}

\subsection{Further remarks on initial ideals and $K$-Polynomials}
Proposition \ref{prop:inKLEQinP} shows that most of the results given in the paper \cite{woo2012grobner} about the initial ideals of Kazhdan-Lusztig ideals are also true for the initial ideals of Schubert patch ideals.
In particular, the $K$-polynomials for these ideals agree.
Below is a result that partly aids the proof of Proposition \ref{prop:inKLEQinP} and its proof follows from the proof of Lemma \ref{lem:kazdan-to-schubert2}.

\begin{lem}\label{lem:KL-double-linkage}
		Let $v \in S_n$ and $b$ be the last descent of $v$. 
		Under the term order $\succ_v$ on $\mathbb{K}[\mathbf{x}^{(v)}]$, we have the equality 
		$\text{in}_{\succ_v}({I}_{v,w})  =  \text{in}_{\succ_v}({L}_{vs_b,ws_b})  + x_{\text{last}} \cdot  \text{in}_{\succ_v}({L}_{vs_b,w}) $.
\end{lem}

	\begin{prop}\label{prop:inKLEQinP}
	Let $v, w$ be permutations in $S_n$.
	Let $\tilde{I}_{v,w}$ be the ideal generated by the resulting minors from applying the substitution map $x_{ij} \mapsto z_{ij}$, for all $i,j$, to the essential minors of $I_{v,w}$. Then $\text{in}_{\succ_v}(Q_{v,w}) = \text{in}_{\succ_v}\big({\tilde{{I}}}_{v,w}\mathbb{K}[\mathbf{z}^{(v)}]\big)$.
	
	\begin{proof}
		Set $R:=\mathbb{K}[\mathbf{z}^{(v)}]$.
		We proceed by induction on $\ell(v)$.
		If $\ell(v) = 0$, then both ideals $Q_{v,w}$ and $\tilde{I}_{v,w}$ are unit ideal.
		If $\ell(v) = 1$ and $Q_{v,w} \neq \mathbb{K}[\mathbf{z}^{(v)}]$, then by Lemma \ref{lem:base-case}, $Q_{v,w}$ is generated by the variable $z_{bb}$.
		Recall from Section \ref{subsec:patch} that by setting the variables $\mathbf{y}^{(v)} \subseteq \mathbf{z}^{(v)}$ to zero in $Q_{v,w}$ and relabeling other variables from $z_{ij}$ to $x_{ij}$, we obtain $I_{v,w}$.
		For this particular $v$, where $\ell(v) = 1$ and $Q_{v,w} \neq \mathbb{K}[\mathbf{z}^{(v)}]$, the only variable in $\mathbf{z}^{(v)}$ that is not in  $\mathbf{y}^{(v)}$ is $z_{bb}$, where $b$ is the last (and only) descent of $v$.
		It therefore follows that by setting the variables in $\mathbf{y}^{(v)}$ to zero in $Q_{v,w} = \langle z_{bb} \rangle$, we obtain $I_{v,w} = \langle x_{bb} \rangle$, and so $\tilde{I}_{v,w} = \langle z_{bb} \rangle$. 
		Suppose the hypothesis is true for all ideals $Q_{v',w}$ and $\tilde{I}_{v',w}$, $v',w \in S_n$ with $v' \leq v$ in Bruhat order.
		Let $b$ be the last descent of $v$.
		Then $vs_b \leq v$ in Bruhat order and $\ell(vs_b) = \ell(v) -1$.
		By the induction hypothesis, we have $\text{in}_{\succ_{vs_b}}(Q_{vs_b,w}) = \text{in}_{\succ_{vs_b}}(\tilde{I}_{vs_b,w}R)$ and $\text{in}_{\succ_{vs_b}}(Q_{vs_b,ws_b})=\text{in}_{\succ_{vs_b}}(\tilde{I}_{vs_b,ws_b}R)$.
		Let $\tilde{L}_{vs_b,w}$ be the ideal generated by resulting minors from applying the substitution map $x_{ij} \mapsto z_{ij}$, for all $i,j$, to essential minors of $L_{vs_b,w}$.
		Then by Lemma \ref{lem:moving-from-T-Q} and Corollary \ref{cor:moving-from-T-Q}, we have the following equality of ideals:
		$\text{in}_{\succ_{vs_b}}(\hat{T}_{vs_b,w}) = \text{in}_{\succ_{vs_b}}(Q_{vs_b,w})$, $\text{in}_{\succ_{vs_b}}(\hat{\tilde{L}}_{vs_b,w}) = \text{in}_{\succ_{vs_b}}(\tilde{{I}}_{vs_b,w})$,		
		$\text{in}_{\succ_{vs_b}}(\hat{T}_{vs_b,ws_b}) = \text{in}_{\succ_{vs_b}}(Q_{vs_b,ws_b})$ and $\text{in}_{\succ_{vs_b}}(\hat{\tilde{L}}_{vs_b,ws_b}) = \text{in}_{\succ_{vs_b}}(\tilde{{I}}_{vs_b,ws_b})$,
		where the ideals $\hat{T}_{vs_b,w}$, $\hat{T}_{vs_b,ws_b}$, $\hat{\tilde{L}}_{vs_b,w}$ and $\hat{\tilde{L}}_{vs_b,ws_b}$ are the ideals generated by minors obtained from applying the substitution map $z_{i,b+1} \mapsto z_{i,b}$ and $z_{i,b} \mapsto z_{i,b+1}$, for all $i$, to (essential) minors of $T_{vs_b,w}$, $T_{vs_b,ws_b}$, $\tilde{L}_{vs_b,w}$ and $\tilde{L}_{vs_b,ws_b}$ respectively.
		Therefore, $\text{in}_{\succ_{vs_b}}(\hat{T}_{vs_b,w}) = \text{in}_{\succ_{vs_b}}(\hat{\tilde{L}}_{vs_b,w}R)$.
		Similarly, $\text{in}_{\succ_{vs_b}}(\hat{T}_{vs_b,ws_b}) = \text{in}_{\succ_{vs_b}}(\hat{\tilde{L}}_{vs_b,ws_b}R)$.
		Reversing the last substitution map, we have
		$\text{in}_{\succ_{v}}({T}_{vs_b,w}) = \text{in}_{\succ_{v}}({\tilde{L}}_{vs_b,w}R)$ and 
		$\text{in}_{\succ_{v}}({T}_{vs_b,ws_b}) = \text{in}_{\succ_{v}}({\tilde{L}}_{vs_b,ws_b}R)$, noting that 
		moving from the order $\succ_{vs_b}$ to $\succ_{v}$ will not pose a problem as the ideals involved can be generated by essential minors that do not involve the variables $z_{ij}$, for all $1 \leq i \leq v(b+1)$ and $b \leq j \leq b+1$ (see Lemma \ref{lem:moving-from-T-Q}).
		Consequently, $\text{in}_{\succ}(Q_{v,w}) = \text{in}_{\succ}(\tilde{I}_{v,w}R)$, since we have that 
		$ \text{in}_{\succ_v}(Q_{v,w})=  \text{in}_{\succ_v}(T_{vs_b,ws_b}) + z_{\text{max}} \cdot  \text{in}_{\succ_v}(T_{vs_b,w}) $
		from the last part of the proof of Lemma \ref{lem:kazdan-to-schubert2} and  $ \text{in}_{\succ_v}({\tilde{{I}}}_{v,w})  = \text{in}_{\succ_v}({\tilde{L}}_{vs_b,ws_b}) + z_{\text{max}} \cdot \text{in}_{\succ_v}({\tilde{L}}_{vs_b,w})$ from Corollary \ref{lem:KL-double-linkage}.
		Note that $z_{\text{max}} = z_{v(b+1),b}$ in $\mathbf{Z}^{(v)}$ and $x_{\text{last}}=x_{v(b+1),b}$ in $\mathbf{X}^{(v)}$, by Lemma \ref{lem:loczmax}.
	\end{proof}
\end{prop}

The Hilbert series of a finitely generated graded module $M:=R/I$ over a polynomial ring $R$ positively multigraded by $\mathbb{Z}^n$ can be uniquely expressed as the quotient function
\[\text{Hilb}(M;\boldsymbol{t}) = \displaystyle \sum_{\boldsymbol{e} \in \mathbb{Z}^{^n}} \dim_{\mathbb{K}} (M_{\boldsymbol{e}}) \boldsymbol{t}^{\boldsymbol{e}} =  \frac{\mathcal{K}(M;\boldsymbol{t})}{\prod_{i}(1-\boldsymbol{t}^{\boldsymbol{e}_i})}.\]
The numerator $\mathcal{K}(M;\boldsymbol{t})$ 
is called the \textbf{$\boldsymbol{K}$-polynomial} of $M$ (see \cite[Definition 8.21]{miller2004combinatorial}).

Let $v \in S_n$ and $b$ be the last descent of $v$.
For Kostant-Kumar recursion for Kazhdan-Lusztig ideals to be well defined, it is necessary that degrees of corresponding variables in $\mathbb{K}[\mathbf{x}^{(v)}]$ and $\mathbb{K}[\mathbf{x}^{(vs_b)}]$ are preserved. See the following example for what we mean by degree preservation.

\begin{ex}\label{ex:our-convention-not-uniform-multidegree}
	Let $v = 34512$. The last descent of $v$ is $b=3$, so that $vs_3 = 34152$. Then	
	\[
	{ 
		\mathbf{X}^{(v)} = \begin{pmatrix}
				{ x_{11}} & { x_{12}} & { x_{13}} & 1 & 0\\
			{ x_{21}} & { x_{22}} & { x_{23}} & 0 & 1\\
			1 & 0 & 0 & 0 & 0\\
			0 & 1 & 0 & 0 & 0\\
			0 & 0 & 1 & 0 & 0
		\end{pmatrix}}{ }  \qquad \text{and} \qquad
	{ 
		\mathbf{X}^{(vs_3)} = \begin{pmatrix}
			{ x_{11}} & { x_{12}} & 1 & 0 & 0\\
		{ x_{21}} & { x_{22}} & 0 & { x_{24}} & 1\\
		1 & 0 & 0 & 0 & 0\\
		0 & 1 & 0 & 0 & 0\\
		0 & 0 & 0 & 1 & 0
		\end{pmatrix}.}{ }
	\]
	Using the same multigrading in Lemma \ref{lem:homo-multi}, while the variable $x_{22}$ in both matrices has the degree $e_{v^{-1}(2)} - e_2 = e_{(vs_3)^{-1}(2)} - e_2 = e_5 - e_2$, the variable $x_{12}$ has degree $e_{v^{-1}(1)} - e_2 = e_4 - e_2$ in $\mathbf{X}^{(v)}$ and degree $e_{(vs_3)^{-1}(1)} - e_2 = e_3 - e_2$ in $\mathbf{X}^{(vs_3)}$.
	So, in this example, degrees of some variables in both matrices are not preserved.
	\qed
\end{ex}

We do not encounter this problem of degree preservation when we change our conventions and work in the complete flag variety $G/B_{+}$, as in the paper \cite{woo2012grobner}.
In $G/B_{+}$, the Kazhdan-Lusztig variety $\mathcal{N}_{u,w}$ is isomorphic to the intersection of the Schubert variety $X_w = \overline{B_{+}wB_{+}/B_{+}}$ with the opposite Schubert cell $\Omega^\circ_{v} := B_{-}vB_{+}/B_{+}$.
Using the same coordinate system in \cite{woo2012grobner}, fix an arbitrary permutation $v \in S_n$ and define a specialized generic matrix $\mathbf{X}^{(v)}$ of size $n \times n$ as follows: 
for all $j$, set
\(\mathbf{X}^{(v)}_{n-v(j)+1,j} = 1,\)
and, for all $j$, set
$\mathbf{X}^{(v)}_{n-v(j)+1,d} = 0$ for 
$d > j$ and 
\(\mathbf{X}^{(v)}_{c,j} = 0\) 
for
$c > n-v(j)$+1.
For all other coordinates $(i,j)$, set
\(\mathbf{X}^{(v)}_{i,j} = x_{ij}.\)
\begin{ex}\label{ex:AlexPatch}
Let $v = 32154$. The last ascent of $v$ is $b=3$, so that $vs_3 = 32514$.	Then
		\[
	\stackrel{ 
		\begin{pmatrix}
			0 & 0 & 1 & 0 & 0\\
			0 & 1 & 0 & 0 & 0\\
			1 & 0 & 0 & 0 & 0\\
			{ x_{21}} & { x_{22}} & { x_{23}} & 0 & 1\\
			{ x_{11}} & { x_{12}} & { x_{13}} & 1 & 0
		\end{pmatrix},}{\mathbf{X}^{(v)}}  \quad
	\stackrel{ 
		\begin{pmatrix}
			0 & 0 & 0 & 1 & 0\\
			0 & 1 & 0 & 0 & 0\\
			1 & 0 & 0 & 0 & 0\\
			{ x_{21}} & { x_{22}} & 0 & x_{24} & 1\\
			{ x_{11}} & { x_{12}} & 1 & 0 & 0
		\end{pmatrix},}{\mathbf{X}^{(vs_3)}}
	\quad 
	\stackrel{
		\begin{pmatrix}
			0 & 0 & 0 & 1 & 0\\
		0 & 1 & 0 & 0 & 0\\
		1 & 0 & 0 & 0 & 0\\
		{ x_{21}} & { x_{22}} & 0 & { x_{23}} & 1\\
		{ x_{11}} & { x_{12}} & 1 & { x_{13}} & 0
		\end{pmatrix}.}{\mathbf{X}^{(v)}s_3} 
	\vspace{-.5cm}
	\]
	\qed
\end{ex}
Note from the above example that the matrix $\mathbf{X}^{(vs_3)}$ is equal to matrix $\mathbf{X}^{(v)}s_3$, up to relabeling the variable $x_{24}$ to $x_{23}$ ($x_{j,b+1} \mapsto x_{j,b}$) and setting $x_{\text{last}} = x_{13} \in \mathbf{x}^{(v)}$ to zero.
Also observe that the matrices $\mathbf{X}^{(v)}$ and $\mathbf{X}^{(vs_3)}$ in the previous example are upside-down versions of matrices $\mathbf{X}^{(v)}$ and $\mathbf{X}^{(vs_3)}$ in Example \ref{ex:our-convention-not-uniform-multidegree}.

Example \ref{ex:our-convention-not-uniform-multidegree} suggests we change our conventions. The following are established in \cite{woo2012grobner} in relation to degree preservation of the variables in $\mathbf{X}^{(v)}$ and $\mathbf{X}^{(vs_b)}$, where $b$ is the last ascent of $v$. 
Given a permutation $v \in S_n$, the variables $x_{ij}$ in the matrix $\mathbf{X}^{(v)}$ has degree $e_{v(j)} - e_{n-i+1}$, where $e_i$ is the $i^\text{th}$ standard basis vector in $\mathbb{Z}^n$.
Setting $R_v:=\mathbb{K}[\mathbf{x}^{(v)}]$ and $R_{vs_b}:=\mathbb{K}[\mathbf{x}^{(vs_b)}]$, we have: (i) in both $R_v$ and $R_{vs_b}$, the variables in $\mathbf{X}^{(v)}$ which are not in column $b$ or $b+1$ have the same degrees as the corresponding variables in $\mathbf{X}^{(vs_b)}$ which are not in column $b$ or $b+1$, (ii) the variables $x_{j,b}$ (except $j = v(b+1)$) in  $R_{v}$ have the same degrees as the variable  $x_{j,b+1}$ in  $R_{vs_b}$, and (iii) the rightmost and lowermost variable of $\mathbf{X}^{(v)}$, denoted $x_{\text{last}}$, does not appear in $\mathbf{X}^{(vs_b)}$, and it has degree $e_{v(b)} - e_{v(b+1)}$.

\begin{remark}\label{rem:changing-var}
	Let $v \in S_n$ and $b$ be the last ascent of $v$.
	Set $A' := \langle f' \,|\, f \in A\rangle$ (as in Lemma \ref{lem:moving-v-vsi}), $I:=I_{v,w}$, $J:=I_{vs_b,w}$ and $N:=I_{vs_b,ws_b}$. 
	The ideals $I$, $J'$ and $N'$ satisfy the hypothesis of Corollary \ref{cor:kazdan-to-schubert2} and consequently, there exists  $\boldsymbol{e} \in \mathbb{Z}^n$ such that there is an $R_v/N'$-module isomorphism $I/N' \cong (J'/N')(-\boldsymbol{e})$.
\end{remark}

Below is a Kostant-Kumar $K$-polynomial recursion for the Kazhdan-Lusztig ideals \cite[Theorem 6.12]{woo2012grobner} for last ascents, originally found in \cite[Proposition 2.4]{kostant1990t}.
We note that our proof is not dependent on the characteristic of the underlying field.

\begin{prop}\label{thm:kostant-recursionGeneral}
	Let $v$ and $w$ be permutations in $S_n$ and let $b$ be the last ascent of $v$, so that $vs_b > v$ in Bruhat order.
	Set $I := I_{v,w}$, $J := I_{vs_b,w}$ and $N := I_{vs_b,ws_b}$.
	Then
	\begin{enumerate}
		\item If $b$ is a descent of $w$, so that $ws_b < w$, then
		\[\mathcal{K}(R_v/I;\boldsymbol{t}) = \mathcal{K}(R_{vs_b}/J;\boldsymbol{t}).\]
		\item If $b$ is an ascent of $w$, so that $ws_b > w$, then
		\begin{eqnarray*}
			\mathcal{K}(R_{v}/I;\boldsymbol{t}) & = & \mathcal{K}(R_{vs_b}/J;\boldsymbol{t}) + 
			(1-\boldsymbol{t}^{\boldsymbol{e}})\mathcal{K}(R_{vs_b}/N;\boldsymbol{t})-\,(1-\boldsymbol{t}^{\boldsymbol{e}})\mathcal{K}(R_{vs_b}/J;\boldsymbol{t}),
		\end{eqnarray*}			
		where $\boldsymbol{e}$ is the degree (weight) of the variable $x_{\text{last}}$ in $\mathbf{X}^{(v)}$.
	\end{enumerate}
	\begin{proof}
		The first part follows immediately from a version of Corollary \ref{cor:descent-ascent-equal-ideal} in $G/B_{+}$ setting.
		For the second part, we proceed as follows: first, since $0 \rightarrow I/N' \rightarrow R_v/N' \rightarrow R_v/I \rightarrow 0$, where $N' = \langle f' \,|\, f \in N \rangle$, we have
		\[\text{Hilb}_{R_v/I}(\boldsymbol{t}) = \sum_{\boldsymbol{\ell} \in \mathbb{Z}^{^n}} \dim_{\mathbb{K}} \big( (R_v/I )_{\boldsymbol{\ell}}\big) \boldsymbol{t}^{\boldsymbol{\ell}} = \sum_{\boldsymbol{\ell} \in \mathbb{Z}^{^n}}\left[  \dim_{\mathbb{K}} \big( (R_v/N' )_{\boldsymbol{\ell}}\big)  - \dim_{\mathbb{K}} \big( (I/N' )_{\boldsymbol{\ell}}\big) \right]\boldsymbol{t}^{\boldsymbol{\ell}}.\]
		Next, since $(I/N')_{\boldsymbol{\ell}} \cong (J'/N')_{\boldsymbol{\ell} - \boldsymbol{e}}$ by Remark \ref{rem:changing-var}, we have
		\[\text{Hilb}_{R_v/I}(\boldsymbol{t}) = \sum_{\boldsymbol{\ell} \in \mathbb{Z}^{^n}}\left[  \dim_{\mathbb{K}} \big( (R_v/N' )_{\boldsymbol{\ell}}\big)  - \dim_{\mathbb{K}} \big( (J'/N' )_{\boldsymbol{\ell-e}}\big) \right]\boldsymbol{t}^{\boldsymbol{\ell}}.\]
		Furthermore, since $0 \rightarrow J'/N' \rightarrow R_v/N' \rightarrow R_v/J' \rightarrow 0$, we have
		\begin{eqnarray*}
			\text{Hilb}_{R_v/I}(\boldsymbol{t}) & = & \sum_{\boldsymbol{\ell} \in \mathbb{Z}^{^n}}  \dim_{\mathbb{K}} \big( (R_v/N' )_{\boldsymbol{\ell}}\big)\boldsymbol{t}^{\boldsymbol{\ell}}  - \sum_{\boldsymbol{\ell} \in \mathbb{Z}^{^n}}\left[\dim_{\mathbb{K}} \big( (R_v/N' )_{\boldsymbol{\ell-e}}\big) - \dim_{\mathbb{K}} \big( (R_v/J' )_{\boldsymbol{\ell-e}}\big) \right]\boldsymbol{t}^{\boldsymbol{\ell}}\\
			& = & \sum_{\boldsymbol{\ell} \in \mathbb{Z}^{^n}}  \dim_{\mathbb{K}} \big( (R_v/N' )_{\boldsymbol{\ell}}\big)\boldsymbol{t}^{\boldsymbol{\ell}}  - \sum_{\boldsymbol{\ell} \in \mathbb{Z}^{^n}}\left[\dim_{\mathbb{K}} \big( (R_v/N' )_{\boldsymbol{\ell-e}}\big) - \dim_{\mathbb{K}} \big( (R_v/J' )_{\boldsymbol{\ell-e}}\big) \right]\boldsymbol{t}^{\boldsymbol{\ell-e}}\boldsymbol{t}^{\boldsymbol{e}}\\
			& = & \sum_{\boldsymbol{\ell} \in \mathbb{Z}^{^n}}  \dim_{\mathbb{K}} \big( (R_{vs_b}/N )_{\boldsymbol{\ell}}\big)\boldsymbol{t}^{\boldsymbol{\ell}}  - \sum_{\boldsymbol{\ell} \in \mathbb{Z}^{^n}}\left[\dim_{\mathbb{K}} \big( (R_{vs_b}/N )_{\boldsymbol{\ell-e}}\big)  \right.\\
			&  &\qquad \qquad \qquad \qquad \qquad \qquad \qquad \qquad \left.-\,\, \dim_{\mathbb{K}} \big( (R_{vs_b}/J )_{\boldsymbol{\ell-e}}\big) \right]\boldsymbol{t}^{\boldsymbol{\ell-e}}\boldsymbol{t}^{\boldsymbol{e}} \\
			& = & \text{Hilb}_{R_{vs_b}/N}(\boldsymbol{t})  - \boldsymbol{t}^{\boldsymbol{e}} \big[\text{Hilb}_{R_{vs_b}/N}(\boldsymbol{t}) - \text{Hilb}_{R_{vs_b}/J}(\boldsymbol{t})\big],
		\end{eqnarray*}
		i.e.,
		\[ \frac{\mathcal{K}(R_v/I;\boldsymbol{t})}{\prod_{i,j}(1-\boldsymbol{t}^{\boldsymbol{a}_{ij}})} 
		=
		\frac{\mathcal{K}(R_{vs_b}/N;\boldsymbol{t})}{\prod_{i,j}(1-\boldsymbol{t}^{\boldsymbol{a}_{ij}})}  - 
		\boldsymbol{t}^{\boldsymbol{e}} \left[
		\frac{\mathcal{K}(R_{vs_b}/N;\boldsymbol{t})}{\prod_{i,j}(1-\boldsymbol{t}^{\boldsymbol{a}_{ij}})}
		- \frac{\mathcal{K}(R_{vs_b}/J;\boldsymbol{t})}{\prod_{i,j}(1-\boldsymbol{t}^{\boldsymbol{a}_{i,j}})}
		\right],\]
		where each $\boldsymbol{a}_{ij} \in \mathbb{Z}^{n}$ is the degree of variable $x_{ij}$ in $\mathbf{X}^{(v)}$.
		Thus,
		\begin{eqnarray*}
			\mathcal{K}(R_v/I;\boldsymbol{t})
			& = &
			\mathcal{K}(R_{vs_b}/N;\boldsymbol{t}) - 
			\boldsymbol{t}^{\boldsymbol{e}} \left[
			\mathcal{K}(R_{vs_b}/N;\boldsymbol{t})
			- \mathcal{K}(R_{vs_b}/J;\boldsymbol{t})
			\right]\\
			& = & \mathcal{K}(R_{vs_b}/J;\boldsymbol{t}) + 
			(1-\boldsymbol{t}^{\boldsymbol{e}})\mathcal{K}(R_{vs_b}/N;\boldsymbol{t})-\,(1-\boldsymbol{t}^{\boldsymbol{e}})\mathcal{K}(R_{vs_b}/J;\boldsymbol{t}). \qquad \qquad \,\,\,\qed 
		\end{eqnarray*}
		\renewcommand\qedsymbol{}
	\end{proof}
\end{prop}
	
	\section{G-Biliaison of Homogeneous Schubert patch ideals}\label{sec:gbiliaison-s-d-i}
	Not all Schubert patch ideals are homogeneous with respect to the standard grading in $R:=\mathbb{K}[\mathbf{z}^{(v)}]$.
	Here, we say a Schubert patch ideal $Q_{v,w}$ is \textbf{standardly homogeneous} if it is homogeneous with respect to the standard grading in $R$.
	Suppose $Q_{v,w}$ is standardly homogeneous and let $b$ be the last descent of $v$.
	If $b$ is an ascent of $w$, then, by Lemma \ref{lem:descent-ascent-equal-ideal}, $T_{vs_b,w}$
	is also standardly homogeneous.
	On the other hand, if $b$ is a descent of $w$, then it follows from Lemma \ref{lem:IvwToIvsbw} and Lemma \ref{lem:IvwToIvsbwsb} that both ideals $T_{vs_b,w}$ and $T_{vs_b,ws_b}$
	are also standardly homogeneous.
	Furthermore, with respect to this standard grading, the isomorphisms in Lemma \ref{lem:kazdan-to-schubert2} are standard graded isomorphisms specifically for the ideals involved, where $n=1$ and $\boldsymbol{e}=1$.
	In summary, if $I:=Q_{v,w}$ is standardly homogeneous, then $I/N \cong (J/N)(-1)$ is an $R/N$ module isomorphism, where $J:=T_{vs_b,w}$ and $N:=T_{vs_b,ws_b}$.
	
	In this section, we will show that every standardly homogeneous Schubert patch ideal is glicci.
	With a minor modification to the proof of this glicci result, we can show that every standardly homogeneous Kazhdan-Lusztig ideal, and hence every Schubert determinantal ideal, is glicci.
	
	Here forward, we will write $I \stackrel{N}{\rightharpoondown}_h J$ to mean that the ideal $I$ can be obtained from another ideal $J$ by an elementary G-biliaison of height $h$ on $N$. 
	So if a standardly homogeneous, saturated and unmixed ideal $I \subset R$ is obtained from a complete intersection generated by linear forms by a finite sequence of ascending elementary G-biliaisons of height $h$, then there exist Cohen-Macaulay and generically Gorenstein ideals $N_1,\ldots,N_t$ in $ R$ such that
	$	I = I_0 \stackrel{N_1}{\rightharpoondown}_h I_1 \stackrel{N_2}{\rightharpoondown}_h I_2 \stackrel{N_3}{\rightharpoondown}_h \cdots \stackrel{N_t}{\rightharpoondown}_h I_t,$
	where $I_{t} \subset R$ is a complete intersection generated by linear forms.
	
	\begin{lem}\label{lem:kazdan-to-schubert}
		Let $v \in S_n$ and $b$ be the last descent of $v$.
		If $T_{vs_b,w}$ can be obtained from an ideal generated by linear forms by a finite sequence of ascending elementary G-biliaisons, then $Q_{v,w}$ can also be obtained from an ideal generated by linear forms by a finite sequence of ascending elementary G-biliaisons.
\begin{proof}
	We prove this lemma in two different cases, namely when $b$ is an ascent of $w$ and when $b$ is a descent of $w$.
	\begin{description}
		\item[Case 1] Suppose $b$ is an ascent of $w$. Then by Lemma \ref{lem:descent-ascent-equal-ideal}, the ideals $Q_{v,w}$ and $T_{vs_b,w}$ are equal. 
		So if
		\(\mathcal{F}: \,\, T_{vs_b,w} = I_0 \stackrel{N_1}{\rightharpoondown}_h I_1 \stackrel{N_2}{\rightharpoondown}_h I_2 \stackrel{N_3}{\rightharpoondown}_h \cdots \stackrel{N_t}{\rightharpoondown}_h I_t,\)
		is a finite sequence of ascending elementary G-biliaisons of height $h=1$ for $T_{vs_b,w}$ up to an ideal generated by linear forms,
		then the same sequence $\mathcal{F}$ is a finite sequence of ascending elementary G-biliaisons of height 1 for $Q_{v,w}$ up to an ideal generated by linear forms.
	
		\item[Case 2]\label{case2:lem} Suppose $b$ is a descent of $w$ and 
		\[\mathcal{F}: \qquad T_{vs_b,w} = I_0 \stackrel{N_1}{\rightharpoondown}_h I_1 \stackrel{N_2}{\rightharpoondown}_h I_2 \stackrel{N_3}{\rightharpoondown}_h \cdots \stackrel{N_t}{\rightharpoondown}_h I_t,
		\]
		is a finite sequence of ascending elementary G-biliaisons of height $h=1$ for $T_{vs_b,w}$ up to an ideal generated by linear forms.
		Then we claim that the sequence
		\[\mathcal{G}: \qquad Q_{v,w} \stackrel{T_{vs_b,ws_b}}{\rightharpoondown_h} T_{vs_b,w} \stackrel{N_1}{\rightharpoondown}_h I_1 \stackrel{N_2}{\rightharpoondown}_h I_2 \stackrel{N_3}{\rightharpoondown}_h \cdots \stackrel{N_t}{\rightharpoondown}_h I_t,\]
		is a finite sequence of ascending elementary G-biliaisons of height $h=1$ for $Q_{v,w}$ up to an ideal generated by linear forms. 
		In order to establish this, it suffices to only show that $Q_{v,w}$ can be obtained from $T_{vs_b,w}$ by an ascending elementary G-biliaison of height $1$ on $T_{vs_b,ws_b}$, i.e., {\small $Q_{v,w} \stackrel{T_{vs_b,ws_b}}{\rightharpoondown_1} T_{vs_b,w}$}.
		To this end, we will verify the statements in Definition \ref{def:elementary-G-biliaison}.
		{First, we can infer from Lemma \ref{lem:moving-from-T-Q} that $\text{ht}\, T_{vs_b,w} = \text{ht}\, Q_{vs_b,w} $ and
			$\text{ht}\, T_{vs_b,ws_b} = \text{ht}\, Q_{vs_b,ws_b}$}, i.e., $\text{ht}\, T_{vs_b,w} = \ell(w) = \text{ht}\, Q_{v,w}$ and
		$\text{ht}\, T_{vs_b,ws_b} = \ell(ws_b) = \ell(w) - 1$.
%
		From Corollary \ref{cor:N-contained-I-int-J}, we have that $T_{vs_b,ws_b} \subseteq Q_{v,w} \cap T_{vs_b,w}$.
		{In addition, since the prime ideal $Q_{vs_b,ws_b}$ is Cohen-Macaulay and generically Gorenstein,} we can also infer from Lemma \ref{lem:moving-from-T-Q} that $T_{vs_b,ws_b}$ is also Cohen-Macaulay and generically Gorenstein.
		Lastly, it follows from the arguments at the very beginning of this section that $Q_{v,w}/T_{vs_b,ws_b} \cong (T_{vs_b,w}/T_{vs_b,ws_b})(-1)$ as $\big(\mathbb{K}[\mathbf{z}^{(v)}]/T_{vs_b,ws_b}$\big)-module.
		Hence, there is an elementary G-biliaison. 
		\qed
	\end{description}
\renewcommand\qedsymbol{}
\end{proof}
\end{lem}

\vspace{-0.5cm}
	We will now consider the linkage process of the Schubert patch ideals that are homogeneous with respect to the standard grading.
	
	\begin{thm}\label{thm:vex-glicci}
		Every standardly homogeneous Schubert patch ideal $Q_{v,w}$, $v,w \in S_n$, 
		belongs to the G-liaison class of a complete intersection. That is, 
	$Q_{v,w}$ is glicci.
		\begin{proof}
			It suffices to show that any standardly homogeneous Schubert patch ideal $Q_{v,w}$ can be obtained from an ideal generated by linear forms by a finite sequence of ascending elementary G-biliaisons, and so by Theorem \ref{connector}, it will follow that $Q_{v,w}$ is glicci. 
			We proceed by induction on $\ell(v)$.
			If $\ell(v) = 1$ and $Q_{v,w} \neq \langle 1 \rangle$, then by Lemma \ref{lem:base-case}, $Q_{v,w}$ is generated by an indeterminate and hence, the result follows.
			Suppose the hypothesis is true for all standardly homogeneous Schubert patch ideals $Q_{v',w}$, $v',w \in S_n$, with $v' \leq v$ in Bruhat order.
			Let $b$ be the last descent of $v$.
			Then $vs_b \leq v$ in Bruhat order and $\ell(vs_b) = \ell(v) -1$.
			Suppose the variable $z_{\text{max}}$ belongs to $Q_{v,w}$.
			Then $Q_{vs_b,{w}}$ is the unit ideal and $Q_{v,w} = Q_{vs_b,{w}s_b} + \langle z_{\text{max}} \rangle$.
			By the induction hypothesis, $Q_{vs_b,{w}s_b}$ can be obtained from an ideal generated by linear forms by a finite sequence of ascending elementary G-biliaisons.
			Consequently, $Q_{v,{w}}$ can also be obtained from an ideal generated by linear forms by a finite sequence of ascending elementary G-biliaisons since $Q_{v,w}=\langle h_1,\ldots,h_{\ell},z_{\text{max}}\rangle$, where the $h_i$ are generators of $Q_{vs_b,ws_b}$.
			Therefore, by Theorem \ref{connector}, the ideal $Q_{v,w}$ is in the G-liaison class of a complete intersection.
			On the other hand, suppose the variable $z_{\text{max}}$ does not belong to $Q_{v,w}$.
			Then $Q_{vs_b,w}$ is a proper ideal and so by the induction hypothesis, it follows that $Q_{vs_b,{w}}$ can be obtained from an ideal generated by linear forms by a finite sequence of ascending elementary G-biliaisons.
			Consequently, we can infer from Remark \ref{rem:trans-Q-T} and Lemma \ref{lem:moving-from-T-Q} 			
			that $T_{vs_b,{w}}$ can be obtained from an ideal generated by linear forms by a finite sequence of ascending elementary G-biliaisons.
			Therefore, by Lemma \ref{lem:kazdan-to-schubert}, the ideal $Q_{v,{w}}$ can be obtained from an ideal generated by linear forms by a finite sequence of ascending elementary G-biliaisons.
			Thus, by Theorem \ref{connector}, $Q_{v,w}$ is in the G-liaison class of a complete intersection.
		\end{proof}
	\end{thm}

The proof for Theorem \ref{thm:vex-glicci} can be easily adapted to show that standardly homogeneous Kazhdan-Lusztig ideals are glicci.
Consequently, the Schubert determinantal ideals are also glicci, since they are a special cases of the standardly homogeneous Kazhdan-Lusztig ideals. 
In particular, the final ideal generated by the linear forms in the linkage process of $I_{w}$ is explicitly given as
$ \big\langle x_{i,(j-\text{rank}(w_{i \times j}))}\,|\, \text{$(i,j)$ is a box in $D(w)$} \big\rangle,$
which is a complete intersection.
This is true since the linkage process of the ideal $I_w$ basically involves ``movement" of essential boxes (or boxes in general) one step at a time to the left until they are no longer ``movable". 
The set of corresponding variables at each of the final destinations of these boxes in $D(w)$ gives the generator for final ideal in the linkage process of $I_{w}$.
For example, for the permutation $w = 2143 \in S_4$, the ideal $\left \langle x_{11},  x_{31}\right \rangle$ is the last ideal in the linkage process of $I_w$.
Another example is $w = 136524 \in S_6$; the final ideal in the linkage process of $I_w$ is $\left \langle x_{21}, x_{22}, x_{23}, x_{41}, x_{42}, x_{51}\right \rangle$. See below for illustrations of these movements.
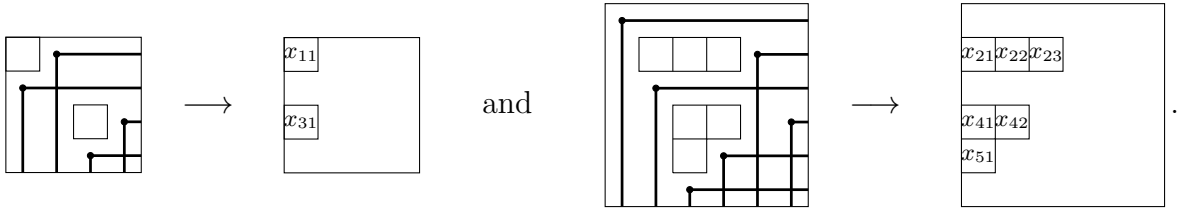
\[ \vcenter{\hbox{
		\begin{tikzpicture}[scale=.45]
			\draw (0,0) rectangle (4,4);
			
			\draw (0,3) rectangle (1,4);
			\draw (2,1) rectangle (3,2);

			\filldraw (0.5,2.5) circle (.5ex); \draw[line width = .2ex] (0.5,0) --(0.5,2.5) --(4,2.5);
			\filldraw (1.5,3.5) circle (.5ex); \draw[line width = .2ex] (1.5,0) --(1.5,3.5) --(4,3.5);
			\filldraw (2.5,0.5) circle (.5ex); \draw[line width = .2ex] (2.5,0) --(2.5,0.5) --(4,0.5);
			\filldraw (3.5,1.5) circle (.5ex); \draw[line width = .2ex] (3.5,0) --(3.5,1.5) --(4,1.5);
\end{tikzpicture}}} \quad \longrightarrow \quad \vcenter{\hbox{\begin{tikzpicture}[scale=.45]
\draw (0,0) rectangle (4,4);

\draw (0,3) rectangle (1,4);
\draw (0,1) rectangle (1,2);

\node at (0.5,3.5) {\scriptsize $x_{11}$};
\node at (0.5,1.5) {\scriptsize $x_{31}$};

\end{tikzpicture}}} \qquad \text{and} \qquad  
\vcenter{\hbox{
		\begin{tikzpicture}[scale=.45]
			\draw (0,0) rectangle (6,6);

			\draw (1,4) rectangle (4,5);
			\draw (2,4) -- (2,5); \draw (3,4) -- (3,5);
			
			\draw (2,1) -- (2,3) -- (4,3) -- (4,2) -- (3,2) -- (3,1) -- (2,1);
			\draw (2,2) -- (3,2) -- (3,3);

			\filldraw (0.5,5.5) circle (.5ex); \draw[line width = .2ex] (0.5,0) -- (0.5,5.5) -- (6,5.5);
			\filldraw (1.5,3.5) circle (.5ex); \draw[line width = .2ex] (1.5,0) --(1.5,3.5) --(6,3.5);
			\filldraw (2.5,0.5) circle (.5ex); \draw[line width = .2ex] (2.5,0) --(2.5,0.5) --(6,0.5);
			\filldraw (3.5,1.5) circle (.5ex); \draw[line width = .2ex] (3.5,0) --(3.5,1.5) --(6,1.5);
			\filldraw (4.5,4.5) circle (.5ex); \draw[line width = .2ex] (4.5,0) --(4.5,4.5) --(6,4.5);
			\filldraw (5.5,2.5) circle (.5ex); \draw[line width = .2ex] (5.5,0) --(5.5,2.5) --(6,2.5);
\end{tikzpicture}}} 
\quad \longrightarrow \quad   \vcenter{\hbox{
		\begin{tikzpicture}[scale=.45]
			\draw (0,0) rectangle (6,6);

			\draw (0,4) rectangle (3,5);
			\draw (1,4) -- (1,5); \draw (2,4) -- (2,5);
			\node at (0.5,4.5) {\scriptsize $x_{21}$};
			\node at (1.5,4.5) {\scriptsize $x_{22}$};
			\node at (2.5,4.5) {\scriptsize $x_{23}$};
			
			\draw (0,1) -- (0,3) -- (2,3) -- (2,2) -- (1,2) -- (1,1) -- (0,1);
			\draw (0,2) -- (1,2) -- (1,3);
			\node at (0.5,2.5) {\scriptsize $x_{41}$};
			\node at (1.5,2.5) {\scriptsize $x_{42}$};
			\node at (0.5,1.5) {\scriptsize $x_{51}$};

\end{tikzpicture}}}
\,.\]
Though we can infer from the paper \cite{klein2020geometric} that Schubert determinantal ideals are glicci, their approach relies on combinatorial results of Knutson-Miller \cite{knutson2005grobner} to deduce this.

\begin{cor}\label{kah-sch-glicci}
	Every standardly homogeneous Kazhdan-Lusztig ideal is glicci. 
	In particular, the Schubert determinantal ideals are glicci.
\end{cor}
	
	\section{On Schubert patch ideals and Homogeneity with respect to the Standard Grading}\label{sec:hom-KL}
	Homogeneity with respect to the standard grading is needed for our results in Section \ref{sec:gbiliaison-s-d-i}.
	Not all Schubert patch ideals are standardly homogeneous as illustrated in Example \ref{ex:inhomogeneous-patch}.
	So it is natural to ask for which pairs $(v,w) \in S_n \times S_n$ is the ideal $Q_{v,w}$ standardly homogeneous (homogeneous with respect to the standard grading in $\mathbb{K}[\mathbf{z}]$).
	The analogue of this problem for Kazhdan-Lusztig ideals is \cite[Problem 5.5]{woo2008governing}.
	
	\begin{ex}\label{ex:inhomogeneous-patch}
		Let $v = 4231$ and $w = 2143$.
		Then
		\[\mathbf{Z}^{(v)} = 
		\begin{pmatrix}
				z_{11} & z_{12} & z_{13} & 1\\
			z_{21} & 1 & z_{23} & 0\\
			z_{31} & 0 & 1 & 0\\
			1 & 0 & 0 & 0
		\end{pmatrix}, \quad  \quad D(w) =  \vcenter{\hbox{
							\begin{tikzpicture}[scale=.45]
								\draw (0,0) rectangle (4,4);
								
								\draw (0,3) rectangle (1,4);
								
								\draw (2,1) rectangle (3,2);

								\filldraw (0.5,2.5) circle (.5ex); \draw[line width = .2ex] (0.5,0) --(0.5,2.5) --(4,2.5);
								\filldraw (1.5,3.5) circle (.5ex); \draw[line width = .2ex] (1.5,0) --(1.5,3.5) --(4,3.5);
								\filldraw (2.5,0.5) circle (.5ex); \draw[line width = .2ex] (2.5,0) --(2.5,0.5) --(4,0.5);
								\filldraw (3.5,1.5) circle (.5ex); \draw[line width = .2ex] (3.5,0) --(3.5,1.5) --(4,1.5);
					\end{tikzpicture}}}\]
		and
		\(Q_{v,w} = \left\langle z_{11}, \, z_{13} z_{31} + z_{12} z_{21} - z_{12} z_{31} z_{23}  \right\rangle,\)
		which is inhomogeneous.
		\hfill \qedsymbol
	\end{ex}

\begin{prop}\label{prop:homo-patch-homo-kazh}
	Let $v , w \in S_n$ be fixed.
	If $Q_{v,w}$ is homogeneous with respect to the standard grading on $\mathbb{K}[\mathbf{z}^{(v)}]$, then ${I}_{v,w}$ is homogeneous with respect to the standard grading on $\mathbb{K}[\mathbf{x}^{(v)}]$.
	\begin{proof}
		Setting some variables to zero in $Q_{v,w}$ will not affect homogeneity of the resulting ideal.
	\end{proof}
\end{prop}
The converse of Proposition \ref{prop:homo-patch-homo-kazh} is not always true. To see this, let $v = 4231$ and $w = 2143$ (as in Example \ref{ex:inhomogeneous-patch}).
While $Q_{v,w}$ is not homogeneous, the ideal \({I}_{v,w} = \left\langle z_{11}, \,z_{13} z_{31} + z_{12} z_{21}  \right\rangle\) is homogeneous.

In this section, we wish to classify some Schubert patch ideals and Kazhdan-Lusztig ideals that are standardly homogeneous.
A permutation $w \in S_n$ is called \textbf{\textit{321}-avoiding} if there do not exist three integers $i_1 < i_2 < i_3$ with $w(i_3) < w(i_2) < w(i_1)$.
A permutation $w \in S_n$ is called \textbf{\textit{231}-avoiding} if there do not exist three integers $i_1 < i_2 < i_3$ with $w(i_3) < w(i_1) < w(i_2)$.
A permutation $w \in S_n$ is called \textbf{\textit{132}-avoiding} if there do not exist three integers $i_1 < i_2 < i_3$ with $w(i_1) < w(i_3) < w(i_2)$.
	
	If a permutation is non \textit{132}-avoiding, then we say it contains \textit{132} pattern. Similarly, if a permutation is non \textit{321}-avoiding (resp. \textit{231}-avoiding), then we say it contains \textit{321} (resp. \textit{231}) pattern.
	For instance, while the non \textit{132}-avoiding permutation $1243 \in S_4$ contains only one \textit{132} pattern, which is \textit{243}, the non \textit{321}-avoiding permutation $52143 \in S_5$ contains only two \textit{321} patterns, which are \textit{521} and \textit{543}.
	
	\renewcommand{\theenumi}{\alph{enumi}} 
	\renewcommand{\theenumii}{\roman{enumii}}
	
		\begin{prop}\label{conj:homo-kaz}
		Let $(v,w) \in S_n \times S_n$ for which $w \leq v$ in Bruhat order. 
		\begin{enumerate}
			\item If $w$ is \textit{132}-avoiding, then $Q_{v,w}$ is standardly homogeneous.
			\item If $v$ is both \textit{321}-avoiding and \textit{231}-avoiding, then $Q_{v,w}$ is standardly homogeneous.
			\item\label{it:parrtc} Suppose $v$ is either non \textit{321}-avoiding or non \textit{231}-avoiding, and $w$ is non \textit{132}-avoiding.
			Then $Q_{v,w}$ is standardly homogeneous if
			\begin{itemize}
				\item for all pairs of patterns $(a_1\,a_2\,a_3,c_1\,c_2\,c_3)$, either 
				\begin{equation}\label{ineq101}
					c_3 < a_2 \qquad \text{or} \qquad  w^{-1}(c_2) < v^{-1}(a_2),
				\end{equation}
			and
				\item for all pairs of patterns $(b_1\,b_2\,b_3,c_1\,c_2\,c_3)$, either 
				\begin{equation}\label{ineq201}
					c_3 < b_1 \qquad \text{or} \qquad  w^{-1}(c_2) \leq v^{-1}(b_1),
				\end{equation}
			\end{itemize}
			where $a_1\,a_2\,a_3$ and $b_1\,b_2\,b_3$ are \textit{321} and \textit{231} patterns in $v$, respectively, and $c_1\,c_2\,c_3$ is a \textit{132} pattern in $w$.
		\end{enumerate}
		\begin{proof}
			\phantom{}\\
			\vspace{-0.5cm}
			\begin{enumerate}
				\item \label{it:conj:homo-kazA}
				If $w$ is $132$-avoiding, then from \cite[Section 1.5]{stanley2011enumerative}, the diagram $D(w)$ forms a partition, and hence, the ideal $Q_{v,w}$, for any $v$, is generated by indeterminate(s).
				Therefore, $Q_{v,w}$ is homogeneous with respect to the standard grading.			
				
				\item \label{it:conj:homo-kazB}
				Suppose $Q_{v,w}$ is not standardly homogeneous.
				Then we will show that $v$ is either non \textit{321}-avoiding or non \textit{231}-avoiding.
				In other words, we will show that $v$ either contains one \textit{321} or one \textit{231} pattern.				
				Since $Q_{v,w}$ is not standardly homogeneous, it follows that that there is at least one essential minor $f$ in $Q_{v,w}$ that is not standardly homogeneous. 
				Suppose the corresponding essential box of $f$ is $(p,q)$.
				Set $M:= \mathbf{Z}_{p \times q}^{(v)}$,  where $\mathbf{Z}^{(v)}_{p \times q}$ is the upper left $p \times q$  submatrix of $\mathbf{Z}^{(v)}$.
				We note that $f$ is one of the minors of size $(1+\text{rank}(w_{p \times q}))$ in $M$.
				We will show that a matrix of the form 
				\begin{eqnarray}\label{eq:N-two-by-two}
					N = \begin{bmatrix}
					z_{ik} & z_{il}\\
					z_{jk} & 1
				\end{bmatrix} \qquad \text{or} \qquad 
				N' = \begin{bmatrix}
					z_{ik} & z_{il}\\
					1 & z_{jl}
				\end{bmatrix}
						\end{eqnarray}
				is a submatrix of $M$.
				We show this by induction on the degree of $f$.
				If the degree of $f$ is 2, then there exists a $2 \times 2$ submatrix $N''$ of $M$ whose determinant is $f$. 
				We first note that since $f$ is not standardly homogeneous, it follows that $\det(N'')$ is not equal to a constant (0 or $\pm 1$). 
				The matrix $N''$ must therefore take one of the following forms:
				\[\begin{bmatrix}
					z_{ik} & z_{il}\\
					z_{jk} & z_{jl}
				\end{bmatrix},\,\,
				\begin{bmatrix}
					1 & z_{il}\\
					0 & z_{jl}
				\end{bmatrix},\,\,
				\begin{bmatrix}
					z_{ik} & 1\\
					z_{jk} & 0
				\end{bmatrix},\,\,
				\begin{bmatrix}
					z_{ik} & z_{il}\\
					1 & z_{jl}
				\end{bmatrix},\,\,
				\begin{bmatrix}
					z_{ik} & z_{il}\\
					z_{jk} & 1
				\end{bmatrix},
								\]
				for which the determinant of the last two matrices are the only polynomials that are not standardly homogeneous. 
				Now, suppose $\deg(f) \geq 2$ and let $M'$ be the corresponding (square) submatrix of $M$ whose determinant is $f$.
				We note that $M'$ can take two forms: 
				\begin{description}
					\item[Case 1] $M'$ does not have a row or a column with exactly one 1 and the rest of the entries in either of this row or column are 0. 
					The polynomial $f$ is not standardly homogeneous in the first place due to at least one 1 at a strategic position in $M$; otherwise, all the terms of $f$ will be of same degree, and hence $f$ will be homogeneous.
					Precisely, $M'$ has at least one 1 among its entries. 
					Therefore, in this case, both the row of this 1 and the column of this 1 contain some variables.
					In fact, if the 1 is in position $(j',l')$, then there must be a variable in either position $(j',k')$, where $k'<l'$, or position $(j',k'')$,  where $l' < k''$, and another variable in position $(i,l')$, where $i < j'$.
					Hence, there is also a variable in either position $(i,k')$ or position $(i,k'')$, as desired (see below).
					\[\begin{bmatrix}
						z_{ik'} & z_{il'}\\
						z_{j'k'} & 1
					\end{bmatrix} \qquad \text{or} \qquad 
					\begin{bmatrix}
						z_{il'} & z_{ik''}\\
						1 & z_{j'k''}
					\end{bmatrix}
					.\vspace{-0.25cm}\]
					\item[Case 2] $M'$ has a row or a column with exactly one 1 and the rest of the entries in either of this row or column are 0.
					In this case, using cofactor expansions along that row or column, we see that up to sign, $\det(M') = \det(M'')$, where $M''$ is the resulting matrix from deleting a row and a column from $M'$.
					By induction, $M''$ has a submatrix of the form $N$ or $N'$, and therefore, $M$ also has either of these submatrices.
				\end{description}
				
				So, in either case,	a matrix of the form $N$ or $N'$
				is a submatrix of $M$.
				If $N$ is a submatrix of $M$, then the 1 in column $k$ of $\mathbf{Z}^{(v)}$ is in row $v(k)$ and the 1 in row $i$ of $\mathbf{Z}^{(v)}$ is in column $v^{-1}(i)$. Thus, there exists three integers $k < l < v^{-1}(i) $ such that $v(v^{-1}(i)) < v(l) < v(k)$, i.e., $i < j < v(k)$, i.e., we have explicitly obtained a \textit{321} pattern in $v$, say $v(k)\,\,v(l)\,\,v(v^{-1}(i))$ (or\, $v(k)\,\,j\,\,i$).
				Similarly, if $N'$ is a submatrix of $M$, then there exists three integers $k < l < v^{-1}(i) $ such that $v(v^{-1}(i)) < v(k) < v(l)$, i.e., $i < j < v(l)$, i.e., we have explicitly obtained a \textit{231} pattern in $v$, say $v(k)\,\,v(l)\,\,v(v^{-1}(i))$ (or\, $j\,\,v(l)\,\,i$).
				Hence, the result follows.

				\item\label{it:conj:homo-kazC}
				Suppose $Q_{v,w}$ is not standardly homogeneous.
				Then we will find a \textit{321} or a \textit{231} pattern in $v$, and a \textit{132} pattern in $w$ for which the inequalities \ref{ineq101} and \ref{ineq201} fail. 
				If $Q_{v,w}$ is not standardly homogeneous,
				then there is at least one (essential) minor $f$ in $Q_{v,w}$ that is not standardly homogeneous. 
				Suppose the corresponding essential box of this determinant $f$ is $(p,q)$.
				Then $\text{rank}(w_{p \times q}) > 0$.
				Consequently, it follows that there is at least one 1 strictly northwest of the location $(p,q)$ in $D(w)$.
				Suppose one of these 1s is in position $(\alpha,\beta)$. That is, $\alpha < p$ and $\beta < q$. We also note that in this case, $\alpha = w(\beta)$.
				In the diagram $D(w)$ of $w$, let the 1 to the right of box $(p,q)$ be in column $\gamma$. Then $q < \gamma$, since the box $(p,q)$ is in column $q$. Since $w(\gamma)$ is the row for which the box $(p,q)$ is located, it follows that $w(\gamma) = p$.
				So, we have three integers $\beta, q$ and $\gamma$ with the property that $1 \leq \beta < q < \gamma \leq n$.
				Since the 1 in column $q$ of diagram $D(w)$ is strictly below the box $(p,q)$, it follows that $p < w(q)$, where $w(q)$ is the row where this 1 is located. 
				Therefore, we have the inequality $\alpha < p < w(q)$, which can also be written as $w(\beta) < w(\gamma) < w(q)$. 
				Thus, we have explicitly obtained a \textit{132} pattern in $w$.
				
				Furthermore, consider the submatrix $M:=\mathbf{Z}^{(v)}_{p \times q}$. 
				From the proof of part (\ref{it:conj:homo-kazB}), we have that a matrix of the form $N$ or $N'$, as in (\ref{eq:N-two-by-two}),
				is a submatrix of $M$.
				The corresponding \textit{321} pattern for $N$ is  $v(k)\,\,v(l)\,\,v(v^{-1}(i))$ (or\, $v(k)\,\,j\,\,i$) and the corresponding \textit{231} pattern for $N'$ is 
				$v(k)\,\,v(l)\,\,v(v^{-1}(i))$ (or\, $j\,\,v(l)\,\,i$).
				If $N$ is a submatrix of $M$, 
				then $j\leq p$ and $v^{-1}(j) \leq w^{-1}(w(q))$ (or $l \leq q$), which contradicts inequality \ref{ineq101}. 
				Furthermore, if $N'$ is a submatrix of $M$, then $j\leq p$ and $v^{-1}(j) < w^{-1}(w(q))$ (or $k < q$), which contradicts inequality \ref{ineq201}.
				\qed
			\end{enumerate}
		\renewcommand\qedsymbol{}
		\end{proof}
	\end{prop}

\vspace{-0.5cm}
	
	The analogue of Proposition \ref{conj:homo-kaz} for Kazhdan-Lusztig ideals is the following result.
	
	\begin{prop}\label{conj:homo-kaz2}
		Let $(v,w) \in S_n \times S_n$ for which $w \leq v$ in Bruhat order.   
		\begin{enumerate}
			\item[(a)] If $v$ is \textit{321}-avoiding or $w$ is \textit{132}-avoiding, then ${I}_{v,w}$ is standardly homogeneous.
			\item[(b)]\label{it:conj:homo-kaz} Suppose $v$ is non \textit{321}-avoiding and $w$ is non \textit{132}-avoiding.
			Then ${I}_{v,w}$ is standardly homogeneous if for all pairs $(a_1\,a_2\,a_3,c_1\,c_2\,c_3)$, either \(c_3 < a_2\) or \(w^{-1}(c_2) < v^{-1}(a_2)\), 
			where $a_1\,a_2\,a_3$ is a \textit{321} pattern in $v$ and $c_1\,c_2\,c_3$ is a \textit{132} pattern in $w$.
		\end{enumerate}
	\begin{proof}
		\phantom{}\\
		\vspace{-0.5cm}
		\begin{enumerate}
			\item[(a)] If $v$ is $321$-avoiding, then $I_{v,w}$, for all $w$, is known to be homogeneous (see for example the footnote on page 25 of \cite{knutson2009frobenius}). 
			If $w$ is $132$-avoiding, then the ideal $I_{v,w}$, for any $v$, is generated by indeterminate(s), and hence, homogeneous with respect to the standard grading.
			\item[(b)] This is a special case of part (\ref{it:parrtc}) of Proposition \ref{conj:homo-kaz}. Recall from Proposition \ref{prop:homo-patch-homo-kazh} that $I_{v,w}$ is standardly homogeneous whenever $Q_{v,w}$ is standardly homogeneous.
			\qedhere
		\end{enumerate}
	\end{proof}
	\end{prop}

	\bibliographystyle{alpha}
	\bibliography{Schubert-Patch-Ideal}

\end{document}